\pdfoutput=1
\documentclass{article}

\usepackage{amsmath,amssymb,mathtools} 
\usepackage{algorithm}                  
\usepackage{algpseudocode}              
\usepackage{float}                     
\usepackage{enumitem}
\usepackage{multirow} 
\usepackage{PRIMEarxiv}
\usepackage[numbers]{natbib}

\usepackage{array}
\usepackage[utf8]{inputenc}
\usepackage[T1]{fontenc}
\usepackage{hyperref}     

\usepackage{url}
\usepackage{cleveref}
\usepackage{booktabs}
\usepackage{amsfonts}
\usepackage{nicefrac}
\usepackage{microtype}

\usepackage{fancyhdr}
\usepackage{graphicx}
\graphicspath{{media/}}
\usepackage{subcaption} 
\usepackage{epstopdf}

\usepackage{amsmath,amsfonts,bm}

\def\eqref#1{equation~\ref{#1}}
\def\Eqref#1{Equation~\ref{#1}}

\def\1{\bm{1}}

\DeclareMathAlphabet{\mathsfit}{\encodingdefault}{\sfdefault}{m}{sl}
\SetMathAlphabet{\mathsfit}{bold}{\encodingdefault}{\sfdefault}{bx}{n}

\newcommand{\R}{\mathbb{R}}

\usepackage{amsmath}
\usepackage{mathtools}
\usepackage{amsthm}
\usepackage{xcolor}

\newtheorem{theorem}{Theorem}[section]
\newtheorem{lemma}[theorem]{Lemma}
\newtheorem{assumption}[theorem]{Assumption}

\newtheorem{corollary}[theorem]{Corollary}

\newtheorem{definition}[theorem]{Definition}
\newtheorem{remark}[theorem]{Remark}

\newcommand{\norm}[1]{\left\lVert #1 \right\rVert}
\newcommand{\abs}[1]{\left\lvert #1 \right\rvert}
\newcommand{\inner}[2]{\left\langle #1, #2 \right\rangle}

\newcommand{\DLa}{D\widetilde{\Lambda}}

\newcommand{\Per}{\mathrm{Per}}
\newcommand{\diam}{\mathrm{diam}}
\newcommand{\dist}{\operatorname{dist}}

\newcommand{\poly}{\mathrm{poly}}
\newcommand{\cA}{\mathcal{A}}
\newcommand{\cY}{\mathcal{Y}}
\newcommand{\cO}{\mathcal{O}}

\newcommand{\Bpoly}{B_{\poly}}
\newcommand{\tLa}{\widetilde{\Lambda}}
\newcommand{\HV}{H_V}
\newcommand{\hV}{h_V}
\newcommand{\gammav}{\gamma_{P(v)}}
\newcommand{\gammatv}{\gamma_{P(v+tV)}}

\newcommand{\dd}{\mathrm{d}}

\newcommand{\en}{\mathrm{en}}
\newcommand{\op}{\mathrm{op}}
\newcommand{\tr}{\mathrm{tr}}
\newcommand{\id}{\mathrm{id}}
\newcommand{\Lip}{\mathrm{Lip}}
\newcommand{\supp}{\mathrm{supp}}

\title{Data Completion for Electrical Impedance Tomography by Conditional Diffusion Models}

\author{Ke Chen \\
Department of Mathematical Sciences\\
University of Delaware\\
\texttt{kechen@udel.edu}\\
\And
Haizhao Yang \\
Department of Mathematics\\
Department of Computer Science\\
University of Maryland, College Park \\
\texttt{hzyang@umd.edu} \\
\And
Chugang Yi\thanks{Corresponding author. Authors are listed in alphabetical order.} \\
Department of Mathematics \\
University of Maryland, College Park \\
\texttt{chugang@umd.edu} \\
}

\begin{document}
\maketitle
\begin{abstract}
Data scarcity is a fundamental barrier in electrical impedance tomography (EIT), as undersampled Dirichlet-to-Neumann (DtN) measurements can substantially degrade conductivity reconstructions. We address this bottleneck by completing partially observed DtN measurements using a diffusion-based generative model. Specifically, we train a conditional diffusion model to learn the distribution of DtN measurements and to infer full measurements from partial observations. Our approach supports flexible source--receiver configurations and can be used as a plug-in preprocessing step for off-the-shelf EIT solvers. For polygonal conductivities, we use the distributed shape derivative of the DtN map to establish a Lipschitz continuity estimate independent of the number of vertices. We then leverage the estimate to obtain a nonasymptotic
total variation error estimate for diffusion-based DtN completion. In numerical experiments, we combine the proposed diffusion completion with a deep learning-based inverse solver and compare its performance against the same solver given full measurements. The results show that diffusion completion yields conductivity reconstructions that closely match the full measurement baseline while using only $\sim1\%$ of the measurements. In contrast, standard baselines such as matrix completion require $25\%$ of measurements to achieve comparable reconstruction quality.
\end{abstract}

\section{Introduction}

Inverse problems governed by partial differential equations (PDEs) arise across a broad range of scientific and engineering domains, including medical imaging \citep{arridge1999optical,cheney1990noser,borcea2002electrical}, geophysical exploration \citep{tarantola1982generalized,pratt1999seismic}, and materials characterization \citep{bonnet2005inverse,garcia2011non}. 
Many such inverse problems are nonlinear and severely ill-posed \citep{bal2012introduction, uhlmann2009electrical, isakov2006inverse}. 
Electrical impedance tomography (EIT) is a canonical example of these difficulties \citep{alessandrini1988stable, allers1991stability}. The goal of EIT is to determine the conductivity within a medium by applying voltage patterns on the boundary and measuring the induced current responses. One of the key components of EIT is the Dirichlet-to-Neumann (DtN) operator, which maps voltage patterns to current responses. Theoretically, the conductivity can be uniquely determined from complete knowledge of the DtN operator under mild conditions \citep{nachman1996global, astala2006calderon,uhlmann2013inverse, sylvester1987global}. In practice, the continuous DtN operator is often discretized into a DtN matrix containing voltage-current measurements on a finite number of boundary nodes. Typical methods for solving EIT include iterative solvers with regularization \citep{hanke1997regularizing, haber2000optimization,jin2012reconstruction, vauhkonen2002tikhonov}, direct methods \citep{siltanen2000implementation,chow2014direct, mueller2020d, chow2021direct}, and recent deep learning approaches~\citep{hamilton2018deep,arridge2019solving, bao2020numerical,ong2022integral, guo2023transformer, cen2023electrical, chen2024pseudo, molinaro2023neural, bhat2025neural}. However, when only limited measurements are available, as is common in medical imaging \citep{heines2023clinical, putensen2019electrical, scaramuzzo2024electrical}, reconstruction pipelines that presuppose fully sampled DtN measurements often fail to produce reliable solutions.

To address a limited measurement budget, active acquisition strategies such as optimal experimental design (OED) and reinforcement learning (RL) aim to identify measurement configurations or sensing policies that extract the most informative measurements from a limited number of experiments. Their effectiveness has been demonstrated in EIT \citep{hyvonen2024bayesian,jin2024continuous} and other inverse problems \citep{huan2013simulation,shen2022learning,alexanderian2021optimal,jiang2024reinforced,go2025accurate,koval2025non}. 
Although these approaches can improve the quality of the collected data, they typically require repeated forward or inverse solves. Moreover, partial measurements can hinder the use of off-the-shelf reconstruction methods that expect complete measurements. For instance, in inverse scattering, limited aperture data leads to artifacts known as ``shadow regions''~\citep{colton2006target, li2015recovering}, while in EIT partial boundary data exacerbates the loss of spatial resolution and causes geometrical distortion~\citep{hauptmann2017approximation, alsaker2017direct}.

Another line of work seeks to reconstruct complete measurements from partial observations across a range of inverse problems \citep{boukari2015convergent, harrach2015interpolation, hauptmann2017approximation,caubet2019new,dou2022data,bui2022bridging}. By separating data completion from inversion, this approach provides a plug-and-play workflow in which the completed measurements can be passed directly to off-the-shelf inverse solvers designed for complete measurements. For EIT, \citet{harrach2015interpolation} predict the missing voltages on current-driven boundary nodes in the difference Neumann-to-Dirichlet (NtD) matrix, relying on the geometry-dependent smoothness of difference NtD data; \citet{hauptmann2017approximation} approximate continuum NtD data from boundary measurements via a two-stage optimization that first lifts voltages to smooth boundary traces using Tikhonov regularization and then applies a prior-regularized least-squares fit to recover an approximate full NtD map; and the most closely related work, \citet{bui2022bridging} completes unobserved entries in the off-diagonal blocks of the hierarchically partitioned DtN matrix by exploiting the low-rank structure inherited from the underlying elliptic operator \citep{bebendorf2003existence}. Specifically, \citet{bui2022bridging} achieve high-accuracy completion when the off-diagonal blocks of the DtN matrix are randomly sampled at a moderate sampling rate. 

These EIT data completion methods \citep{harrach2015interpolation,hauptmann2017approximation,bui2022bridging} leverage PDE-induced smoothness or low-rank structure, but they may become less reliable under highly structured missingness, extremely sparse electrode layouts, or strong measurement noise. They typically perform completion in an instance-wise manner, requiring a problem-specific interpolation, regularization, or matrix-completion routine for each new measurement set and acquisition configuration. In contrast, we propose a data-driven completion framework based on conditional diffusion models. By exploiting the low-dimensional structure of the conductivity distribution, our method amortizes completion across instances and enables accurate DtN completion under flexible and challenging measurement configurations. Diffusion models learn high-dimensional distributions
\citep{song2020score,ho2020denoising}, and recent convergence analyses
exploit low-dimensional structure in the target distribution
\citep{tang2024conditional,azangulov2024convergence,liang2025low}.
Low-dimensional conductivity models have also been studied in EIT
\citep{seo2019learning, alberti2022inverse,alberti2024manifold}. The low-dimensional structure of the conductivity passes to the measurements through the forward DtN map. Its differentiability under polygon vertex perturbations was established by \citet{beretta2017differentiability}. Using the distributed form of the shape derivative, we bound the operator norm of the derivative, showing that the DtN map is Lipschitz with a constant independent of the number of vertices. Building on the low-dimensional parameterization of the conductivity class and a Lipschitz bound for the forward DtN map, we analyze the distributional error of conditional diffusion-based DtN completion and derive a nonasymptotic bound in total variation. The main contributions of this work can be summarized as follows.
\begin{itemize}
  \item \textbf{DtN completion via conditional diffusion.} To the best of our knowledge, this is the first framework that learns the distribution of discrete DtN measurements via a diffusion model and performs DtN completion by posterior sampling conditioned on partial observations.

\item \textbf{Completion under extreme undersampling.} The ill-posedness of EIT makes learning-based completion particularly challenging under extreme undersampling, since DtN measurements generated by different conductivities may be very similar. Motivated by this difficulty, we incorporate an EIT-tailored normalization into the learning pipeline and demonstrate accurate completion from only $\sim1\%$ of the measurement entries in our experiments, a sampling rate at which low-rank matrix completion fails.

\item \textbf{Flexible measurement configurations.} Our method accommodates flexible source--receiver patterns by encoding partial measurements as conditioning inputs to the score network. Thus, our model can learn to infer missing entries under varied masks, providing robust completion under diverse measurement configurations.

\item \textbf{Error analysis.} For polygonal conductivities, building on the differentiability of the DtN map with respect to vertex perturbations \citep{beretta2017differentiability}, we establish a chartwise Lipschitz estimate that is uniform in the number of vertices. In an idealized conditional DDPM formulation of discretized DtN completion, we combine this estimate with the low-dimensional diffusion theory of \citet{liang2025low} to derive a nonasymptotic error bound in total variation.

\end{itemize}

The paper is organized as follows. Section~\ref{sec:problem_setup} formalizes the problem; Section~\ref{sec:method} presents our diffusion completion framework for DtN measurements; Section~\ref{sec:theory} develops the error analysis; and Section~\ref{sec:numeric} reports numerical results.

\section{Problem Setup}\label{sec:problem_setup} 

Let $\Omega\subset\mathbb R^2$ be the unit disk. Assume that the
conductivity $\gamma\in L^\infty(\Omega)$ satisfies
$\operatorname*{ess\,inf}_{\Omega}\gamma>0$.
Given a boundary voltage $f\in H^{1/2}(\partial\Omega)$, the forward
problem is to solve
\begin{equation}\label{eq:forward}
    \nabla \cdot \bigl(\gamma(x)\nabla u(x)\bigr)=0
    \quad \text{in } \Omega,
    \qquad
    u(x)=f(x)
    \quad \text{on } \partial\Omega.
\end{equation}
We denote the corresponding solution to \Eqref{eq:forward} by $u_\gamma^f$. The induced boundary
current is given by
\begin{equation}\label{eq:current}
    g(x)=\gamma(x)\frac{\partial u_\gamma^f}{\partial n}(x)
    \quad \text{on } \partial\Omega.
\end{equation}
Here $\partial/\partial n$ denotes the outward normal derivative,
and the boundary flux is understood in the weak sense. The map $f\mapsto g$ defines the Dirichlet-to-Neumann operator
\begin{equation}\label{eq:DtN_cont}
    \Lambda_\gamma:H^{1/2}(\partial\Omega)\to H^{-1/2}(\partial\Omega),
\end{equation}
which maps an applied boundary voltage to the corresponding boundary current.
Equivalently, for a given voltage $f$, we write
    $g=\Lambda_\gamma f$.

The weak form of \Eqref{eq:forward} is: find $u\in H^1(\Omega)$ with
$u|_{\partial\Omega}=f$ such that
\begin{equation}\label{eqn:weak_form}
    \int_\Omega \gamma(x)\nabla u(x)\cdot\nabla v(x)\,\mathrm{d}x=0,
    \qquad \forall v\in H_0^1(\Omega).
\end{equation}

In electrical impedance tomography (EIT), the goal is to recover the conductivity
$\gamma(x)$ from multiple boundary voltage--current measurement pairs
$\{(f_i,g_i)\}_{i=1}^m$. Here $f_i$ is the applied boundary voltage determined
by the experimental design, and $g_i=\Lambda_\gamma f_i$
is the corresponding measured boundary current. Ideally, if one has access to
the full Dirichlet-to-Neumann operator $\Lambda_\gamma$, namely the boundary
current $g$ for every admissible boundary voltage $f$, then the
conductivity $\gamma$ is uniquely determined under suitable assumptions
\citep{uhlmann2009electrical,isakov2006inverse,astala2006calderon}.
However, infinitely many boundary measurements are not attainable in practice.
Therefore, recovering $\gamma$ from a finite collection of measured voltage--current
pairs $(f_i,g_i)$ is the practically relevant inverse problem.

\paragraph{Finite element discretization.}
We approximate $\Omega$ by a polygonal domain $\Omega_h$ equipped with a
shape-regular triangulation $\mathcal T_h$, and use the linear ($P_1$) finite
element space $V_h\subset H^1(\Omega_h)$ of continuous, piecewise-linear
functions on $\mathcal T_h$. Let $\{\phi_k\}_{k=1}^{N_h}$ be the nodal basis of
$V_h$. We split the nodal indices into the interior nodes $\mathcal I$ and the
boundary nodes $\mathcal B$, with
$N_{\mathcal I}=|\mathcal I|$ and $N_{\mathcal B}=|\mathcal B|$.

Define
\[
    V_h^0=\{v^h\in V_h: v^h|_{\partial\Omega_h}=0\}
\]
and the discrete trace space
\[
    V_h(\partial\Omega_h)
    =
    \operatorname{span}\{\phi_j|_{\partial\Omega_h}: j\in\mathcal B\}.
\]
Given boundary data $f^h\in V_h(\partial\Omega_h)$, we discretize
\Eqref{eqn:weak_form} by finding $u^h\in V_h$ with
$u^h|_{\partial\Omega_h}=f^h$ such that
\[
    a(u^h,v^h)
    \coloneqq
    \int_{\Omega_h}\gamma(x)\nabla u^h(x)\cdot\nabla v^h(x)\,\mathrm{d}x
    =
    0,
    \qquad
    \forall v^h\in V_h^0 .
\]

We assemble the stiffness matrix
$\mathbf K\in\mathbb R^{N_h\times N_h}$ with entries
\[
    \mathbf K_{ij}=a(\phi_i,\phi_j).
\]
Let $\mathbf u$, $\mathbf f$, and $\mathbf g$ denote the discrete solution
vector, the prescribed boundary voltage vector, and the measured boundary current
vector, respectively. With the ordering of degrees of freedom given by
interior nodes followed by boundary nodes, we partition $\mathbf K$ and
$\mathbf u$ as
\[
    \mathbf K=
    \begin{bmatrix}
        \mathbf K_{\mathcal I\mathcal I} & \mathbf K_{\mathcal I\mathcal B} \\
        \mathbf K_{\mathcal B\mathcal I} & \mathbf K_{\mathcal B\mathcal B}
    \end{bmatrix},
    \qquad
    \mathbf u=
    \begin{bmatrix}
        \mathbf u_{\mathcal I}\\
        \mathbf u_{\mathcal B}
    \end{bmatrix}
    =
    \begin{bmatrix}
        \mathbf u_{\mathcal I}\\
        \mathbf f
    \end{bmatrix}.
\]

The interior finite element equations give
\[
    \mathbf K_{\mathcal I\mathcal I}\mathbf u_{\mathcal I}
    +
    \mathbf K_{\mathcal I\mathcal B}\mathbf f
    =
    \mathbf 0 .
\]
The discrete Neumann data vector $\mathbf g\in\mathbb R^{N_{\mathcal B}}$ is
defined by
\[
    (\mathbf g)_j=a(u^h,\phi_j),
    \qquad j\in\mathcal B,
\]
with entries ordered according to the boundary nodes. Eliminating the interior
degrees of freedom yields the block relation
\begin{equation}
\label{eq:g_lambda}
    \mathbf g
    =
    \mathbf K_{\mathcal B\mathcal B}\mathbf f
    +
    \mathbf K_{\mathcal B\mathcal I}\mathbf u_{\mathcal I}
    =
    \underbrace{
    \left(
        \mathbf K_{\mathcal B\mathcal B}
        -
        \mathbf K_{\mathcal B\mathcal I}
        \mathbf K_{\mathcal I\mathcal I}^{-1}
        \mathbf K_{\mathcal I\mathcal B}
    \right)
    }_{\mathbf\Lambda_\gamma}
    \mathbf f .
\end{equation}
Therefore, the discrete Dirichlet-to-Neumann matrix
$\mathbf\Lambda_\gamma\in\mathbb R^{N_{\mathcal B}\times N_{\mathcal B}}$
is the Schur complement of the stiffness matrix $\mathbf K$ with respect to the
interior block $\mathbf K_{\mathcal I\mathcal I}$. If the $k$-th experiment imposes the canonical boundary voltage
$\mathbf f=\mathbf e_k$, then
\[
    \mathbf g=\mathbf\Lambda_\gamma\mathbf e_k .
\]
Hence the corresponding measured boundary current vector is exactly the
$k$-th column of $\mathbf\Lambda_\gamma$. Consequently, under the canonical set of
boundary excitations, $\mathbf\Lambda_\gamma$ coincides with the complete discrete
current-measurement table on the prescribed boundary node set.

\paragraph{Partial observations and data completion.}
When only some boundary excitations and current measurements are available, only a subset of the entries of $\mathbf\Lambda_\gamma$ is observed. Let
$\mathbf M_s\in\{0,1\}^{N_{\mathcal B}\times N_{\mathcal B}}$ be the binary mask
matrix associated with the chosen measurement configuration, where
\[
    (\mathbf M_s)_{ij}
    =
    \begin{cases}
        1, & \text{if the entry } (i,j) \text{ is observed},\\
        0, & \text{otherwise}.
    \end{cases}
\]
Here $s\in(0,1)$ specifies the target sampling level used to construct the mask, while the
realized fraction of observed entries for a particular mask is
$
    \frac{\|\mathbf M_s\|_0}{N_{\mathcal B}^2}.
$
Thus, the \emph{configuration} of how the DtN measurements are obtained is encoded
by $\mathbf M_s$, and the observed incomplete DtN measurements are denoted by
\begin{equation}
\label{eq: lambda_obs}
    \mathbf{\Lambda}^{\mathrm{o}}_\gamma
    =
    \mathbf M_s\odot \mathbf{\Lambda}_\gamma,
\end{equation}
where $\odot$ denotes the Hadamard product. The specific measurement configurations used in our numerical experiments are illustrated in Section~\ref{sec:data_gen}. The $\mathbf M_s$ encodes the measurement configuration and may
represent highly structured missingness, such as missing columns corresponding
to unperformed boundary excitations or subsampled receiver sets within each
experiment. In particular, the observed entries are not necessarily generated by
i.i.d.\ entrywise sampling. Given $\mathbf\Lambda_\gamma^{\mathrm o}$ and $\mathbf M_s$, we seek a completed matrix $\widehat{\mathbf\Lambda}_\gamma$ that approximates $\mathbf\Lambda_\gamma$.

\section{Method}\label{sec:method}
In this section, we will formulate data completion as conditional sampling and approximate the conditional law with a diffusion model. As multiple completions may plausibly explain the same partial observations, we model $\mathbf{\Lambda}_\gamma$ as a random matrix whose randomness comes from some distribution of $\gamma$ and the goal is to generate samples from the learned conditional distribution $p\!\left(\mathbf{\Lambda}_{\gamma}\,\middle|\,\mathbf{\Lambda}^{\mathrm{o}}_{\gamma}, \mathbf{M}_s\right)$. 
The generated posterior samples can then be used to estimate the full DtN matrix $\widehat{\mathbf{\Lambda}}_\gamma$, and standard inverse solvers including iterative solvers, optimization-based reconstruction, or deep learning-based methods can be applied to recover conductivity $\gamma$.

\subsection{Diffusion Model} 
Given i.i.d.\ samples from a target distribution $p_{\text{target}}$ on $\R^d$, a diffusion model seeks to generate new samples whose distribution approximates $p_{\text{target}}$.
Let $\{\alpha_t\}_{t=1}^T\subset(0,1)$ and define $\bar{\alpha}_t\coloneqq \prod_{i=1}^t\alpha_i$. Let $X_0\in\R^d$ be a random vector with law $p_{\text{target}}$.
Define the forward Markov chain by \(X_t=\sqrt{\alpha_t}\,X_{t-1}+\sqrt{1-\alpha_t}\,W_t\) for \(t=1,\dots,T\),  with \(W_t\stackrel{\text{i.i.d.}}{\sim}\mathcal N(0,I_d)\); equivalently, \(X_t\mid X_0\sim \mathcal N(\sqrt{\bar\alpha_t}\,X_0,(1-\bar\alpha_t)I_d)\), hence \(X_t=\sqrt{\bar\alpha_t}\,X_0+\sqrt{1-\bar\alpha_t}\,\overline W\) for \(\overline W\sim\mathcal N(0,I_d)\). The sequence $\{\alpha_t\}_{t=1}^T$ is a user-specified noise schedule. It is customary to choose $\alpha_t$ nonincreasing in $t$, so that the incremental noise magnitude $1-\alpha_t$ increases along the chain, meaning that early steps only weakly perturb $X_0$ while later
steps inject progressively stronger Gaussian noise. One also typically selects the schedule so that $\bar{\alpha}_T$ is very small, in which case $X_T$ is dominated by its Gaussian component, which motivates initializing the reverse process from a standard Gaussian. Time reversal of the diffusion equation models~\citep{anderson1982reverse} motivates the sampling algorithm in diffusion models. The key aspect of the sampling algorithm from a diffusion model is to construct a reverse process $Y_t$, which starts from pure Gaussian noise and can gradually convert to a new sample $Y_0$ such that $p_{Y_0} \approx p_{\text{target}}$. In particular, the reverse process relies on the availability of the so called \emph{score function} $s_t^\star(\cdot)\coloneqq \nabla\log p_{X_t}(\cdot)$, thus the diffusion model is also called a score-based generative model \cite{song2020score}. A classical choice of the reverse process is the Denoising Diffusion Probabilistic Model (DDPM) \citep{ho2020denoising} sampler as follows, which can be treated as a discretization of the reverse SDE. 
\begin{definition}[Unconditional DDPM sampler]\label{def:ddpm}
Given a score estimator $s_t(\cdot)$ and the same noise schedule $\{\alpha_t\}$ as in the forward process, the DDPM sampler is defined as
\begin{equation}
Y_T\sim\mathcal N(0,I_d),\qquad
Y_{t-1}
=\frac{1}{\sqrt{\alpha_t}}\Big(Y_t+(1-\alpha_t)\,s_t(Y_t)+\sigma_t W_t\Big),
\quad W_t\stackrel{\text{i.i.d.}}{\sim}\mathcal N(0,I_d)\,,
\end{equation}
with 
$\sigma_t
\;=\;\sqrt{\frac{(\alpha_t-\bar{\alpha}_t)(1-\alpha_t)}{1-\bar{\alpha}_t}}\,,\ t=T,\ldots,1\,.$

\end{definition}
When a score estimator $s_t(\cdot) \approx s_t^\star(\cdot)$ is available, the DDPM sampler in Definition~\ref{def:ddpm} provides a way to generate new samples from the target distribution \citep{ho2020denoising}. In practice, the score estimator is parametrized by a neural network $s_{\theta}(\cdot, t)$, and the parameters $\theta$ can be learned through the \emph{score matching loss} \citep{ho2020denoising, song2020score, vincent2011connection}:
\begin{equation}
\label{eq:score-matching} 
    \min_{\theta} \mathbb{E}_{t \sim \mathrm{Unif}{(1,\dots,T)}, X_0 \sim p_{\text{target}}, X_t \sim p_{X_t \mid X_0}} \lambda(t)\| s_{\theta}(X_t,t) - \nabla_{x_t} \log p(X_t \mid X_0) \|^2\,,
\end{equation}
where $\lambda(t)>0$ is a weighting function. The Gaussian forward
transition gives the analytic regression target
$
\nabla_{x_t}\log p(X_t\mid X_0)
=
-\frac{X_t-\sqrt{\bar\alpha_t}X_0}{1-\bar\alpha_t}.
$
At the population level, minimizing this loss over all square-integrable
functions of $(X_t,t)$ yields the marginal score
$s_t^\star=\nabla\log p_{X_t}$
\citep{hyvarinen2005estimation,song2020score}.
In practice, we compute the average loss on each batch of data using randomly sampled diffusion times and Gaussian noise.

\subsection{Conditional Diffusion Models for Data Completion}
\label{sec:proposed_method}
The diffusion model is also flexible in incorporating a conditional covariate, denoted as $Z$, to guide the generation of new samples. For each fixed $Z=z$, the DDPM sampler in Definition~\ref{def:ddpm} can be extended to conditional generation when the
conditional score estimator $s_t(\cdot, z) \approx s_t^{\star}(\cdot,z)$ is given, where
$s_t^{\star}(\cdot,z)=\nabla \log p_{X_t\mid Z=z}(\cdot)$. The following conditional DDPM sampler can generate a new sample $Y_0^{z}$ such that $p_{Y_0^z} \approx p_{X_0 \mid Z=z}$.
\begin{definition}[Conditional DDPM sampler]\label{def:cond-DDPM-sampler}
For a fixed $Z=z$, let $s_t(\cdot,z)$ approximate the
conditional score $s_t^\star(\cdot,z)$. Using the same $\alpha_t$, $\sigma_t$ from the unconditional DDPM sampler in Definition~\ref{def:ddpm}, the conditional DDPM sampler is defined as 
\begin{equation}
Y_T^z\sim\mathcal N(0,I_{d}),\qquad
Y_{t-1}^z=\frac{1}{\sqrt{\alpha_t}}\Big(Y_t^z+(1-\alpha_t)\,s_t(Y_t^z, z)+\sigma_tW_t\Big),
\ \ W_t\stackrel{\text{i.i.d.}}{\sim}\mathcal N(0,I_{d})\,.
\end{equation}

\end{definition}
Similarly, the conditional score function can also be parametrized by a neural network $s_{\theta}(\cdot, z, t)$, and trained using the conditional score-matching loss,
\begin{equation}
\label{eq:cond-score-matching}
    \min_{\theta} \mathbb{E}_{t \sim \mathrm{Unif}{(1,\dots,T)}, (X_0, Z) \sim p_{(X_0,Z)}, X_t \sim p_{X_t \mid X_0}} \lambda(t)\| s_{\theta}(X_t,Z,t) - \nabla_{x_t} \log p(X_t \mid X_0) \|^2\,,
\end{equation}
where $p_{(X_0,Z)}$ is the joint distribution of the target and
conditioning variables. The Gaussian perturbations are sampled
independently of $(X_0,Z)$, and the regression target is the same as in the unconditional case. At the population level, minimizing this loss over all square-integrable
functions of $(X_t,Z,t)$ yields the conditional score
$s_t^\star(x,z)=\nabla_x\log p_{X_t\mid Z=z}(x)$
\citep{vincent2011connection,batzolis2021conditional}, though $p_{X_t \mid Z}$ does not appear in the loss function explicitly. 

Recall that in EIT data completion we aim to sample the conditional law
$p(\mathbf{\Lambda}_\gamma \mid \mathbf{\Lambda}_{\gamma}^{\mathrm{o}},\mathbf{M})$
given partial DtN measurements $\mathbf{\Lambda}_{\gamma}^{\mathrm{o}}$ and a mask configuration
$\mathbf{M}\in\{0,1\}^{N_{\mathcal B}\times N_{\mathcal B}}$.
Throughout, bold symbols denote matrix-valued EIT quantities, while the diffusion model is formulated on
their vectorized representations in $\R^{d_N}$, with $d_N\coloneqq N_{\mathcal B}^2$.
We instantiate the conditional diffusion model by setting
\[
X_0 \coloneqq \mathrm{vec}(\mathbf{\Lambda}_\gamma)\in\R^{d_N},
\qquad
M \coloneqq \mathrm{vec}(\mathbf{M})\in\{0,1\}^{d_N}.
\]
The observed vector satisfies
$
 M\odot X_0
= \mathrm{vec}(\mathbf{\Lambda}_\gamma^{\mathrm{o}})
$
and we take the conditioning variable as $Z\coloneqq (M\odot X_0,\,M)$. During training, fully observed DtN samples from
$\mathcal{D}_{\mathrm{train}}$ are paired with masks drawn from
$\mathcal{M}$ to form $(X_0,Z)$ as defined above.
The conditional score network $s_\theta(X_t,Z,t)$ is trained using
\Eqref{eq:cond-score-matching}. At test time, given $\mathbf{\Lambda}_{\gamma}^{\mathrm{o}}=\mathbf{M}^{\mathrm{test}}\odot \mathbf{\Lambda}_{\gamma}$
and $\mathbf{M}^{\mathrm{test}}$, we form
$Z^{\mathrm{test}} \coloneqq \bigl(\mathrm{vec}(\mathbf{\Lambda}_{\gamma}^{\mathrm{o}}),\, \mathrm{vec}(\mathbf{M}^{\mathrm{test}})\bigr)$
and run the conditional DDPM sampler in Definition~\ref{def:cond-DDPM-sampler} to generate
$Y^z_0\in\R^{d_N}$.
We reshape $Y^z_0$ into an $N_{\mathcal B}\times N_{\mathcal B}$ matrix and denote it by
$\widehat{\mathbf{\Lambda}}_\gamma$.

The training and sampling procedures are summarized in
Algorithm~\ref{alg:train-conditional-score} and Algorithm~\ref{alg:complete-then-reconstruct}.

\begin{algorithm}[htbp]
\caption{Training the conditional score network for DtN completion}
\label{alg:train-conditional-score}
\begin{algorithmic}[1]
\Require Fully observed DtN dataset $\mathcal{D_{\rm train}}=\{\mathbf{\Lambda}_{\gamma_i}\}_{i=1}^{N_{\text{train}}}$; measurement configuration distribution $\mathcal{M}$; schedule $\{\alpha_t\}_{t=1}^T$, weight $\lambda(t)$, batch size $B$; conditional score network $s_\theta(X_t,Z,t)$ to be optimized and an optimizer
\While{not converged}
  \State Sample $\{\mathbf{\Lambda}_{\gamma_i}\}_{i=1}^B$ from $\mathcal{D_{\rm train}}$, and $\{\mathbf{M}^{(i)}\}_{i=1}^B \sim \mathcal{M}$
  \State Set $X_0^{(i)} \gets \mathrm{vec}(\mathbf{\Lambda}_{\gamma_i})$ and $M^{(i)} \gets \mathrm{vec}(\mathbf{M}^{(i)})$
\State Set $Z^{(i)} \gets (M^{(i)}\odot X_0^{(i)},\, M^{(i)})$
\State Sample batch $\{t_i\}_{i=1}^B \sim \operatorname{Unif}\{1,\dots,T\}$ and noise $\{W^{(i)}\}_{i=1}^B \sim \mathcal{N}(0,I_{d_N})$
  
  \State $X_{t_i}^{(i)} \gets \sqrt{\bar\alpha_{t_i}}\,X_0^{(i)} + \sqrt{1-\bar\alpha_{t_i}}\,W^{(i)}$ \Comment{forward diffusion}
  
  \State \(r_{t_i}^{(i)}\gets
-\dfrac{X_{t_i}^{(i)}-\sqrt{\bar\alpha_{t_i}}X_0^{(i)}}
{1-\bar\alpha_{t_i}}\)
\State \(\mathcal L\gets\dfrac1B\sum_{i=1}^B\lambda(t_i)
\bigl\|s_\theta(X_{t_i}^{(i)},Z^{(i)},t_i)
-r_{t_i}^{(i)}\bigr\|_2^2\)
  \State Update $\theta$ using optimizer step to minimize $\mathcal{L}$
\EndWhile
\State \textbf{return} $\theta$ \Comment{trained conditional score network}
\end{algorithmic}
\end{algorithm}

\begin{algorithm}[htbp]
\caption{DtN measurements completion via conditional diffusion model}
\label{alg:complete-then-reconstruct}
\begin{algorithmic}[1]
\Require Observed (masked) DtN measurements $\mathbf{\Lambda}^{\text{o}}_{\gamma}$ and its mask $\mathbf{M}$; trained $s_\theta$; schedule $\{\alpha_t\}_{t=1}^T$; conditional DDPM sampler in Definition~\ref{def:cond-DDPM-sampler}
\State Set $M \gets \mathrm{vec}(\mathbf{M})$, and set $Z \gets (\mathrm{vec}(\mathbf{\Lambda}_\gamma^{\mathrm{o}}),\, M)$
\State Initialize $Y^z_T \sim \mathcal{N}(0,I_{d_N})$

\For{$t=T,\,T\!-\!1,\,\dots,\,1$}
  \State Sample $W_t \sim \mathcal{N}(0,I_{d_N})$
  \State $Y^z_{t-1} \gets \dfrac{1}{\sqrt{\alpha_t}}\Big(Y^z_t + (1-\alpha_t)\, s_\theta(Y^z_t, Z, t) + \sigma_t W_t\Big)$ \Comment{Conditional DDPM update, Definition~\ref{def:cond-DDPM-sampler}}
\EndFor
\State Reshape $Y^z_0\in\R^{d_N}$ into an $N_{\mathcal B}\times N_{\mathcal B}$ matrix and denote it by $\widehat{\mathbf{\Lambda}}_\gamma$

\State \textbf{return} $\widehat{\mathbf{\Lambda}}_\gamma $ 
\end{algorithmic}
\end{algorithm}

\subsection{Solve the inverse problem by data completion}
\begin{figure}[!htb]
    \centering
    \includegraphics[width=0.8\linewidth]{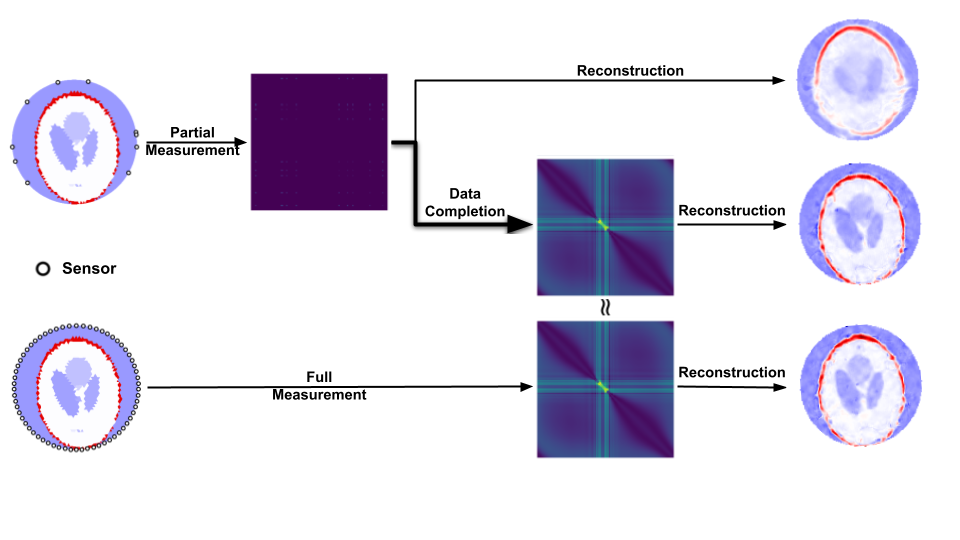}
\caption{Three strategies for EIT reconstruction. The first row shows direct
reconstruction from sparse measurements. The second row adds DtN
completion before reconstruction. The third row uses full measurements.
}
\label{fig:diagram}
\end{figure}
We illustrate our framework for solving the EIT inverse problem in Figure~\ref{fig:diagram}. The figure compares three reconstruction strategies: direct reconstruction from sparse DtN measurements (first row), reconstruction from DtN measurements that are first completed and then passed to the same solver (second row), and direct reconstruction from fully sampled DtN measurements with the same solver (third row). The completed DtN measurements produced by our method are designed to be plug-and-play. Once the partially observed operator is completed to $\widehat{\mathbf{\Lambda}}_\gamma$, it can be passed to a compatible full-data inverse solver using the same measurement representation. In our experiments, rather than using a classical PDE-based inverse solver, we adopt deep learning-based inverse solvers for their stable and excellent performance. Moreover, these deep learning models offer a unified pipeline to handle both full and partial measurements, making it transparent to isolate and evaluate the effect of the data completion step.

\section{Error Analysis}\label{sec:theory}

In this section, we analyze an idealized conditional DDPM formulation for completing discretized DtN measurements from polygonal conductivities. Let $\gamma$ be a random conductivity taking values in the admissible class $K$ introduced below, and let $(X_t)_{t=0}^T$ be the forward diffusion initialized at $X_0=F_N(\gamma)$. For each fixed condition $Z=z$, let $(Y_t^z)_{t=0}^T$ be the reverse chain in Definition~\ref{def:cond-DDPM-sampler} using the conditional score estimator $s_\theta(\cdot,z,t)$. Under the conditions of Theorem~\ref{thm:cond-ddpm}, our main result gives, for almost every $z$ with respect to the law of $Z$, the total variation (TV) bound $\mathrm{TV}\bigl(p_{X_1\mid Z=z},p_{Y_1^z}\bigr)\lesssim \frac{2n_{\mathrm v}\log^3 T}{T}+\varepsilon_{\mathrm{score}}\sqrt{\log T}$.

Following \citet[Thm.~2]{liang2025low}, the comparison is made at the first positive noise level $t=1$. The theorem does not provide a TV comparison between the clean laws $p_{X_0\mid Z=z}$ and $p_{Y_0^z}$ at $t=0$. Such a bound cannot generally be expected when the clean target law is supported on a lower-dimensional set of Lebesgue measure zero and the sampler output has a density. See Remark~\ref{rmk: t1} for details.

The analysis proceeds as follows. Section~\ref{sec:K-and-param} introduces the admissible conductivity class $K$, its vertex parametrization, and a finite collection of local charts covering $K$. Sections~\ref{sec:chartwise_DtN} and~\ref{sec:chartwise_lipschitz_discretized_DtN} establish chartwise Lipschitz estimates for the continuum DtN map and the discretized measurement map $F_N$, respectively, with constants independent of the number of vertices. Section~\ref{sec:entropy-bdd-FN} uses these estimates to bound the covering numbers of $F_N(K)$ and show that this set is uniformly bounded. Finally, Section~\ref{sec:ddpm_FNK} combines these bounds with Assumption~\ref{ass:cond_score} and Lemma~\ref{lem:liang-ddpm} to prove Theorem~\ref{thm:cond-ddpm}.

\paragraph{Notation.}
We denote by 
$\dist(\cdot,\cdot)$ the Euclidean distance between points or subsets in $\mathbb{R}^2$, and $|\cdot|$ the Euclidean norm in $\mathbb{R}^2$. 
For a measurable set $E\subset\R^2$, $|E|$ denotes its Lebesgue measure.
For nonempty sets $A,B\subset\R^2$, the Hausdorff distance is  
$
d_H(A,B)\coloneqq \max\Big\{\sup_{a\in A} \inf_{b \in B}\dist(a,b),\ \sup_{b\in B} \inf_{a \in A} \dist(b,a)\Big\}.
$
For measurable sets $E,F\subset\R^2$, the symmetric difference is
$E\Delta F\coloneqq (E\setminus F)\cup(F\setminus E)$. For $r>0$, let
$B_2(x,r)\coloneqq \{y\in\R^2:\ |y-x|<r\}$, and 
$\overline{B}_2(x,r)\coloneqq \{y\in\R^2:\ |y-x|\le r\}$. We write
$\diam(\Omega)\coloneqq \sup\{|x-y|:\ x,y\in\Omega\}$ and denote the indicator function of a set $E\subset\R^2$ by $\chi_E$. Throughout, $\log$ denotes the natural logarithm, and
$
\log_+(x)\coloneqq\max\{\log x,0\}.
$

\subsection{Admissible conductivities and parametrization}\label{sec:K-and-param}
We parameterize the polygonal conductivity class by its vertices.
For an integer $n_{\mathrm v}\ge3$, let
$v=(v_0,\ldots,v_{n_{\mathrm v}-1})\in(\R^2)^{n_{\mathrm v}}$
be a list of pairwise distinct points, and set $v_{n_{\mathrm v}}=v_0$.
We define the closed polygonal chain
\[
\partial P(v)\coloneqq \bigcup_{i=0}^{n_{\mathrm v}-1}[v_i,v_{i+1}],
\qquad
[v_i,v_{i+1}]\coloneqq \{(1-t)v_i+t v_{i+1}:t\in[0,1]\}.
\]
We denote by $\mathcal V^{n_{\mathrm v}}$ the set of vertex lists $v$ such that $\partial P(v)$ is a
simple closed polygonal curve and the vertices are listed counterclockwise. 
For $v\in\mathcal V^{n_{\mathrm v}}$, the Jordan curve theorem yields a unique bounded connected component
$\operatorname{int}(P(v))$ of $\R^2\setminus \partial P(v)$; we define the (closed) $n_{\mathrm v}$-gon $P(v)\coloneqq \overline{\operatorname{int}(P(v))}$. Note that the parametrization $v\mapsto P(v)$ is not injective due to the ordering of vertices (cyclic shifts). To eliminate this labeling ambiguity, we will work with local charts that fix a reference ordering of the vertices; within each chart the map $\gamma\mapsto v$ becomes well-defined and injective, which is crucial for differentiability and Lipschitz estimates later.

For $v\in\mathcal V^{n_{\mathrm v}}$, we denote by $\beta_i(v)$
the interior angle of $P(v)$ at the vertex $v_i$. We also define the polygon norm
\[
\|v\|_{\mathrm{poly}}\coloneqq \max_{0\le i\le n_{\mathrm v}-1}|v_i|,
\qquad
\|v-w\|_{\mathrm{poly}}\coloneqq \max_{0\le i\le n_{\mathrm v}-1}|v_i-w_i|.
\]
For $r>0$, define
$
B_{\mathrm{poly}}(v,r)
\coloneqq
\left\{w\in\R^{2n_{\mathrm v}}:
\|w-v\|_{\mathrm{poly}}<r\right\}.
$
Identifying $v$ with an element of $\R^{2n_{\mathrm v}}$ by concatenation, we have
$\|v\|_{\mathrm{poly}}\le \|v\|_2 \le \sqrt{n_{\mathrm v}}\,\|v\|_{\mathrm{poly}}$. 

\begin{assumption}\label{ass:admissible_poly}
Fix integers $n_{\mathrm v}\ge 3$ and constants $\beta_0\in(0,\pi/2)$, $d_0>0$, $d_1>0$, $\kappa>0$, $\kappa\neq 1$.
Let $\Omega\subset\R^2$ be the unit disk.
\begin{definition}[Admissible polygon class]\label{def:admiss}
Define the \emph{admissible polygon class} $\mathcal A$ as the set of all polygons $P=P(v)$ with their vertices $v\in\mathcal V^{n_{\mathrm v}}$
such that:
\begin{enumerate}
\item[(A1)] $P\subset\Omega$ and $\dist(P,\partial\Omega)\ge d_0$.
\item[(A2)] every side length satisfies $|v_{i+1}-v_i|\ge d_1$ for all $i$ (indices mod $n_{\mathrm v}$).
\item[(A3)] every interior angle satisfies $\beta_0\le \beta_i(v)\le \pi-\beta_0$ for all $i$.
\end{enumerate}
\end{definition}
\begin{definition}[Relaxed admissible polygon class]
Define the \emph{relaxed class} $\mathcal A_{1/2}$ by requiring the same conditions as above,
but with strict inequalities, and the parameters are relaxed to half of their values in $\mathcal{A}$:
\begin{enumerate}
\item[(A1$'$)] $P\subset\Omega$ and $\dist(P,\partial\Omega) > d_0/2$.
\item[(A2$'$)] $|v_{i+1}-v_i| > d_1/2$ for all $i$ (indices mod $n_{\mathrm v}$).
\item[(A3$'$)] every interior angle satisfies $\frac{\beta_0}{2}< \beta_i(v)< \pi-\frac{\beta_0}{2}$ for all $i$.
\end{enumerate}
\end{definition}
\end{assumption}

We say $v$ is admissible if $P(v) \in \mathcal{A}$ and relaxed admissible if $P(v) \in \mathcal{A}_{1/2}$. 
For a polygon $P$ we consider conductivity
$
\gamma_{P}\coloneqq 1+(\kappa-1)\chi_{P}\,,
$
and define the relaxed and admissible conductivity classes
\[
K_{1/2}\coloneqq \{\gamma_P:\ P\in\mathcal A_{1/2}\},
\qquad
K\coloneqq \{\gamma_P:\ P\in\mathcal A\}\subset K_{1/2}\,.
\]
Throughout,
$
\gamma_{\min}:=\min\{1,\kappa\}$, $
\gamma_{\max}:=\max\{1,\kappa\}.
$
To remove the residual cyclic ambiguity globally, we fix once and for all
a representative selection as follows. For a counterclockwise vertex set, the
index $0$ vertex is chosen to be the lexicographically smallest vertex, and the remaining
vertices are listed counterclockwise. Since the vertices are pairwise distinct,
this starting vertex is unique.
Let
\[
\mathcal V_{\mathcal A}
:=
\{v\in\mathcal V^{n_{\mathrm v}}:\ P(v)\in\mathcal A
\text{ and }v\text{ is labeled by the above representative rule}\}.
\]

\begin{remark}[A priori data]\label{rem:apriori_poly}
Unless stated otherwise, all constants introduced below depend only on the
\emph{a priori data} $(\Omega, \kappa, d_0, d_1, \beta_0)$, and in particular
are independent of $n_{\mathrm v}$. The constants
are also independent of the discretization level $N$ and the diffusion step
count $T$ used later in
Sections~\ref{sec:chartwise_lipschitz_discretized_DtN},
\ref{sec:entropy-bdd-FN}, \ref{sec:ddpm_FNK}. When a constant is allowed to
depend on $n_{\mathrm v}$, this will be stated explicitly. In particular, the
vertex-ordering stability constant $C_{\mathrm{vs}}$ in
Lemma~\ref{lem:vertex_stability_poly} is such an exception: it may depend on
$n_{\mathrm v}$ through the vertex-count parameter in the Beretta--Francini
stability result. This constant is used only for qualitative chart topology and
is not used in the quantitative entropy or diffusion bounds.
\end{remark}
\begin{remark}\label{rem:Lipschitz-bdy-derivable}
The assumptions in Definition~\ref{def:admiss} imply that $\mathcal{A}$
consists of strictly convex polygons with edge length bounded below and
positive distance from $\partial\Omega$. In particular, the uniform
Lipschitz-boundary property required by the shape-derivative results of
\cite{beretta2017differentiability} can be derived from (A1)--(A3); we
verify this in Appendix~\ref{sec:lipschitz_bdy_class}
(Lemma~\ref{lem:A4_from_A23_conv}).
\end{remark}
The admissible class $\mathcal A$ is defined by closed inequalities, so an admissible polygon may fail to have an open neighborhood that stays admissible under vertex perturbations.
We therefore introduce a slightly relaxed class $\mathcal A_{1/2}$ as a buffer, which guarantees that small perturbations of an admissible polygon remain within a regime where vertex-based charts and shape-differentiability results apply.
All local constructions are carried out in $\mathcal A_{1/2}$, and we then restrict the final bounds back to $K\subset K_{1/2}$.

\begin{lemma}[Patch radius from local admissibility]\label{lem:patch-radius}
There exists $r_0>0$, depending only on $(d_0,d_1,\beta_0)$, such that for every
$v\in\R^{2n_{\mathrm v}}$ with $P(v)\in\cA$ and every $V\in\R^{2n_{\mathrm v}}$
with $\|V\|_{\poly}\le r_0$, the perturbed configuration satisfies
$P(v+V)\in\cA_{1/2}$. One may take
\begin{equation}\label{eq:r0-explicit}
r_0 = \frac{1}{128}\,\min(d_0,\ d_1,\ d_1\beta_0^2).
\end{equation}
\end{lemma}

Lemma~\ref{lem:patch-radius}, proved in
Appendix~\ref{sec:proof_rho_s}, provides a uniform radius $r_0$
that depends only on $(d_0,d_1,\beta_0)$ and is independent of
$n_{\mathrm v}$.
For each admissible polygon $P^\star\in\mathcal A$,
Lemma~\ref{lem:chart-M} uses this radius to construct a chart
neighborhood $U_{P^\star}\subset K_{1/2}$.
These neighborhoods cover $K$.
Lemma~\ref{lem:logm} then gives a finite cover with $\log m=O(n_{\mathrm v})$.

\begin{lemma}[Uniform local chart centered at an admissible polygon]\label{lem:chart-M}
Let $r_0$ be the patch radius from Lemma~\ref{lem:patch-radius}. For each
$P^\star\in\cA$ with counterclockwise ordered vertices
$v^\star=(v_i^\star)_{i=0}^{n_{\mathrm v}-1}$, define
\[
\cO_{P^\star}\coloneqq \Bpoly(v^\star,r_0)
              = \{v\in\R^{2n_{\mathrm v}}:\|v-v^\star\|_{\poly}<r_0\},
\qquad
U_{P^\star}\coloneqq \{\gamma_{P(v)}:v\in\cO_{P^\star}\},
\]
and let $\varphi_{P^\star}:U_{P^\star}\to\cO_{P^\star}$ be the
vertex-coordinate map $\gamma_{P(v)}\mapsto v$. Then
\begin{enumerate}
\item[(i)] For every $v\in\cO_{P^\star}$, $P(v)\in\cA_{1/2}$.
\item[(ii)] $U_{P^\star}$ is open in $K_{1/2}$ in the subspace $L^1(\Omega)$-topology.
\item[(iii)] $\varphi_{P^\star}$ is a well-defined homeomorphism from
$U_{P^\star}$ onto $\cO_{P^\star}$.
\end{enumerate}
\end{lemma}

With the uniform $r_0$ from Lemma~\ref{lem:patch-radius}, the chart radius
is the same for every $P^\star\in\cA$, which is what enables the quantitative
patch-count bound below. The choice $r_0\le d_1\beta_0^2/128$
also guarantees that the vertex-label map $\varphi_{P^\star}$ is injective; the proof in Appendix~\ref{sec:chart-M} verifies
injectivity explicitly.

\begin{lemma}[Patch covering with explicit count]\label{lem:logm}
Let $r_0$ be the patch radius from Lemma~\ref{lem:patch-radius}. There exists
a finite collection of centers
$
\{v^i\}_{i=1}^m\subset\mathcal V_{\mathcal A}
$
such that:
\begin{enumerate}
\item[\textup{(i)}] $\{\Bpoly(v^i,r_0)\}_{i=1}^m$ covers
$\mathcal V_{\mathcal A}$ in the poly-norm. Consequently, with
$P_i:=P(v^i)$ and $U_{P_i}$ the chart neighborhood from
Lemma~\ref{lem:chart-M}, $K\subset\bigcup_{i=1}^m U_{P_i}$.
\item[\textup{(ii)}] $m\le(1+2/r_0)^{2n_{\mathrm v}}$, so that
$\log m\le 2n_{\mathrm v}\log(1+2/r_0)\ =\ O(n_{\mathrm v})$.

\item[\textup{(iii)}] For each $i$, the patch $E_i:=\Bpoly(v^i,r_0)$ is
convex, $E_i\subset\cO_{P_i}$, and $P(v)\in\cA_{1/2}$ for every $v\in E_i$.
\end{enumerate}
\end{lemma}

The patches $\{E_i\}_{i=1}^m$ are convex open balls in poly-norm. The
explicit bound $\log m=O(n_{\mathrm v})$ is what ensures that the union-bound
contribution to the metric-entropy estimate (Theorem~\ref{thm:FN-main}) remains
$O(n_{\mathrm v})$.

\subsection{Chartwise regularity of the DtN map}\label{sec:chartwise_DtN}
To transfer low-dimensional structure from vertices to measurements, we need
a regularity statement for the forward map and a Lipschitz constant that is independent of $n_{\mathrm v}$. The Fr\'echet differentiability of
$v\mapsto\Lambda_{\gamma_{P(v)}}$ under affine vertex movements was established
by \citet{beretta2017differentiability}; the Lipschitz bound below
(Theorem~\ref{thm:Linf}) is proved in Appendix~\ref{sec:Linf} via a
distributed-shape-derivative argument that yields an $n_{\mathrm v}$-independent
constant.

Consider
\[
\mathcal{Y} \coloneqq  \mathcal{L}\!\big(H^{1/2}(\partial\Omega),\,H^{-1/2}(\partial\Omega)\big),
\qquad
\|A\|_{\mathcal{Y}} \coloneqq  \sup_{\|f\|_{H^{1/2}(\partial\Omega)}=1}\|Af\|_{H^{-1/2}(\partial\Omega)}\,,
\]
Lemma~\ref{lem:DtN-diff-chart} states the regularity of the forward map in local coordinates.

\begin{lemma}[Chartwise differentiability of the DtN map]
\label{lem:DtN-diff-chart}
Fix $P^\star\in\mathcal A$ and let
$(U_{P^\star},\varphi_{P^\star})$ be the chart from
Lemma~\ref{lem:chart-M}.  Set
$
\mathcal O_{P^\star}:=\varphi_{P^\star}(U_{P^\star})
\subset\mathbb R^{2n_{\mathrm v}}$, and 
$
\widetilde\Lambda_{P^\star}(v):=\Lambda_{\gamma_{P(v)}}.
$
Then $\widetilde\Lambda_{P^\star}$ is Fr\'echet differentiable on
$\mathcal O_{P^\star}$.  For every
$v\in\mathcal O_{P^\star}$ and $V\in\mathbb R^{2n_{\mathrm v}}$, the derivative
$D\widetilde\Lambda_{P^\star}(v)[V]\in\mathcal Y$ is given by the distributed
shape-derivative formula in Lemma~\ref{lem:dist-formula}. 
\end{lemma}

\begin{proof}[Proof of Lemma~\ref{lem:DtN-diff-chart}]
By Lemma~\ref{lem:chart-M}(i), every $v\in\mathcal O_{P^\star}$ satisfies
$P(v)\in\mathcal A_{1/2}$. Thus the whole chart lies in the relaxed geometric
regime where the polygon is simple, strictly convex, and has uniformly
controlled edge lengths, angles, and distance from $\partial\Omega$. The Fr\'echet differentiability of $v \mapsto \Lambda_{\gamma_{P(v)}}$ under
affine vertex movements is established
in \cite[Lemma 4.4, Corollary 4.5 and Remark 2.3]{beretta2017differentiability}.
\end{proof}

The differentiability above gives existence of the chartwise derivative, but
does not directly provide a Lipschitz constant uniform in $n_{\mathrm v}$.

\begin{theorem}[Uniform chartwise Lipschitz bound for the DtN map]\label{thm:Linf}
Let $P^\star\in\mathcal A$, let $(U_{P^\star},\varphi_{P^\star})$ be a chart from
Lemma~\ref{lem:chart-M}, and let $\widetilde\Lambda_{P^\star}:\cO_{P^\star}\to\cY$
be as in Lemma~\ref{lem:DtN-diff-chart}. There exists a constant
\[
L_\infty=L_\infty(\Omega,\kappa,d_0,d_1,\beta_0)<\infty,
\]
independent of $n_{\mathrm v}$, such that:
\begin{enumerate}
\item[\textup{(i)}] (Derivative bound) For every $v\in\cO_{P^\star}$ and every
$V\in\R^{2n_{\mathrm v}}$,
\[
\|D\widetilde\Lambda_{P^\star}(v)[V]\|_{\cY}\le L_\infty\cdot\|V\|_{\poly}.
\]
\item[\textup{(ii)}] (Map Lipschitz on a convex patch) For any convex
$E\subset\cO_{P^\star}$ and any $v,w\in E$,
\[
\|\widetilde\Lambda_{P^\star}(v)-\widetilde\Lambda_{P^\star}(w)\|_{\cY}
\le L_\infty\cdot\|v-w\|_{\poly}.
\]
In particular, this applies to $E=E_i=\Bpoly(v^i,r_0)$ from
Lemma~\ref{lem:logm}.
\end{enumerate}
The explicit form is $L_\infty=4\gamma_{\max}C_1 C_{\mathrm{en}}^2$, where
$C_1=24\sqrt 2/d_0+\sqrt 2\,C_\partial(d_1,\beta_0)$,
$C_\partial=\max(4\sqrt{1+K_*^2}/d_1,\,4/r_*)$, with $r_*,K_*$ from
Lemma~\ref{lem:A4_from_A23_conv}, and $C_{\mathrm{en}}=C_{\mathrm{en}}(\Omega,\gamma_{\min},\gamma_{\max})$
the elliptic-energy constant.
\end{theorem}

The proof is given in Appendix~\ref{sec:Linf} and combines a Hadamard flow
construction with the distributed-shape-derivative formula and a global
$L^2$ energy estimate. Its key feature is to express the derivative as
a volume integral on $\Omega$, so that corner singularities of $\nabla u$
enter only through the square-integrable global energy and do not pick up
an $n_{\mathrm v}$-dependent prefactor.

\subsection{Chartwise Lipschitz for the DtN measurements}\label{sec:chartwise_lipschitz_discretized_DtN}

We next specify the finite-dimensional representation of DtN measurements. Starting from the DtN operator $\Lambda_\gamma\in \mathcal{Y}$, we will (i) restrict it to finite-dimensional subspaces through $R_N$ and obtain the DtN matrix as in \Eqref{eqn:R_N}, 
and (ii) vectorize and normalize the resulting DtN matrix to obtain a vector in Euclidean space, which is aligned to the training object in \Eqref{eq:cond-score-matching}. 
Combining all these yields a discretized measurement map
$F_N:K_{1/2}\to\R^{d_N}$ in \Eqref{eqn:F_N}, which will be restricted to the
final prior $K\subset K_{1/2}$ in the entropy and diffusion estimates.

Let $V_N^1\subset H^{1/2}(\partial\Omega)$ and $V_N^2\subset H^{-1/2}(\partial\Omega)$ be finite-dimensional Hilbert subspaces with inherited inner products, and denote
\[
n_1\coloneqq \dim V_N^1,\qquad n_2\coloneqq \dim V_N^2,\qquad N\coloneqq \min\{n_1,n_2\},\qquad d_N\coloneqq n_1n_2,\qquad \mathcal{Y}_N\coloneqq \mathcal{L}(V_N^1,V_N^2)\,.
\]

For notational simplicity, in the sequel we focus on the square case $n_1=n_2=N$, which is also consistent with our numerical setup and in this case $d_N=N^2$.
Let 
$P_N^2:H^{-1/2}(\partial\Omega)\to V_N^2$
be the orthogonal projection, and $\iota_N^1:V_N^1\hookrightarrow H^{1/2}(\partial\Omega)$ be the natural
inclusion, and define
\begin{equation}\label{eqn:R_N}
R_N:\mathcal{Y}\to \mathcal{Y}_N,\qquad R_N(A)\coloneqq P_N^2A\iota_N^1\,.
\end{equation}
Equip $\mathcal{Y}_N$ with the Hilbert--Schmidt norm $\|\cdot\|_{\mathrm{HS}}$, which is independent of the choice of basis of $V_N^1$ and $V_N^2$.
Fix orthonormal bases $\{\varphi_j\}_{j=1}^{n_1}$ of $V_N^1$ and $\{\psi_i\}_{i=1}^{n_2}$ of $V_N^2$ and define the vectorization
$
\mathrm{vec}:\mathcal{Y}_N\to\R^{d_N},\
\mathrm{vec}(A)\coloneqq \Big(\langle A\varphi_j,\psi_i\rangle_{V_N^2}\Big)_{\substack{1\le i\le n_2\\1\le j\le n_1}}\,,
$
thus $\|\mathrm{vec}(A)\|_2=\|A\|_{\mathrm{HS}}$ for all $A\in \mathcal{Y}_N$. 
\begin{lemma}\label{lem:RN-stable}
For any $A\in \mathcal{Y}$,
$
\|R_N(A)\|_{\mathrm{op}(V_N^1\to V_N^2)}\ \le\ \|A\|_\mathcal{Y},
\ 
\|R_N(A)\|_{\mathrm{HS}}\ \le\ \sqrt{N}\,\|A\|_\mathcal{Y}\,.
$
\end{lemma}

\begin{proof}[Proof of Lemma~\ref{lem:RN-stable}]
Since $P_N^2$ is an orthogonal projection and $V_N^1$ carries the inherited
norm, we have $\|\iota_N^1\|=1$ and
$
\|R_N(A)\|_{\mathrm{op}(V_N^1\to V_N^2)}
=\|P_N^2A\iota_N^1\|_{\mathrm{op}(V_N^1\to V_N^2)}
\le \|P_N^2\|\,\|A\|_\mathcal{Y}\,\|\iota_N^1\|
= \|A\|_\mathcal{Y}.
$
For the Hilbert--Schmidt norm, for any linear map $B:V_N^1\to V_N^2$,
$\|B\|_{\mathrm{HS}}\le \sqrt{\mathrm{rank}(B)}\,\|B\|_{\mathrm{op}}$ and
$\mathrm{rank}(R_N(A))\le N$, so
$
\|R_N(A)\|_{\mathrm{HS}}
\le \sqrt{N}\,\|R_N(A)\|_{\mathrm{op}}
\le \sqrt{N}\,\|A\|_\mathcal{Y}.
$
\end{proof}
Finally we define the discretized DtN measurement scaled by $1/\sqrt{N}$,
\begin{equation}\label{eqn:F_N}
F_N:K_{1/2}\to\R^{d_N},\qquad
F_N(\gamma)\coloneqq \frac{1}{\sqrt{N}}\mathrm{vec}\big(R_N(\Lambda_\gamma)\big)\,,
\end{equation}
and establish the chartwise Lipschitz for $F_N$ as follows. The $1 / \sqrt{N}$ normalization rescales distances in $\R^{d_N}$, and Lipschitz constants rescale accordingly. The normalization absorbs the factor $\sqrt N$ in
Lemma~\ref{lem:RN-stable}, so the Lipschitz and entropy bounds below
have constants independent of $N$.

We now combine the chartwise regularity of the DtN mapping and Lemma~\ref{lem:RN-stable} to conclude that, on each parameter patch $E_i$, the chartwise discretized forward map $G_{N,i}$ is Lipschitz from the poly-norm parameter patch to the Euclidean measurement space.

\begin{lemma}[Chartwise Lipschitz for measurements]\label{lem:GN-Lip}
Let $L_\infty$ be the uniform Lipschitz constant from Theorem~\ref{thm:Linf}.
For each $i\in\{1,\dots,m\}$, define $G_{N,i}:E_i\to\R^{d_N}$ by
$G_{N,i}(v):=F_N(\gamma_{P(v)})$. Then $G_{N,i}$ is $L_\infty$-Lipschitz
from $(E_i,\|\cdot\|_{\poly})$ to $(\R^{d_N},\|\cdot\|_2)$:
\[
\|G_{N,i}(v)-G_{N,i}(w)\|_2 \le L_\infty\cdot\|v-w\|_{\poly},
\quad \forall v,w\in E_i.
\]
The constant $L_\infty$ is independent of $i$, of $n_{\mathrm v}$, and of $N$.
\end{lemma}

\begin{proof}[Proof of Lemma~\ref{lem:GN-Lip}]
Let $\widetilde\Lambda_i:=\widetilde\Lambda_{P_i}$.
By Lemma~\ref{lem:RN-stable} and the $1/\sqrt N$ normalization in $F_N$,
\[
\|G_{N,i}(v)-G_{N,i}(w)\|_2
= \frac{1}{\sqrt N}\bigl\|R_N(\widetilde\Lambda_i(v)-\widetilde\Lambda_i(w))\bigr\|_{\mathrm{HS}}
\le \|\widetilde\Lambda_i(v)-\widetilde\Lambda_i(w)\|_{\cY}.
\]
Since $E_i=\Bpoly(v^i,r_0)\subset\cO_{P_i}$ is convex
(Lemma~\ref{lem:logm}(iii)), Theorem~\ref{thm:Linf}(ii) gives
$\|\widetilde\Lambda_i(v)-\widetilde\Lambda_i(w)\|_{\cY}\le L_\infty\|v-w\|_{\poly}$.
\end{proof}

\subsection{\texorpdfstring{Metric entropy and bounded support of $F_N(K)$}{Metric entropy and bounded support of FN(K)}}\label{sec:entropy-bdd-FN}
In our setting, for $Z$-a.e.\ $z$ the conditional support of $X_0$ is contained in the closure of $F_N(K)$; hence it suffices to bound covering numbers of $F_N(K)$. The diffusion convergence theorem we invoke later depends on the intrinsic complexity of the measurement
distribution, quantified via covering numbers of its support at a prescribed resolution
$\varepsilon_0 = T^{-c_{\varepsilon_0}}$, where $T\in\mathbb{N}$ is the number of diffusion steps and
$c_{\varepsilon_0}>0$ is fixed.
We therefore aim to bound the metric entropy of the measurement set $F_N(K)$ with respect to the
Euclidean norm $\|\cdot\|_2$.

\begin{definition}[Covering number and metric entropy]\label{def:covering}
Let $(\mathcal{X},d)$ be a metric space and $S\subset\mathcal{X}$.
For $\varepsilon>0$, the covering number $\mathcal{N}_\varepsilon(S;d)$ is the smallest $n$ such that
$S\subset\bigcup_{j=1}^n B(x_j,\varepsilon)$ for some $x_1,\dots,x_n\in\mathcal{X}$, where
$B(x,\varepsilon)\coloneqq \{y:\ d(y,x)\le\varepsilon\}$.
The metric entropy is defined as $\log \mathcal{N}_\varepsilon(S;d)$.
\end{definition}

Lemma~\ref{lem:cover-lip} shows that an \(L\)-Lipschitz map does not increase covering numbers after rescaling
the covering radius.

\begin{lemma}[Coverings under Lipschitz maps]\label{lem:cover-lip}
Let $(\mathcal X,d_{\mathcal X})$ and $(\mathcal Z,d_{\mathcal{Z}})$ be metric spaces.
If $f:\mathcal X\to\mathcal Z$ is $L$-Lipschitz with $L>0$, then for any $S\subset\mathcal X$ and $\varepsilon>0$,
\[
\mathcal{N}_\varepsilon\!\bigl(f(S);d_{\mathcal{Z}}\bigr)\ \le\ \mathcal{N}_{\varepsilon/L}\!\bigl(S;d_{\mathcal X}\bigr)\,.
\]
\end{lemma}

\begin{proof}[Proof of Lemma~\ref{lem:cover-lip}]
Let $\{x_j\}_{j=1}^n$ be an $\varepsilon/L$-net of $S$ in $d_{\mathcal X}$ with
$n=\mathcal{N}_{\varepsilon/L}\!\bigl(S;d_{\mathcal X}\bigr)$, so
$S\subset\bigcup_{j=1}^n B_{\mathcal X}(x_j,\varepsilon/L)$.
For any $z=f(s)\in f(S)$ choose $j$ with $d_{\mathcal X}(s,x_j)\le\varepsilon/L$.
Then $d_{\mathcal{Z}}(z,f(x_j))\le L d_{\mathcal X}(s,x_j)\le\varepsilon$.
This implies that $\{f(x_j)\}$ is an $\varepsilon$-net of $f(S)$, and as a consequence,
the covering number of $f(S)$ is no greater than $n$.
\end{proof}

\begin{lemma}[Volumetric covering in poly-norm]\label{lem:volumetric}
There is an absolute constant $C_{\mathrm{vol}}>0$ such that, for every
$v^\ast\in\R^{2n_{\mathrm v}}$, every $R>0$, every $\delta>0$, and every
$E\subset\Bpoly(v^\ast,R)$,
\[
\mathcal{N}_\delta(E;\|\cdot\|_{\poly})\ \le\ \max\Bigl\{1,\ \bigl(C_{\mathrm{vol}}\,R/\delta\bigr)^{2n_{\mathrm v}}\Bigr\}.
\]
One may take $C_{\mathrm{vol}}=3$.
\end{lemma}

We now bound (i) the metric entropy of $F_N(K)$ and (ii) a uniform radius bound for its support.
The proof covers $K$ by finitely many patches, uses Lipschitz bounds for the chartwise measurement maps,
and applies a volumetric covering estimate in $\R^{2n_{\mathrm v}}$.

We first record a standard elliptic energy estimate that is used both here
and in Appendix~\ref{sec:Linf}.

\begin{lemma}[Elliptic energy estimate]\label{lem:energy}
There exists $C_{\mathrm{en}}=C_{\mathrm{en}}(\Omega,\gamma_{\min},\gamma_{\max})$
such that for every measurable $\gamma:\Omega\to[\gamma_{\min},\gamma_{\max}]$
and every $f\in H^{1/2}(\partial\Omega)$, the unique weak solution
$u\in H^1(\Omega)$ of $\nabla\cdot(\gamma\nabla u)=0$ with $u|_{\partial\Omega}=f$
satisfies
\[
\|\nabla u\|_{L^2(\Omega)}\le C_{\mathrm{en}}\,\|f\|_{H^{1/2}(\partial\Omega)}.
\]
\end{lemma}

\begin{proof}
Let $F\in H^1(\Omega)$ be a bounded linear extension of $f$ with
$\|F\|_{H^1(\Omega)}\le C_\Omega\|f\|_{H^{1/2}(\partial\Omega)}$ (which exists
by the right inverse of the trace map). Write $u=F+w$ with $w\in H_0^1(\Omega)$.
Then $w$ solves $\int_\Omega\gamma\,\nabla w\cdot\nabla\psi\,\mathrm{d}x
=-\int_\Omega\gamma\,\nabla F\cdot\nabla\psi\,\mathrm{d}x$ for all
$\psi\in H_0^1(\Omega)$. Testing with $\psi=w$:
\[
\gamma_{\min}\|\nabla w\|_{L^2}^2
\le\int_\Omega\gamma\,|\nabla w|^2\,\mathrm{d}x
=-\int_\Omega\gamma\,\nabla F\cdot\nabla w\,\mathrm{d}x
\le\gamma_{\max}\|\nabla F\|_{L^2}\|\nabla w\|_{L^2},
\]
so $\|\nabla w\|_{L^2}\le(\gamma_{\max}/\gamma_{\min})\|\nabla F\|_{L^2}$. Triangle inequality and the trace bound give
$\|\nabla u\|_{L^2}\le(1+\gamma_{\max}/\gamma_{\min})C_\Omega\|f\|_{H^{1/2}}$.
Set $C_{\mathrm{en}}:=(1+\gamma_{\max}/\gamma_{\min})C_\Omega$.
\end{proof}
\begin{lemma}[Uniform boundedness of DtN operator on $K$]\label{lem:DtN-bdd}
There exists $r_K=r_K(\Omega,\kappa)<\infty$,
such that $\sup_{\gamma\in K}\|\Lambda_\gamma\|_{\cY}\le r_K$.
Explicitly, $r_K\le\gamma_{\max}\,C_{\mathrm{en}}^2$, where
$\gamma_{\max}=\max(1,\kappa)$ and $C_{\mathrm{en}}$ is from Lemma~\ref{lem:energy}.
\end{lemma}

\begin{proof}[Proof of Lemma~\ref{lem:DtN-bdd}]
By the symmetric form
$\langle\Lambda_\gamma f,g\rangle_{H^{-1/2},H^{1/2}}
=\int_\Omega\gamma\,\nabla u^f\cdot\nabla u^g\,\mathrm{d}x$ where $u^f\in H^1(\Omega)$
is the unique weak solution with trace $f$, Cauchy--Schwarz and
Lemma~\ref{lem:energy} give
\[
\bigl|\langle\Lambda_\gamma f,g\rangle_{H^{-1/2},H^{1/2}}\bigr|
\le\gamma_{\max}\,\|\nabla u^f\|_{L^2}\,\|\nabla u^g\|_{L^2}
\le\gamma_{\max}\,C_{\mathrm{en}}^2\,\|f\|_{H^{1/2}}\,\|g\|_{H^{1/2}}.
\]
Taking sup over $\|f\|_{H^{1/2}}=\|g\|_{H^{1/2}}=1$ gives the bound.
\end{proof}

By Lemma~\ref{lem:logm}(iii), every patch is $E_i=\Bpoly(v^i,r_0)$.

\begin{theorem}[Metric entropy and bounded support of discretized DtN measurements]\label{thm:FN-main}
Let $r_0$ be the patch radius (Lemma~\ref{lem:patch-radius}), $m$ the patch
count (Lemma~\ref{lem:logm}), $L_\infty$ the uniform Lipschitz constant
(Theorem~\ref{thm:Linf}), $C_{\mathrm{vol}}=3$ (Lemma~\ref{lem:volumetric}),
and $r_K$ the support bound (Lemma~\ref{lem:DtN-bdd}). Fix $c_{\varepsilon_0}>0$
and set $\varepsilon_0:=T^{-c_{\varepsilon_0}}$. Then:
\begin{enumerate}
\item[\textup{(i)}] (Entropy bound) For all $T\ge 2$,
\begin{equation}
\label{eq:entropy-FNK}
\log\mathcal{N}_{\varepsilon_0}\bigl(F_N(K);\|\cdot\|_2\bigr)\le C_F\log T,
\end{equation}
where one may take $C_F=2n_{\mathrm v}\,C_{\mathrm{cover}}$ with
\begin{equation}
\label{eq:C_cover}
C_{\mathrm{cover}}:=c_{\varepsilon_0}
+\frac{\log(1+2/r_0)+\log_+(C_{\mathrm{vol}}\,r_0\,L_\infty)}{\log 2}.
\end{equation}
\item[\textup{(ii)}] (Bounded support) $\sup_{\gamma\in K}\|F_N(\gamma)\|_2\le r_K$.
\end{enumerate}
\end{theorem}

\begin{proof}
We first cover $F_N(K)$ by chartwise images. By Lemma~\ref{lem:logm},
$\mathcal V_{\mathcal A}\subset\bigcup_{i=1}^m\Bpoly(v^i,r_0)
=\bigcup_{i=1}^m E_i$. The chart homeomorphism $\varphi_{P_i}$
(Lemma~\ref{lem:chart-M}) identifies $E_i$ with $U_{P_i}\subset K_{1/2}$, and
$K\subset\bigcup_i U_{P_i}$. Hence
\[
F_N(K)\subset\bigcup_{i=1}^m G_{N,i}(E_i),
\]
where $G_{N,i}:E_i\to\R^{d_N}$ is as in Lemma~\ref{lem:GN-Lip}.

Now we can compute the entropy of each patch image. By Lemma~\ref{lem:GN-Lip},
$G_{N,i}$ is $L_\infty$-Lipschitz from $(E_i,\|\cdot\|_{\poly})$ to
$(\R^{d_N},\|\cdot\|_2)$. Hence by Lemma~\ref{lem:cover-lip},
\[
\mathcal{N}_\varepsilon(G_{N,i}(E_i);\|\cdot\|_2)
\le\mathcal{N}_{\varepsilon/L_\infty}(E_i;\|\cdot\|_{\poly}).
\]
Since $E_i=\Bpoly(v^i,r_0)$, Lemma~\ref{lem:volumetric} gives
\[
\mathcal{N}_{\varepsilon/L_\infty}(E_i;\|\cdot\|_{\poly})
\le\max\Bigl\{1,\bigl(C_{\mathrm{vol}}r_0 L_\infty/\varepsilon\bigr)^{2n_{\mathrm v}}\Bigr\}.
\]

Taking the union over patches gives
\[
\mathcal{N}_\varepsilon(F_N(K);\|\cdot\|_2)\le m\cdot
\max\Bigl\{1,(C_{\mathrm{vol}}r_0 L_\infty/\varepsilon)^{2n_{\mathrm v}}\Bigr\}.
\]
Taking $\varepsilon=\varepsilon_0=T^{-c_{\varepsilon_0}}$ gives:
\[
\log\mathcal{N}_{\varepsilon_0}(F_N(K))
\le\log m + 2n_{\mathrm v}\log_+(C_{\mathrm{vol}}r_0 L_\infty) + 2n_{\mathrm v}\,c_{\varepsilon_0}\log T.
\]
Inserting $\log m\le 2n_{\mathrm v}\log(1+2/r_0)$ (Lemma~\ref{lem:logm}(ii))
and using $\log T\ge\log 2$ for $T\ge 2$ to absorb the constant term:
\[
\log\mathcal{N}_{\varepsilon_0}(F_N(K))\le C_F\log T,
\]
with $C_F=2n_{\mathrm v}C_{\mathrm{cover}}$ as in~\Eqref{eq:C_cover}. Every term in $C_{\mathrm{cover}}$
depends only on $(\Omega,\kappa,d_0,d_1,\beta_0,c_{\varepsilon_0})$;
the entire $n_{\mathrm v}$-dependence of $C_F$ is the linear prefactor
$2n_{\mathrm v}$.

Finally, using $\|F_N(\gamma)\|_2=\frac{1}{\sqrt N}\|R_N(\Lambda_\gamma)\|_{\mathrm{HS}}$,
Lemma~\ref{lem:RN-stable}, and Lemma~\ref{lem:DtN-bdd},
$\|F_N(\gamma)\|_2\le\|\Lambda_\gamma\|_{\cY}\le r_K$.
\end{proof}

\subsection{Convergence guarantee for diffusion completion}\label{sec:ddpm_FNK}
Consider the full discretized DtN measurement vector
$
X_0 \coloneqq  F_N(\gamma)\in\R^{d_N},
$
and the corresponding observed data
$
Z \coloneqq  (M\odot X_0,M)
$, 
$
M\in\{0,1\}^{d_N}.
$
Note that $F_N$ includes the normalization $1/\sqrt{N}$ in \Eqref{eqn:F_N}, and $Z$ here is the vectorized counterpart of \Eqref{eq: lambda_obs}.
In this subsection we adopt the same noise schedule as \citet[Eq.~(9)]{liang2025low}.
We write $\beta_t=1-\alpha_t$ for the per-step noise variance and choose
$\beta_1=T^{-c_0}$, while for $1\le t<T$ we set
$
\beta_{t+1}=\frac{c_1\log T}{T}\min\!\left\{\beta_1\Bigl(1+\frac{c_1\log T}{T}\Bigr)^t,\,1\right\},
$
for sufficiently large fixed constants $c_0,c_1>0$
chosen as specified in Appendix~\ref{app:ddpm-parameters}.  $\beta_t$ grows exponentially at the beginning, and then stays flat after reaching the order of $\log T/T$. We recall the unconditional result of \citet[Thm.~2]{liang2025low} and will apply it to each conditional distribution $p_{X_{t}\mid Z=z}$ by treating $z$ as a fixed parameter. 

\begin{lemma}[Unconditional DDPM convergence under low intrinsic complexity
{\rm\citep[Thm.~2]{liang2025low}}]\label{lem:liang-ddpm}
Let $X_0\sim p_{\mathrm{data}}$ be a distribution on $\R^d$ satisfying Assumptions~1--3 of
\citet{liang2025low} (intrinsic dimension $k$ quantified by metric entropy at $\varepsilon_0=T^{-c_{\varepsilon_0}}$,
bounded support with radius $R=T^{c_R}$, and averaged $\ell_2$ score error $\varepsilon_{\mathrm{score}}$).
Assume additionally that $k\ge\max\{1,\log d\}$, and fix the
parameters as specified in Appendix~\ref{app:ddpm-parameters}.
There exists a threshold $T_*=T_*(d,k)$ such that, for every
$T\ge T_*$, the DDPM sampler in Definition~\ref{def:ddpm} satisfies
\begin{equation}\label{eq:TV_bound_uncond}
\mathrm{TV}\bigl(p_{X_1},\,p_{Y_1}\bigr)
\ \lesssim\ \frac{k\,\log^3 T}{T}\ +\ \varepsilon_{\mathrm{score}}\sqrt{\log T}\,.
\end{equation}
\end{lemma}

In our polygonal EIT setting, the natural parameter dimension is $k=2n_{\mathrm v}$. The metric entropy bound for intrinsic dimension and the boundedness of support have been verified in
Theorem~\ref{thm:FN-main}, which depends on the assumptions for admissible polygons. For fixed $z$, recall that
$
s_t^\star(x,z):=\nabla_x\log p_{X_t\mid Z=z}(x)
$
denotes the conditional score of the forward noised law, and the forward noising
process is
$
X_t=\sqrt{\bar\alpha_t}X_0+\sqrt{1-\bar\alpha_t}W_t,
$
where $W_t\sim N(0,I_{d_N})$ is independent of $(X_0,Z)$. Thus, conditional on
$Z=z$, the process is an ordinary unconditional forward diffusion initialized
from $p_{X_0\mid Z=z}$. We assume the conditional score estimation error as follows.
\begin{assumption}[$\ell_2$ conditional score error]\label{ass:cond_score}
There exists $\varepsilon_{\mathrm{score}}\ge 0$ such that for $Z$-almost every $z$,
\[
\frac{1}{T}\sum_{t=1}^T
\mathbb E\!\left[\ \big\|s_{\theta}(X_t,z,t)-s_t^{\star}(X_t,z)\big\|_2^2\ \mid Z=z \right]
\ \le\ \varepsilon_{\mathrm{score}}^2\,.
\]
\end{assumption}

\begin{remark}\label{rem:score-assump} In this paper we follow the common approach in the diffusion convergence literature and treat
$\varepsilon_{\mathrm{score}}$ as an exogenous accuracy parameter, as in
\citep{chen2023improved,benton2023nearly,liang2025low}. The analysis concerns the sampling error and does not bound the error
of learning the conditional score from finite data. 
\end{remark}
With conditional score error estimation in Assumption~\ref{ass:cond_score}, we
verify the assumptions of \citet[Thm.~2]{liang2025low} for each conditional law $p_{X_0\mid Z=z}$ and obtain
the following guarantee.

\begin{theorem}[Conditional DDPM completion error in TV]\label{thm:cond-ddpm}
Consider the set of full DtN measurements from all admissible conductivities $F_N(K)\subset\R^{d_N}$, and let
$\varepsilon_0\coloneqq T^{-c_{\varepsilon_0}}$ with $c_{\varepsilon_0}>0$ as in \citep[Assumption~1]{liang2025low}. Assume furthermore that Assumption~\ref{ass:cond_score} holds and $2n_{\mathrm{v}}\geq \log d_N$. Fix the geometric a priori data and the parameters specified in Appendix~\ref{app:ddpm-parameters}. There exists a threshold
$T_0=T_0(N,n_{\mathrm v})$, such that, for every
$T\ge T_0$ and for $Z$-almost every $z$, the conditional DDPM
sampler produces $Y_1^z$ whose law satisfies
\begin{equation}\label{eq:TV_bound_cond}
\mathrm{TV}\bigl(p_{X_1\mid Z=z},\,p_{Y^z_1}\bigr)
\ \lesssim\ \frac{2n_{\mathrm v}\,\log^3 T}{T}\ +\ \varepsilon_{\mathrm{score}}\sqrt{\log T}\,.
\end{equation}
\end{theorem}
\begin{proof}
Fix $z$ such that Assumption~\ref{ass:cond_score} holds. Since $X_0\in F_N(K)$ almost surely,
$
\operatorname{supp}(p_{X_0\mid Z=z})\subset \overline{F_N(K)}.$
Since the covering balls in Definition~\ref{def:covering} are closed, every
finite $\epsilon$-cover of $F_N(K)$ also covers $\overline{F_N(K)}$. The same
closure argument preserves the radius bound. Consequently, for any $\epsilon>0$,
$
\mathcal{N}_\epsilon\bigl(\operatorname{supp}(p_{X_0\mid Z=z});\|\cdot\|_2\bigr)
\ \le\ \mathcal{N}_\epsilon\bigl(F_N(K);\|\cdot\|_2\bigr)$, and 
$
\sup_{x\in\operatorname{supp}(p_{X_0\mid Z=z})}\|x\|_2\ \le\ r_K.
$
By the metric entropy bound from Theorem~\ref{thm:FN-main}, taking $\epsilon=\varepsilon_0$ yields
$
\log \mathcal{N}_{\varepsilon_0}\bigl(\mathrm{supp}(p_{X_0\mid Z=z})\bigr)\ \le\ C_F\log T
$. 

Therefore Assumption~1 of \citet{liang2025low} holds for $p_{X_0\mid Z=z}$
with intrinsic dimension $k=2n_{\mathrm v}$ and constant $C_{\mathrm{cover}}$ is independent of $n_{\mathrm v}$ (by Theorem~\ref{thm:FN-main}(i)). The radius bound from Theorem~\ref{thm:FN-main} verifies Assumption~2 of \citet{liang2025low}, since $r_K$ does not depend on $T$ and we can choose $c_R\ge \max\{1,\log_2 r_K\}$ such that $r_K\le T^{c_R}$ for $T\ge 2$. Assumption~\ref{ass:cond_score} is the conditional analogue of Assumption~3 of \citet{liang2025low}.
For each fixed $z$, the conditional DDPM is an unconditional DDPM on $\R^{d_N}$ with data law
$p_{X_0\mid Z=z}$. By Appendix~\ref{app:ddpm-parameters},
Lemma~\ref{lem:liang-ddpm} applies uniformly in $z$
for $T\ge T_0$, which yields \Eqref{eq:TV_bound_cond}.
\end{proof}

\begin{remark} 
\label{rmk: t1} Theorem~\ref{thm:cond-ddpm} concerns the one-step-noised conditional laws
$p_{X_1\mid Z=z}$ and $p_{Y_1^z}$. It does not provide a
total-variation guarantee for the clean conditional law
$p_{X_0\mid Z=z}$ or for the clean output $Y_0^z$ returned by
Algorithm~\ref{alg:complete-then-reconstruct}. For data distributions supported on low-dimensional sets in $\R^d$,
a TV guarantee directly for $p_{X_0\mid Z=z}$ is generally impossible against any absolutely continuous
sampler output, due to mutual singularity. The TV guarantee above compares $p_{X_1\mid Z=z}$ and $p_{Y^z_1}$, consistent with Theorem~2 of \citet{liang2025low}.
\end{remark}

The ambient dimension of the measurement vector is $d_N=N^2$, whereas the admissible family $F_N(K)$ is
parameterized by only $2n_{\mathrm v}$ real parameters.
This low-dimensional structure is reflected in the metric-entropy bound \Eqref{eq:entropy-FNK}:
after passing to local coordinates, the forward map is Lipschitz locally, so the covering numbers of
$F_N(K)$ are controlled by volumetric coverings in $\R^{2n_{\mathrm v}}$, yielding the power $2n_{\mathrm v}$
in the covering-number bound.

\section{Numerical Experiments}\label{sec:numeric}

In this section, we evaluate the proposed method against baseline approaches on two tasks: (i) completing full DtN measurements from randomly sampled sparse measurements, and (ii) reconstructing EIT conductivities from sparse measurements. The details of the experimental setup and model training are provided in Sections~\ref{sec:data_gen} and~\ref{sec:Implement_details}, while the testing results are reported in Sections \ref{sec:result_DC} and~\ref{sec:result_IP}.

\subsection{Data generation}\label{sec:data_gen}
\paragraph{Discretizations.} We consider the unit disk as the computational domain $\Omega$, which is discretized using a conforming triangular mesh (see Figure~\ref{fig:conductivity_discretizations}). 
 This results in $N_\mathcal{I} = 1324$ interior nodal points, $N_\mathcal{B} = 128$ boundary nodal points, and 2774 triangular elements. The elliptic PDE \Eqref{eq:forward} is then solved using the finite element method (FEM) with $P_1$ finite elements. Consequently, the DtN measurement matrix has size $128 \times 128$, while the conductivity $\gamma$ is discretized into a vector of length 2774 at the centroids of all elements. To facilitate the use of conductivity data in deep learning models, $\gamma$ is zero-padded and extended to a function on the square $[-1,1]\times[-1,1]$, and then discretized over a $128 \times 128$ grid via linear interpolation, as shown in Figure~\ref{fig:conductivity_discretizations}.  

\paragraph{Conductivity distributions.}
We consider two different distributions of conductivity throughout this section as follows.
\begin{itemize}
\item \textbf{Disks:} 
This dataset consists of conductivities over a unit disk $\Omega$ with unit background value, together with constant values over a few circular inclusions. In particular, the conductivities are defined as the following.
\[
\gamma(x) \;=\; 1 \;+\; \sum_{i=1}^{n} \bigl(c_i-1\bigr)\,\mathbf{1}_{\{\|x-x_i\|_2 \le r_i\}}(x), \qquad x\in\Omega\,,
\]
where \(n\in\{2,3,4,5\}\), and the inclusion centers \(x_i\in\Omega\), radii \(r_i\), and contrasts \(c_i\) vary across samples.
For each fixed \(n\), we independently draw radii $r_i \sim \mathcal{U}[0.2,\,0.4]$, contrasts $c_i \sim \mathcal{U}[2,\,8]$, and centers $x_i$ uniformly over $\Omega$. We use rejection sampling to ensure the inclusions are fully contained within $\Omega$ with $\dist(x_i,\partial\Omega)> r_i$ for all $i$ and mutually disjoint with 
$
\|x_i-x_j\|_2 > r_i+r_j,
$
for all  $i\neq j$. 
For each \(n\in\{2,3,4,5\}\), \(500\) training samples and \(100\) test samples are generated, yielding a total of \(2000\) training and \(400\) test samples.

\item \textbf{Shepp--Logan Phantom:} This dataset is based on \citet{shepp1974fourier}, a standard benchmark in medical imaging that consists of piecewise constant functions that model human skulls, ventricles and tumors.
Our code follows the setup in \citet{Chung2020}, which provides a randomized construction of the Shepp--Logan phantom.
\end{itemize}
Examples of conductivities in each dataset can be found in the leftmost columns of Figure~\ref{fig:disk_deepsolver_noiseless} and Figure~\ref{fig:sl_deepsolver_noiseless} respectively.

\begin{figure}[htbp]
    \centering    \includegraphics[width=0.45\linewidth]{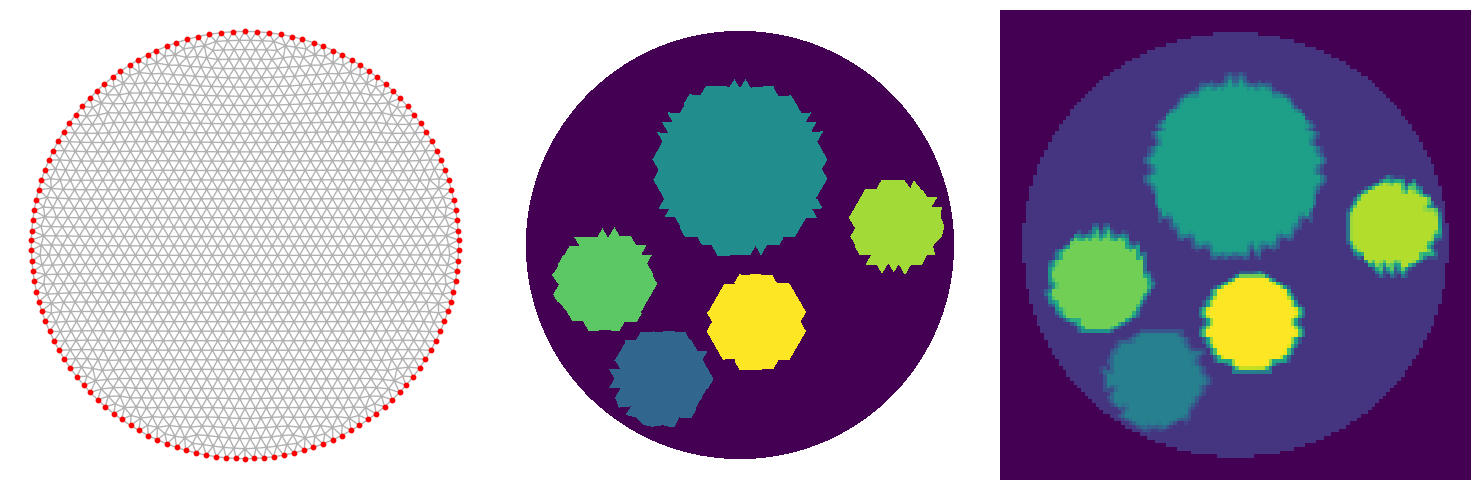}
\caption{Discretization of conductivity $\gamma$. \textbf{Left:} Triangular mesh over the unit disk where red boundary points mark candidate sensor locations. 
\textbf{Middle:} A conductivity over the triangular mesh from the Disks distribution. 
\textbf{Right:} The same conductivity converted to a uniform $128\times128$ Cartesian grid on $[-1,1]\times[-1,1]$ with zero padding.
}
\label{fig:conductivity_discretizations}
\end{figure}

\paragraph{Normalization of the DtN measurement.} 
The DtN measurement $\mathbf{\Lambda}_\gamma$ exhibits strong heterogeneity between diagonal and off-diagonal entries, as the off-diagonal entries are fast-decaying for any conductivity $\gamma \in L^\infty(\Omega)$~\citep{bebendorf2003existence,bechtel2024off}. 
The diagonal and off-diagonal entries of the DtN matrices have
different scales, which can obscure variations across conductivity
samples in the raw data. We normalize each DtN measurement by that of the background (unit) conductivity.
The normalized DtN measurement is defined as the elementwise ratio $\bar{\mathbf{\Lambda}}_\gamma\coloneqq \mathbf{\Lambda}_\gamma \oslash \mathbf{\Lambda}_{\textbf{1}}$. 
All models are trained and evaluated on the normalized DtN measurement \(\bar{\mathbf{\Lambda}}_\gamma\). As illustrated in Figure~\ref{fig:DtN_normalization}, raw DtN measurements are visually similar across distinct conductivities, whereas their normalized version reveals discriminative patterns.

\begin{figure}[htbp]
    \centering
    \includegraphics[width=0.45\linewidth]{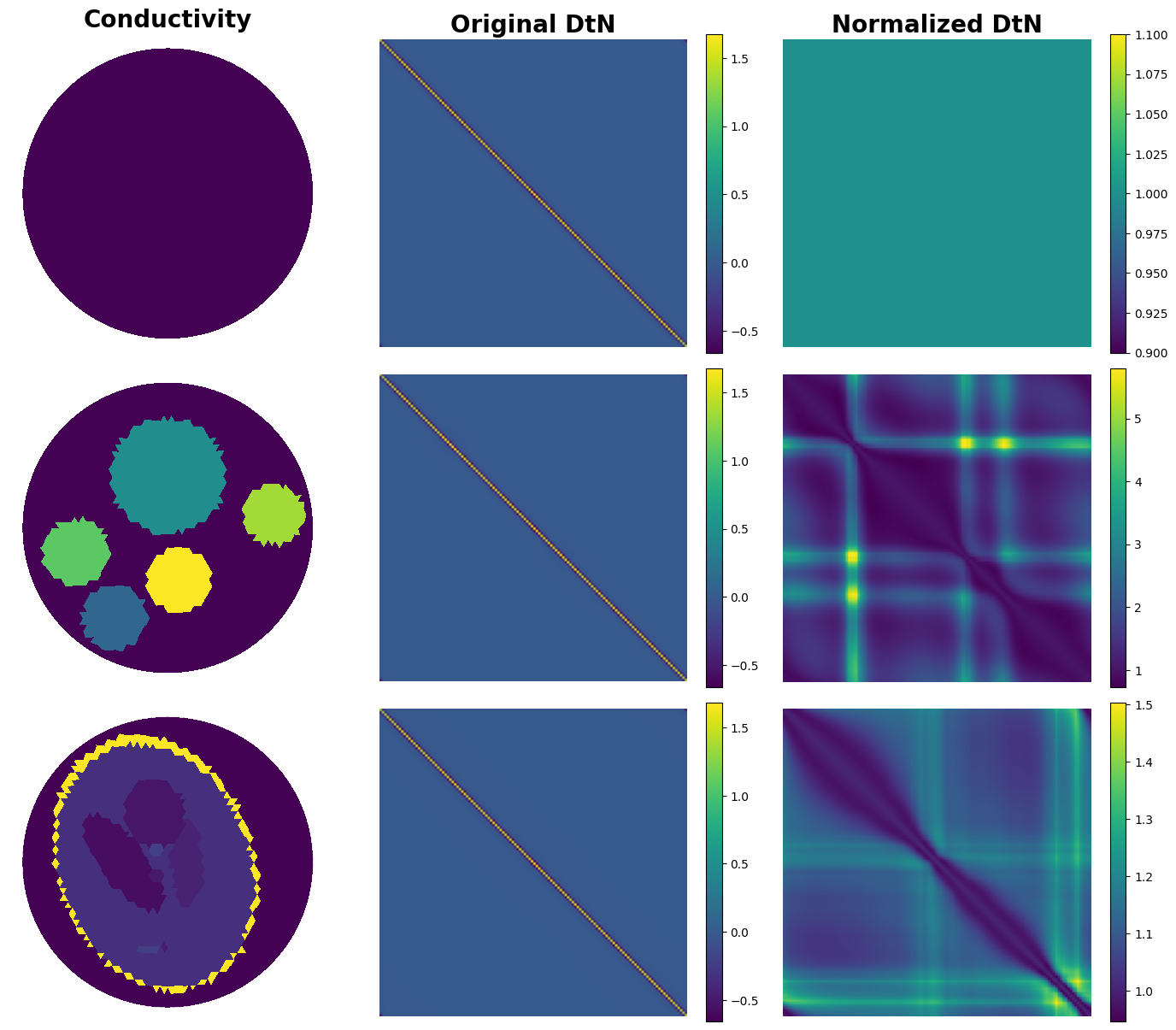}
\caption{Normalization of DtN measurements. 
\textbf{Left column:} conductivity fields (top: background; middle: disks; bottom: Shepp--Logan phantom). 
\textbf{Middle column:} raw DtN matrices computed from the corresponding conductivities. 
\textbf{Right column:} normalized DtN matrices produced by our preprocessing step.}
\label{fig:DtN_normalization}
\end{figure}

\paragraph{Measurement configurations.} 
Two types of mask matrices are considered in our numerical experiments.
\begin{itemize}
    \item \textbf{Principal submatrix}: 
    Sources and receivers are placed on a fixed number of boundary nodes that are chosen uniformly at random. In this case, the mask matrix $\mathbf{M}_s$ is a principal submatrix, i.e.
\[
(\mathbf{M}_{s})_{ij} = 
\begin{cases}
1, & i\in\mathcal{S} \text{ and } j\in\mathcal{S}\,,\\[2pt]
0, & \text{otherwise}\,.
\end{cases}
\]
    where an index set $\mathcal{S} \subset [N_{\mathcal B}]$ is drawn uniformly at random without replacement. The total number of boundary nodes $|\mathcal{S}| = \lfloor \sqrt{s}\,N_{\mathcal B}\rfloor$ is determined by the sampling rate $s\in (0,1)$. 
\item \textbf{Random}: Each entry of $\mathbf{\Lambda}_\gamma$ is observed independently with probability $s$ as in \citet{bui2022bridging}. The mask $\mathbf{M}_s$ is therefore a Bernoulli random matrix with parameter $s$, i.e.,
\[
(\mathbf{M}_{s})_{ij} = 
\begin{cases}
1, & \text{w.p. } s\,,\\[2pt]
0, & \text{w.p. } 1-s\,.
\end{cases}
\]
\end{itemize}
 
Examples of both types of mask matrices are plotted in Figure~\ref{fig:comparison_completion_DM_MC}.

\paragraph{Noise injection.} To investigate the robustness of our methods, we apply a noise pattern considered in \citet{borcea2013study} as follows, which is similar to multiplicative noise. 
\begin{equation*}
    \mathbf{\Lambda}_{\gamma}^{\text{Noisy}} = \mathbf{\Lambda}_{\gamma} + \sigma_\text{noise}\mathbf{\Lambda}_{\textbf{1}}\odot\mathcal{E} \,,
\end{equation*}
where $\mathbf{\Lambda}_{\textbf{1}}$ is the background DtN measurement, $\sigma_\text{noise}\in (0,1)$ is the strength of noise, and $\mathcal{E} \in \mathbb{R}^{N_{\mathcal{B}}\times N_{\mathcal{B}}}$ is a random matrix whose entries are independently drawn from the standard normal distribution. We consider $\sigma_\text{noise}= 0.05$ ($5\%$) in our experiments, and inject noise into $\mathbf{\Lambda}_\gamma$ first and then apply the normalization to obtain $\bar{\mathbf{\Lambda}}_\gamma$.

\subsection{Implementation details}\label{sec:Implement_details}
We evaluate the proposed method through two stages: data completion in Section~\ref{sec:result_DC} and conductivity reconstruction in Section~\ref{sec:result_IP}. 
In the data completion stage, we compare our method with the hierarchical completion baseline introduced in~\citet{bui2022bridging}. In the conductivity reconstruction stage, we adopt \texttt{IAE-Net}, an efficient neural operator introduced in~\citet{ong2022integral}, which achieves excellent accuracy on forward and inverse PDE operator learning tasks.
To demonstrate the effectiveness of our proposed method, we compare the results of three inverse problem solving strategies illustrated in Figure~\ref{fig:diagram}: 
(i) \emph{direct reconstruction from sparse measurements}; 
(ii) \emph{data completion followed by reconstruction}, where the data completion part is implemented by our diffusion model or by the hierarchical matrix completion; and (iii) \emph{direct reconstruction from full measurements}, which sets the best achievable accuracy for the specified inverse solver (e.g. \texttt{IAE-Net}).

\subsubsection{Data completion methods}
\paragraph{Diffusion model implementation.}
We adopt the Elucidated Diffusion Model (EDM) framework \citep{karras2022elucidating} with the preconditioned denoising network $D_{\theta}$, which can be interpreted as a reparameterization of the score functions in our setting. The denoiser network $D_\theta(x,\sigma)$ in \citet{karras2022elucidating} predicts a clean DtN measurement $x_0$ from a perturbed input $x = x_0 + \sigma \varepsilon$, where $\varepsilon \sim \mathcal{N}(0,I)$ and $\sigma$ is the noise level. 
For training, we use the same loss function and noise-level sampling strategy as in \citet{karras2022elucidating}, and train the model on normalized DtN measurement data. 
The denoiser $D_{\theta}$ is implemented as a 2D U-Net \citep{ronneberger2015u}, where each DtN matrix is treated as a single-channel image. The network has a base width of 128, includes two residual blocks at each resolution level, and performs dyadic downsampling and upsampling by a factor of 2, repeated three times, with an encoder--decoder symmetric structure and skip connections. Each residual block consists of $3\times 3$ convolutions, normalization, and SiLU activations. The scalar noise level $\sigma$ is encoded using a Fourier embedding and injected into the residual blocks via affine modulation. Attention mechanisms are applied only in the bottleneck layer. 
Optimization is carried out with AdamW \citep{loshchilov2017decoupled} using a learning rate of $10^{-4}$, an exponential moving average of model parameters with decay $0.999$, and a batch size of $32$. Unless otherwise specified, all remaining EDM hyperparameters and conventions follow \citet{karras2022elucidating}.

\paragraph{Baseline model: hierarchical matrix completion.} 
\begin{figure}[htbp]
    \centering
    \includegraphics[width=0.8\linewidth]{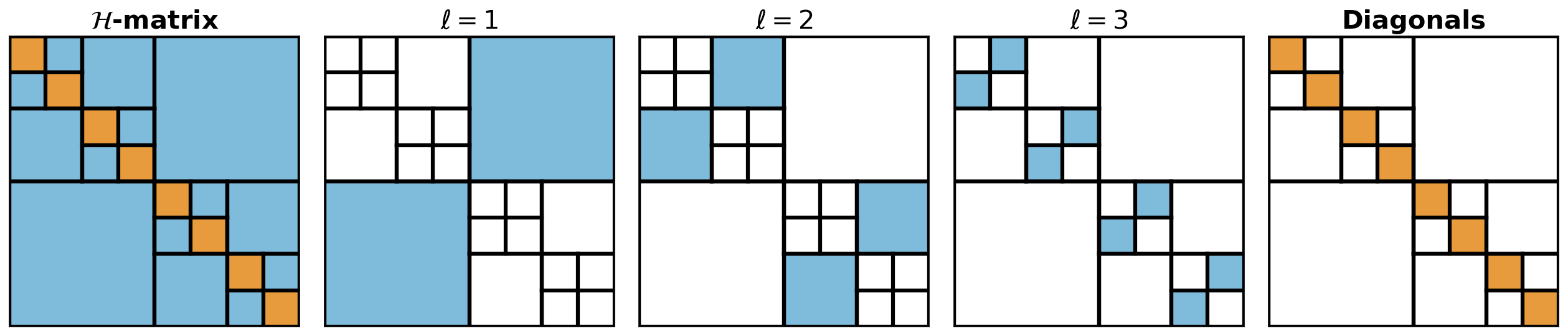}
    \caption{Hierarchical partition. \emph{Diagonal blocks} are colored in orange, and \emph{off-diagonal blocks} are colored in blue.}
    \label{fig:h-matrix}
\end{figure}
We consider the hierarchical matrix completion method in~\citet{bui2022bridging} as a baseline for completing DtN measurements. This method exploits the hierarchical off-diagonally low-rank structure of the DtN measurement matrix \citep{bebendorf2003existence,ballani2017matrices} and employs standard matrix completion methods~\citep{recht2011simpler}. Concretely, the DtN matrix is divided into a hierarchical tree of diagonal blocks and off-diagonal blocks as shown in Figure~\ref{fig:h-matrix}. 
To apply matrix completion globally, a corresponding hierarchical masking matrix needs to be considered for this method.
In particular, it is assumed that the observed EIT measurements consist of all data on the diagonal blocks and partial data that are sampled uniformly at random over all off-diagonal blocks. See Figure~\ref{fig:mask_comparison} for the global masking matrix considered in such hierarchical setting. 
Specifically, we implement a three-level decomposition of the measurement matrix, which leads to in total eight diagonal blocks and fourteen off-diagonal blocks. 
The matrix completion algorithm is applied in parallel to each of the off-diagonal blocks to obtain the completed data, which are then assembled together to form the completed measurement matrix.
In particular, consider an arbitrary off-diagonal block $\mathbf{\Lambda}_{\gamma}^{l} \in \mathbb{R}^{\frac{N_{\mathcal B}}{2^l}\times \frac{N_{\mathcal B}}{2^l}}$ at level $l=1,2,3$ and a random mask $\mathbf{M}_s^{l}$ of the same size, the matrix completion algorithm outputs a solution to the following constrained nuclear norm optimization problem
\begin{equation}
\label{eq:mc_nucnorm}
\widehat{\mathbf{\Lambda}}_\gamma^{l}
\;=\;
\arg\min_{X} \ \|X\|_{*}
\quad \text{subject to} \quad
\mathbf{M}_s^{l}\odot X \;=\; \mathbf{M}_s^{l}\odot \mathbf{\Lambda}_\gamma^{l}\,, \quad X \in\mathbb{R}^{\frac{N_{\mathcal B}}{2^l}\times \frac{N_{\mathcal B}}{2^l}}\,,
\end{equation}
where \(\|\cdot\|_{*}\) denotes the matrix nuclear norm. 
The above minimization problem admits a unique solution if the sampling rate in the random mask $\mathbf{M}_s^{l}$ is enough~\citep{recht2011simpler, bui2022bridging}.
We implement the nuclear norm minimization \Eqref{eq:mc_nucnorm} with the \texttt{MOSEK} solver from the \texttt{CVXPY} package.

\subsubsection{Inverse problem solver}  
We employ the neural operator \texttt{IAE-Net}~\citep{ong2022integral} as our inverse solver for its excellent accuracy across various PDE forward and inverse operator learning tasks.
Two $\texttt{IAE-Net}$ models are trained separately on clean data to serve as a full-data solver and sparse-data solver, respectively. 
In particular, the full-data solver is trained to learn the operator \(\mathcal{G}:\bar{\mathbf{\Lambda}}_\gamma \mapsto \gamma\) that maps the full DtN measurements to target conductivity field, while the sparse-data solver approximates \(\mathcal{G}_{\mathrm{masked}}: \bar{\mathbf{\Lambda}}_\gamma^{\text{o}} \mapsto \gamma\), where \(\bar{\mathbf{\Lambda}}_\gamma^{\text{o}} = \mathbf{M}_s \odot \bar{\mathbf{\Lambda}}_\gamma\). Specifically, $\mathbf{M}_s$ is resampled independently at each training iteration, enabling the model to adapt to diverse sparse-measurement patterns. Both \texttt{IAE-Net} models are configured with 32 modes and a stack of IAE blocks, using four blocks for the Disks dataset and two for the Shepp--Logan dataset. Training minimizes the mean-squared error between predicted and ground-truth conductivities, optimized with Adam and weight decay $10^{-4}$. The learning rate follows a reduce-on-plateau schedule (factor $0.5$, patience $20$), with early stopping after $40$ stagnant epochs and a maximum of $300$ epochs. The initial learning rate is $10^{-2}$ for the Disks dataset and $10^{-3}$ for the Shepp--Logan dataset.

\subsubsection{Evaluation metrics}\label{sec:eval_metr}

To compare the proposed data completion method with the baseline, we consider the following evaluation metrics. These metrics quantify the fidelity of the completed DtN data and the accuracy of the reconstructed conductivity respectively.
\paragraph{DtN data completion error.} We measure reconstruction quality using the normalized relative Frobenius error
$
\text{RE}(\widehat{\bar{\mathbf{\Lambda}}}_\gamma) = {\|\widehat{\bar{\mathbf{\Lambda}}}_\gamma  - \bar{\mathbf{\Lambda}}_\gamma \|_F}/{\|\bar{\mathbf{\Lambda}}_\gamma \|_F}\,,$
where $\widehat{\bar{\mathbf{\Lambda}}}_\gamma$ is the completed normalized DtN matrix. 

\paragraph{Conductivity reconstruction error.}
Note that the conductivity is defined on unit disk $\Omega \subset [-1,1]^2$ and padded zeros outside $\Omega$. We only measure the errors within $\Omega$. In particular, let \(\{x_i\}_{i=1}^{128}\) and \(\{y_j\}_{j=1}^{128}\) be the uniform grid coordinates on \([-1,1]\) and define
$
\Omega_{\mathrm{grid}} \;=\; \bigl\{(i,j)\,:\ x_i^2 + y_j^2 \le 1 \bigr\}, \
N_\Omega \;=\; \lvert \Omega_{\mathrm{grid}}\rvert\,.
$
For a given neural network prediction $\hat{\gamma}$, we measure its  relative \(\ell_2\) error and the mean absolute error with respect to the ground truth $\gamma$ restricted to \(\Omega\):
\[
\mathrm{RE}\!\left(\hat{\gamma}\right)
= {\sqrt{\sum_{(i,j)\in \Omega_{\mathrm{grid}}}\bigl(\hat{\gamma}_{i,j}-\gamma_{i,j}\bigr)^2}} /
       {\sqrt{\sum_{(i,j)\in \Omega_{\mathrm{grid}}} \gamma_{i,j}^2}}, 
\qquad
\mathrm{MAE}\!\left(\hat{\gamma}\right)
= \frac{1}{N_\Omega}\sum_{(i,j)\in \Omega_{\mathrm{grid}}}\bigl|\hat{\gamma}_{i,j}-\gamma_{i,j}\bigr|\,.
\]

In addition, we also introduce structural similarity (SSIM) for evaluating conductivity reconstruction.
SSIM quantifies perceptual similarity by jointly comparing local luminance, contrast, and structure, thereby capturing preservation of geometry and edges more faithfully than pixel-wise metrics such as $\ell_2$ error or MAE. In our implementation, we compute SSIM using the \texttt{scikit-image} package \citep{van2014scikit} with default settings.

\subsection{Results for data completion}\label{sec:result_DC}
We compare our diffusion completion to the hierarchical matrix completion method as in \citet{bui2022bridging}. Since identical matrix completion algorithms \Eqref{eq:mc_nucnorm} are applied to all off-diagonal blocks with the same sampling rate, it is equivalent to compare data completion results restricted to a single off-diagonal block.
To this end, we implement our method and matrix completion over the upper right off-diagonal block at the first level in the subsequent subsection, and present the accuracy of data completion with various sampling rates from $s=1\%$ to $s=30\%$.

\paragraph{Off-diagonal block completion.}
Both methods are evaluated on the upper-right off-diagonal block of the DtN matrix, which has size $64\times64$ and a numerical rank of $10$. 
We test the diffusion completion method under principal submatrix masking with sampling rate $s=1\%$, and matrix completion method under both principal submatrix masking and random masking with various sampling rates, see Figure~\ref{fig:comparison_completion_DM_MC}. 
We find that the diffusion model is able to reconstruct the off-diagonal block of the DtN matrix within a relative error of $0.9\%$ with only $1\%$ of observed data.
In comparison, matrix completion fails to find any missing data under the same setting. This is because the masked measurement is already of low rank when the sampling rate $s=1\%$ is low, and thus becomes a solution to \Eqref{eq:mc_nucnorm}. Consequently, the objective does not enforce recovery of the unobserved rows and columns under this mask. 
We further test matrix completion under random masking with different sampling rates $s=1\%,15\%$ and $30\%$.
We observe that the same issue persists with low sampling rate $1\%$. An unsatisfactory reconstruction can be achieved with larger sampling rate $s=15\%$ with relative error $8.2\%$. 
When the random sampling rate is raised to $s=30\%$, a competitive reconstruction result with relative error $1.0\%$ is achieved. 
Overall, Figure~\ref{fig:comparison_completion_DM_MC} illustrates that a learned distributional prior (diffusion) can interpolate across ultra-sparse observations, while low-rank matrix completion requires denser, randomly distributed observations to perform well. 

\begin{figure}[htbp]
    \centering
    \includegraphics[width=0.8\linewidth]{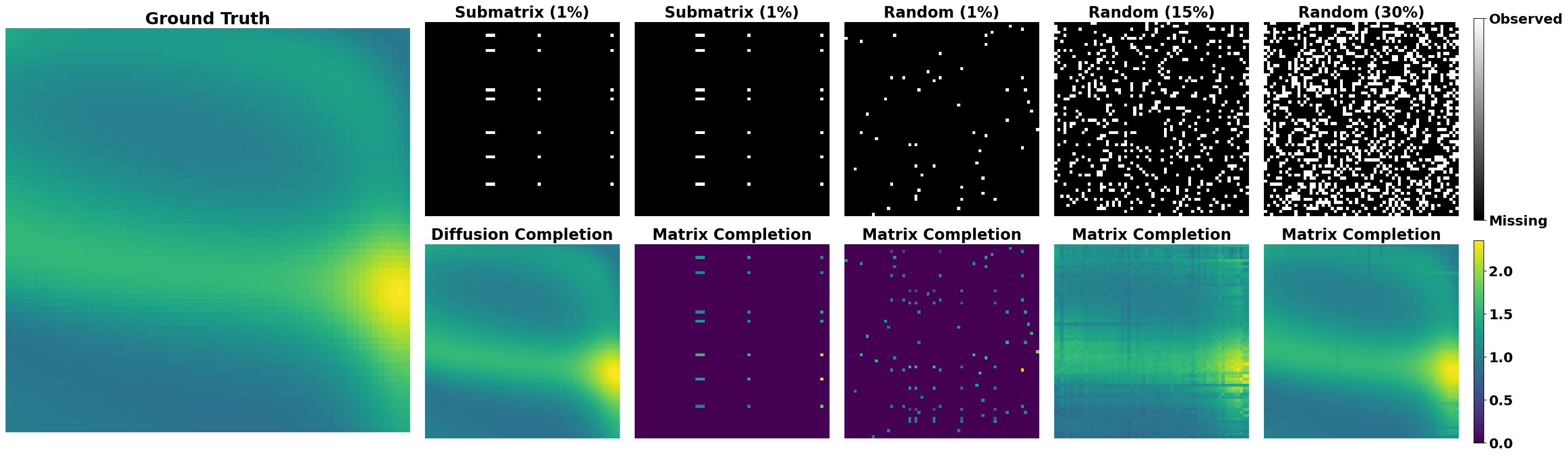}
    \caption{
    Diffusion completion versus matrix completion on an off-diagonal block under different masking patterns and sampling rate. Note that the matrix-completion baseline is only applied to off-diagonal blocks, where low-rank structure is expected, whereas our diffusion model is applied to the whole matrix and presents the off-block in this plot.
    \textbf{Left}: ground truth (normalized) upper-right block of DtN matrix. 
    \textbf{Right, top row}: masking matrices with sampling rate $s=1\%$, $s=15\%$ and $s=30\%$.
    \textbf{Right, bottom row}: corresponding reconstructions under the masks shown above. The diffusion model (first column) achieves high-quality reconstruction under the submatrix masking even at an extremely low sampling rate \(s=1\%\) with relative error (RE) \(0.9\%\). In contrast, matrix completion fails at \(s=1\%\) under both submatrix and random masking (second and third columns). Matrix completion requires substantially higher random sampling rates (fourth and fifth columns) to achieve comparable accuracy, reaching RE \(8.2\%\) at \(s=15\%\) and RE \(1.0\%\) at \(s=30\%\). In contrast, our diffusion model performs completion over the entire DtN matrix.
    }
    \label{fig:comparison_completion_DM_MC}
\end{figure}

\subsection{Results for inverse problem solving}\label{sec:result_IP}

In this section, we investigate the effectiveness of the proposed measurement completion methods in inverse problem solving. Given sparse DtN measurements, we evaluate the conductivity reconstruction of different methods, including \textbf{i)} a direct reconstruction by a sparse-data solver; \textbf{ii)} a level-$3$ hierarchical matrix completion followed by a full-data solver; and \textbf{iii)} a diffusion completion followed by the same full-data solver. 
Note that even with the ground truth data, the inherent nonlinearity and ill-posedness of the EIT problem can lead to non-negligible reconstruction errors, particularly under noisy or sparse measurement conditions. 
To this end, the output of a full-data solver with full data represents the \emph{best possible reconstruction} from a specific inverse problem solver. We report this reconstruction as a reference to demonstrate the effectiveness of the aforementioned methods.

Specifically, to obtain reasonable reconstructions by matrix completion, we consider a higher sampling rate ($15\%$) in each off-diagonal block and $100\%$ in diagonal blocks. This results in a global hierarchical random mask with an overall sampling rate of $25.625\%$. 
For diffusion completion and the direct sparse solver, a more practical principal submatrix masking is considered with sampling rate $1\%$. 
These two different masking matrices are plotted in Figure~\ref{fig:mask_comparison}.
\begin{figure}[htbp]
    \centering
    \includegraphics[width=0.45\linewidth]{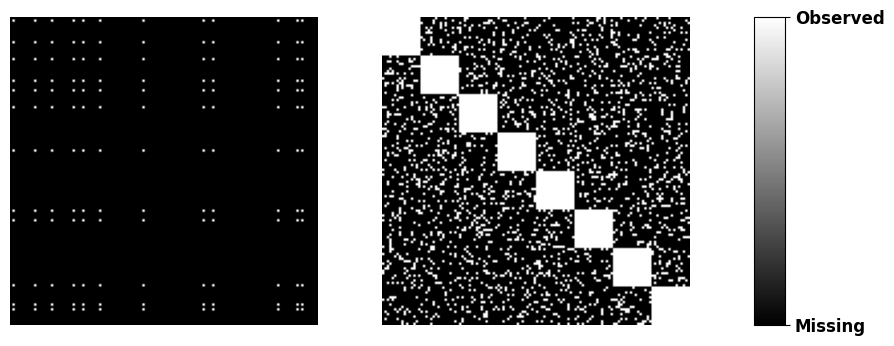}
\caption{Global masking matrices. \textbf{Left:} The mask matrix used for diffusion completion and the sparse-data solver. This mask consists of nonzero values over a principal submatrix with sampling rate $s=1\%$. \textbf{Right:} The hierarchical masking matrix with level $l=3$. In this case, all diagonal blocks are fully observed, and off-diagonal blocks are randomly sampled with sampling rate $s=15\%$, yielding an overall sampling rate of $25.625\%$.}

    \label{fig:mask_comparison}
\end{figure}

\subsubsection{Random disks}\label{sec:random_disk}
We present a visual comparison of reconstructions over different samples in Figure~\ref{fig:disk_deepsolver_noiseless} with clean data, and Figure~\ref{fig:disk_deepsolver_noise5e-2} with $5\%$ noise. A summary of errors over the testing dataset in both cases is shown in Table~\ref{tab:disk_recon_combined}.
In the noiseless setting, the full-data solver with full measurements achieves the best accuracy across all metrics, representing the solver’s best attainable performance. 
Using the same solver with diffusion-completed measurements from sparse data yields the closest performance to the best possible reconstructions. 
The direct sparse-data solver lags behind across all metrics, consistent with the observed blur and boundary erosion. H-matrix completion outperforms the direct sparse-data solver and has larger reconstruction errors than diffusion completion, and it requires a much higher sampling rate (\(25.625\%\)) compared to the other cases (\(1\%\)). The errors increase under noise, while diffusion completion retains the best values among the tested sparse-data pipelines.
\begin{table}[htbp]
    \centering
    \caption{Reconstruction results for the \textbf{Disks} dataset. 
    \textbf{First row:} best possible reconstruction using full data.
    \textbf{Last three rows:} various inverse problem solving strategies with sparse data. The overall sampling rate is indicated in the brackets, and the best results with sparse data are marked in bold. }
    \label{tab:disk_recon_combined}
    \renewcommand{\arraystretch}{1.0}
    \begin{tabular}{@{}lcccccc@{}}
        \toprule
        \multirow{2}{*}{Method} &
        \multicolumn{3}{c}{Noiseless DtN} &
        \multicolumn{3}{c}{5\% Noisy DtN} \\
        \cmidrule(lr){2-4}\cmidrule(lr){5-7}
        & SSIM$\uparrow$ & RE$\downarrow$  & MAE$\downarrow$  
        & SSIM$\uparrow$  & RE$\downarrow$   & MAE$\downarrow$  \\
        \midrule
        full-data solver ($100\%$)        & 0.819 & 20.7\% & 0.176  & 0.808 & 21.6\% & 0.186  \\
        \hline
        diffusion completion ($1\%$) + full-data solver & \textbf{0.774} & \textbf{33.4}\% & \textbf{0.288}  & \textbf{0.762} & \textbf{34.4}\% & \textbf{0.302}  \\
        direct sparse-data solver ($1\%$)                     & 0.597 & 44.9\% & 0.495 & 0.578 & 46.5\% & 0.524 \\
        H-matrix completion ($25.625\%$) + full-data solver & 0.730 & 41.4\% & 0.361 & 0.712& 43.7\% & 0.388\\
        \bottomrule
    \end{tabular}
\end{table}

\begin{figure}[htbp]
    \centering
    \includegraphics[width=0.8\linewidth]{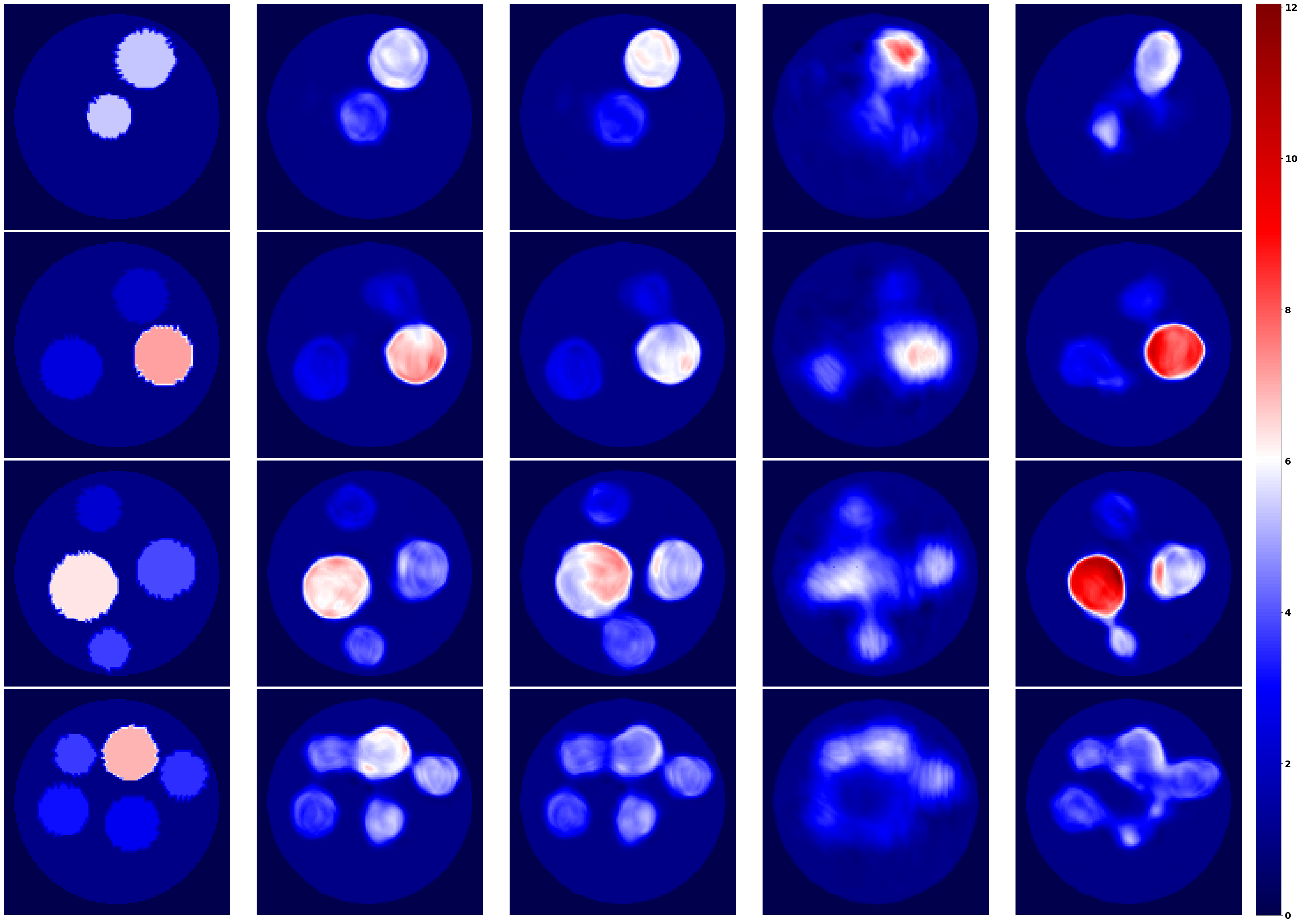}

\caption{
Reconstructions from the \emph{noiseless} \textbf{random disks} dataset. 
\textbf{Columns (left$\to$right):} (1) \emph{Ground truth}; (2) \emph{Best possible reconstruction} from the full-data solver (overall sampling rate $100\%$); (3) Reconstructions from the full-data solver applied to \emph{diffusion completion} of measurements (overall sampling rate $1\%$); (4) Reconstructions from the \emph{direct sparse-data solver} (overall sampling rate $1\%$); (5) Reconstructions from the full-data solver applied to \emph{H-matrix completion} of measurements (overall sampling rate $25.625\%$). \textbf{Rows (top$\to$bottom):} number of disks $n=2,3,4,5$.}
\label{fig:disk_deepsolver_noiseless}
\end{figure}

\begin{figure}[htbp]
    \centering
    \includegraphics[width=0.8\linewidth]{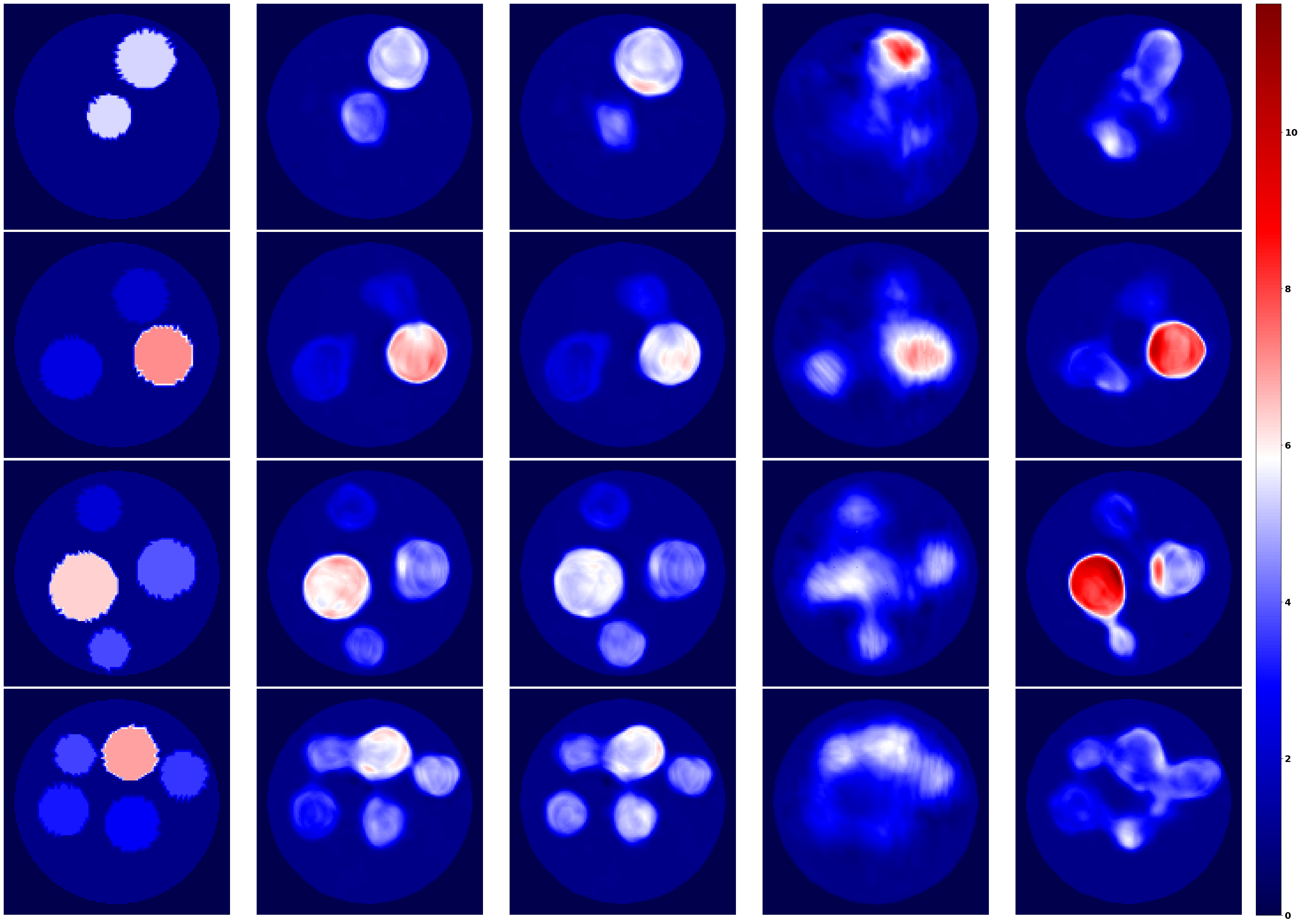}
\caption{
Reconstructions from the \textbf{random disks} dataset with \emph{$5\%$ noise}. 
\textbf{Columns (left$\to$right):} (1) \emph{Ground truth}; (2) \emph{Best possible reconstruction} from the full-data solver (overall sampling rate $100\%$); (3) Reconstructions from the full-data solver applied to \emph{diffusion completion} of measurements (overall sampling rate $1\%$); (4) Reconstructions from the \emph{direct sparse-data solver} (overall sampling rate $1\%$); (5) Reconstructions from the full-data solver applied to \emph{H-matrix completion} of measurements (overall sampling rate $25.625\%$). \textbf{Rows (top$\to$bottom):} number of disks $n=2,3,4,5$.
}

\label{fig:disk_deepsolver_noise5e-2}
\end{figure}

\subsubsection{Random Shepp--Logan phantoms}\label{sec:SL}
The same experiments and comparisons are conducted for the Shepp--Logan dataset with visual reconstructions in Figure~\ref{fig:sl_deepsolver_noiseless} and Figure~\ref{fig:sl_deepsolver_noise5e-2}, and errors shown in Table~\ref{tab:sl_recon_combined}. In the noiseless setting, diffusion completion + full-data solver produces reconstructions close to the best possible result, preserving the thin skull boundary and the intracranial cavities. Unlike the Disks dataset, the {direct sparse-data solver} performs better than {H-matrix completion} on Shepp--Logan: the {direct sparse-data solver} recovers the main structures but tends to smooth thin features, whereas {H-matrix completion + full-data solver} fails to recover both the skull boundary and the cavities. With \(5\%\) noise, robustness differences become more apparent. The {best possible reconstruction} and the {diffusion-completion} strategy both remain stable, while the {direct sparse-data solver} exhibits edge distortions at the boundary and amplified blur artifacts in the zoomed region. The {H-matrix completion + full-data solver} continues to fail as in the noiseless setting. Overall, diffusion completion narrows the gap to the best possible reconstruction across all metrics and maintains this advantage under noise, whereas the direct sparse-data solver and H-matrix completion struggle to solve the inverse problem.
\begin{table}[htbp]
    \centering
    \caption{Reconstruction results for the \textbf{Shepp--Logan} dataset. 
    \textbf{First row:} best possible reconstruction using full data.
    \textbf{Last three rows:} various inverse problem solving strategies with sparse data. The overall sampling rate is indicated in the brackets, and the best results with sparse data are marked in bold.}
    \label{tab:sl_recon_combined}
    \renewcommand{\arraystretch}{1.0}
    \begin{tabular}{@{}lcccccc@{}}
        \toprule
        \multirow{2}{*}{Method} &
        \multicolumn{3}{c}{Noiseless DtN} &
        \multicolumn{3}{c}{5\% Noisy DtN} \\
        \cmidrule(lr){2-4}\cmidrule(lr){5-7}
        & SSIM$\uparrow$ & RE$\downarrow$ & MAE$\downarrow$ 
        & SSIM$\uparrow$ & RE$\downarrow$ & MAE$\downarrow$ \\
        \midrule
        full-data solver ($100\%$)                 
        & 0.930 & 8.2\% & 0.030 & 0.898 &10.2\% &0.042 \\
        \hline
        diffusion completion ($1\%$) + full-data solver & \textbf{0.910} & \textbf{10.3}\% & \textbf{0.036} & \textbf{0.842} & \textbf{14.3\%} & \textbf{0.055} \\

        direct sparse-data solver ($1\%$)                     
        &0.886 & 11.3\% & 0.043 & 0.776 & 16.4\% & 0.075 \\
        H-matrix completion ($25.625\%$) + full-data solver & 0.825 & 17.0\% & 0.067 & 0.788 & 19.5\% & 0.085\\
        \bottomrule
    \end{tabular}
\end{table}

\begin{figure}[htbp]
    \centering
    \includegraphics[width=0.8\linewidth]{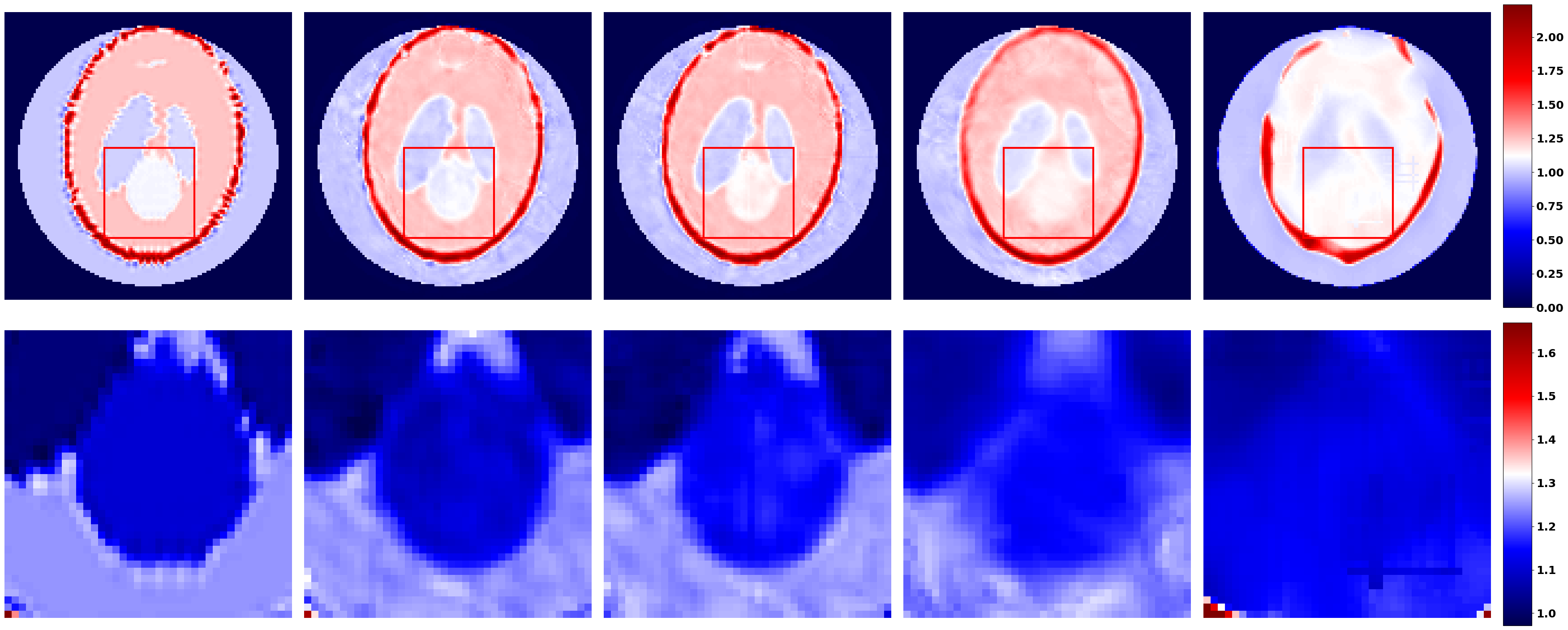}
\caption{
Reconstructions from the \emph{noiseless} \textbf{Shepp--Logan} dataset. 
\textbf{Columns (left$\to$right):} (1) \emph{Ground truth}; (2) \emph{Best possible reconstruction} from the full-data solver (overall sampling rate $100\%$); (3) Reconstructions from the full-data solver applied to \emph{diffusion completion} of measurements (overall sampling rate $1\%$); (4) Reconstructions from the \emph{direct sparse-data solver} (overall sampling rate $1\%$); (5) Reconstructions from the full-data solver applied to \emph{H-matrix completion} of measurements (overall sampling rate $25.625\%$). \textbf{Rows (top$\to$bottom):} full field of view (top) with a red box indicating the zoom region; zoomed-in crops of that region (bottom).
}

\label{fig:sl_deepsolver_noiseless}
\end{figure}

\begin{figure}[htbp]
    \centering
    \includegraphics[width=0.8\linewidth]{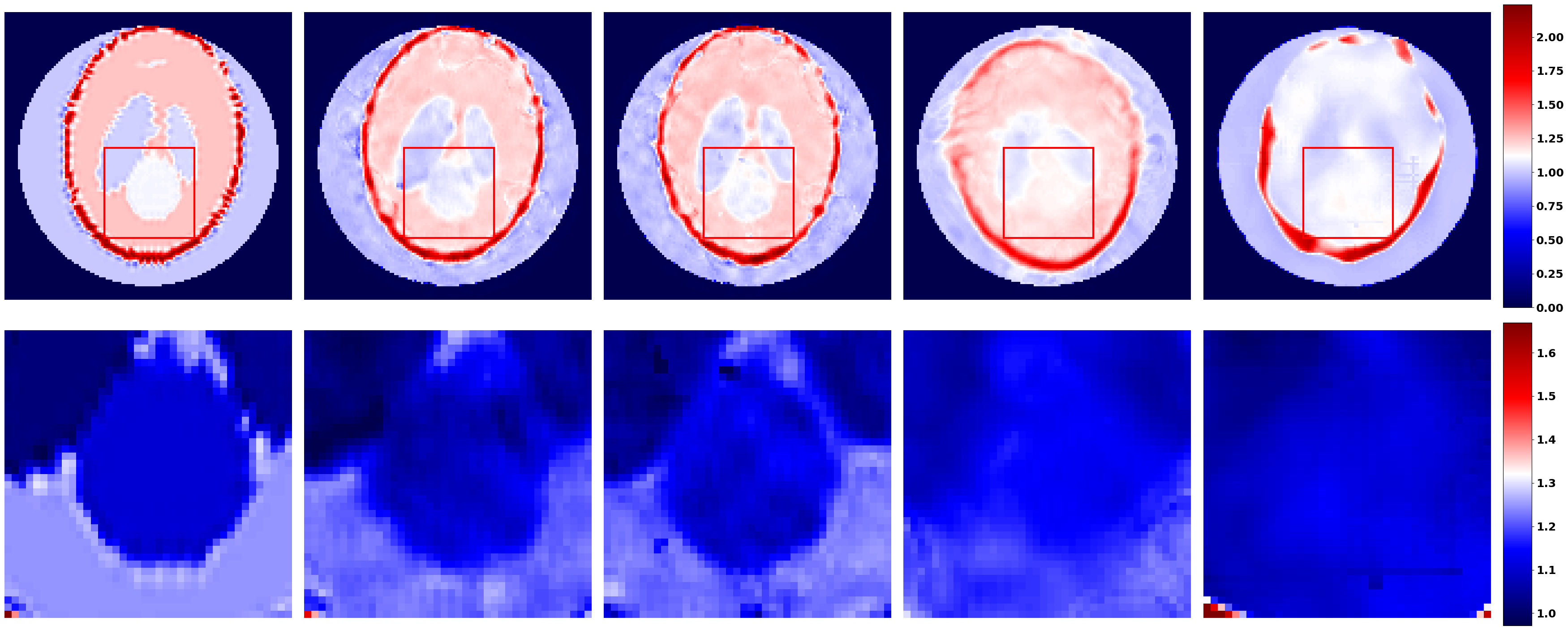}
\caption{
Reconstructions from the random \textbf{Shepp--Logan} dataset \emph{$5\%$ noise}. 
\textbf{Columns (left$\to$right):} (1) \emph{Ground truth}; (2) \emph{Best possible reconstruction} from the full-data solver (overall sampling rate $100\%$); (3) Reconstructions from the full-data solver applied to \emph{diffusion completion} of measurements (overall sampling rate $1\%$); (4) Reconstructions from the \emph{direct sparse-data solver} (overall sampling rate $1\%$); (5) Reconstructions from the full-data solver applied to \emph{H-matrix completion} of measurements (overall sampling rate $25.625\%$). \textbf{Rows (top$\to$bottom):} full field of view (top) with a red box indicating the zoom region; zoomed-in crops of that region (bottom).
}
    \label{fig:sl_deepsolver_noise5e-2}
\end{figure}

\section{Conclusion and Discussion}
We have proposed a diffusion-based DtN completion framework for EIT with extremely sparse boundary measurements and evaluated it on the Disks and Shepp--Logan datasets. Using the same inverse solver across all settings, reconstructions from diffusion completion are consistently closest to the full measurement baseline, outperforming both the direct sparse-data solver and H-matrix completion. These results support our diffusion completion design, which improves stability and accuracy without modifying off-the-shelf inverse solvers.
Our analysis shows that the discretized DtN measurement set $F_N(K)$ inherits the low-dimensional
structure of the polygon conductivity class $K$. Using the distributed form of the shape derivative, we proved that the forward map $\gamma \to \Lambda_\gamma$ is Lipschitz with a constant independent of the number of vertices. As a consequence, the metric entropy of $F_N(K)$ is governed by the parameter dimension $2n_{\mathrm{v}}$ rather than by the ambient measurement dimension. Combining the resulting complexity estimate with the conditional diffusion convergence theorem of \citet{liang2025low} yields a nonasymptotic total variation guarantee for DtN completion, with explicit dependence on the diffusion steps $T$ and the number of vertices $n_{\mathrm{v}}$. 

We next outline several promising directions for further development. First, although we considered simulated data in this paper, diffusion completion can be directly applied to real EIT applications provided sufficient lab data are available for training. Second, encoding the structure of the DtN matrix can improve generalization and data efficiency. In addition to hierarchical low-rankness \citep{bebendorf2003existence, ballani2017matrices}, DtN measurements in our setting are symmetric and satisfy zero row or column sum constraints \citep{somersalo1992existence}, which can be enforced by specific network architectural designs or training penalties. Third, accelerating sampling is important for practical deployment. Distillation \citep{song2023consistency} and few-step samplers \citep{song2020denoising} in diffusion models can be directly incorporated into our framework and reduce inference cost while maintaining accuracy on conditional tasks. Fourth, the generative measurement prior can guide the choice of boundary signals. Expected information gain provides a clear objective for selecting informative measurements, and modern Bayesian optimal experimental design methods make this feasible in EIT settings \citep{dong2025variational, hellmuth2025data}.

\section*{Acknowledgments}
The authors were partially supported by the US National Science Foundation under awards DMS-2244988, CAIG RISE-2530596, IIS-2520978, and the DARPA D24AP00325-00.
   
\bibliographystyle{plainnat}
\bibliography{ref}  

\appendix
\section{Uniform Lipschitz boundary for convex admissible polygons}
\label{sec:lipschitz_bdy_class}
\begin{definition}[Lipschitz boundary class]
\label{def:lipschitz_bdy_class}
Fix constants $r>0$ and $K>1$. We say that an open set $D\subset\R^2$ has
\emph{Lipschitz boundary of class $(r,K)$} if for every point $x_0\in\partial D$
there exist a rigid motion $\mathcal T_{x_0}$ with $\mathcal T_{x_0}(x_0)=0$ and a $K$-Lipschitz function
$\phi:[-r,r]\to\R$ with $\phi(0)=0$ such that, writing $Q:=[-r,r]\times[-Kr,Kr]$,
\[
\mathcal T_{x_0}(D)\cap Q=\{(x_1,x_2)\in Q:\ x_2>\phi(x_1)\}.
\]
\end{definition}

\begin{lemma}[Uniform Lipschitz boundary for relaxed admissible polygons]\label{lem:A4_from_A23_conv}
Let $P\in\cA_{1/2}$. Then $D:=\operatorname{int}P$ has Lipschitz boundary
of class $(r_*,K_*)$ in the sense of Definition~\ref{def:lipschitz_bdy_class}, with
\[
r_*=\frac{d_1}{16}\sin^2\!\Bigl(\frac{\beta_0}{4}\Bigr),
\qquad
K_*=\cot\!\Bigl(\frac{\beta_0}{4}\Bigr).
\]
\end{lemma}
\begin{proof}
Set $\alpha:=\beta_0/4$, so that $K_*=\cot\alpha$ and $r_*=\frac{d_1}{16}\sin^2\alpha$.
Let
\[
Q:= [-r_*,r_*]\times[-K_*r_*,K_*r_*],
\qquad
R:= r_*\sqrt{1+K_*^2}.
\]
Then $Q\subset \overline{B_2(0,R)}$ since $x_1^2+x_2^2\le r_*^2+(K_*r_*)^2=R^2$ for
$(x_1,x_2)\in Q$. Moreover $\sqrt{1+K_*^2}=\sqrt{1+\cot^2\alpha}=1/\sin\alpha$,
hence $R=r_*/\sin\alpha=\frac{d_1}{16}\sin\alpha$.

Fix $x_0\in\partial P$. By (A3$'$) we have $\beta_i<\pi$ for all $i$, hence $P$ is a
strictly convex polygon and in particular $P=\operatorname{conv}\{v_0,\dots,v_{n_{\mathrm v}-1}\}$.
Choose an edge $[v_i,v_{i+1}]$ containing $x_0$, and without loss of generality
assume $|x_0-v_i|\le |x_0-v_{i+1}|$.
Let $T_i:=\operatorname{conv}\{v_{i-1},v_i,v_{i+1}\}\subset P$ and let $S_i:=[v_{i-1},v_{i+1}]$.
Denote by $L_{S_i}$ the line through $S_i$ and set $h_i:=\dist(v_i,L_{S_i})$.
Write $a:=|v_{i-1}-v_i|$ and $b:=|v_{i+1}-v_i|$.

Since $P\in\cA_{1/2}$, the relaxed constraints give $a,b>d_1/2$ (by (A2$'$))
and $\beta_i\in(\beta_0/2,\pi-\beta_0/2)$ (by (A3$'$)); in particular
$a,b\ge d_1/2$ and $\sin\beta_i\ge\sin(\beta_0/2)$. The area of $T_i$ equals both
$\tfrac12 ab\sin\beta_i$ and $\tfrac12|S_i|\,h_i$. Since $|S_i|\le a+b$,
\[
h_i=\frac{ab\sin\beta_i}{|S_i|}
\ge \frac{ab\,\sin(\beta_0/2)}{a+b}
\ge \frac{d_1}{4}\sin\frac{\beta_0}{2},
\]
where the last step uses $\dfrac{ab}{a+b}=\bigl(\tfrac1a+\tfrac1b\bigr)^{-1}
\ge\bigl(\tfrac{2}{d_1}+\tfrac{2}{d_1}\bigr)^{-1}=\dfrac{d_1}{4}$.

The point $x_0$ lies on the edge $[v_i,v_{i+1}]$. Along this segment the distance
to $L_{S_i}$ varies affinely from $h_i$ at $v_i$ to $0$ at $v_{i+1}$. Since $|x_0-v_i|\le|x_0-v_{i+1}|$, i.e.\
$|x_0-v_i|\le\frac12|v_i-v_{i+1}|$, we have
\[
\dist(x_0,L_{S_i})\ge \frac{h_i}{2}\ge \frac{d_1}{8}\sin\frac{\beta_0}{2}.
\]
Since $\alpha=\beta_0/4\in(0,\pi/8)$ we have $\cos\alpha>1/4$, hence
\[
R=\frac{d_1}{16}\sin\alpha
<\frac{d_1}{4}\sin\alpha\cos\alpha
=\frac{d_1}{8}\sin\frac{\beta_0}{2}
\le \dist(x_0,L_{S_i}).
\]
Therefore $\overline{B_2(x_0,R)}$ is contained in the closed half-plane $H_i$
determined by $L_{S_i}$ that contains $v_i$.

Let $W_i$ be the closed infinite wedge given by the intersection of the two
closed inward half-planes bounded by the lines through the two edges
incident to $v_i$. By
convexity of $P$, $P\subset W_i$. Moreover $T_i=W_i\cap H_i$. Since
$\overline{B_2(x_0,R)}\subset H_i$, it follows that
$W_i\cap \overline{B_2(x_0,R)}\subset T_i\subset P$. Combining with $P\subset W_i$,
\[
P\cap \overline{B_2(x_0,R)}=W_i\cap \overline{B_2(x_0,R)}.
\]

Let $\mathcal R$ be a rotation sending the inward bisector direction
of the interior angle at $v_i$ to the positive $x_2$-direction, and set
$\mathcal T_{x_0}(x):=\mathcal R(x-x_0)$.
In these coordinates, the two supporting lines of the incident edges
are graphs $x_2=\phi_\pm(x_1)$ of affine functions satisfying
\[
|\phi_\pm'|
=\cot\!\Bigl(\frac{\beta_i}{2}\Bigr)
\le \cot\!\Bigl(\frac{\beta_0}{4}\Bigr)
=K_*.
\]
Set $\phi:=\max\{\phi_-,\phi_+\}$.
Then $\phi$ is $K_*$-Lipschitz and $\phi(0)=0$, since
$\mathcal T_{x_0}(x_0)=0$ lies on the boundary of the transformed wedge.
Moreover,
\[
\mathcal T_{x_0}(W_i)
=\{(x_1,x_2):x_2\ge\phi(x_1)\},
\qquad
\mathcal T_{x_0}(\operatorname{int}W_i)
=\{(x_1,x_2):x_2>\phi(x_1)\}.
\]

Since $R<\dist(x_0,L_{S_i})$, the closed ball
$\overline B_2(x_0,R)$ is contained in $\operatorname{int}H_i$.
The relations $T_i=W_i\cap H_i\subset P\subset W_i$ therefore imply
that $P$ and $W_i$ coincide in an open neighborhood of this closed ball.
Their interiors consequently coincide there as well.
Using $Q\subset\overline B_2(0,R)$ and the fact that
$\mathcal T_{x_0}$ is an isometry, we obtain
\[
\mathcal T_{x_0}(\operatorname{int}P)\cap Q
=\mathcal T_{x_0}(\operatorname{int}W_i)\cap Q
=\{(x_1,x_2)\in Q:x_2>\phi(x_1)\}.
\]
This is the Lipschitz boundary condition of class $(r_*,K_*)$
at $x_0$.
\end{proof}
\section{Proof of Theorem~\ref{thm:Linf}}\label{sec:Linf}

Theorem~\ref{thm:main} and Corollary~\ref{cor:map-lipschitz}
establish parts~(i) and~(ii) of Theorem~\ref{thm:Linf}, respectively.

\begin{theorem}[Uniform bound on $\DLa_i$ in the polygon sup-norm]
\label{thm:main}
Suppose $v \in \R^{2 n_{\mathrm{v}}}$ satisfies $P(v) \in \cA_{1/2}$ and the chart is chosen so that $\DLa_i(v)$ exists in the Fr\'echet sense. Then there is a constant
\[
L_\infty = L_\infty(\Omega, \kappa, d_0, d_1, \beta_0) < \infty,
\]
independent of $n_{\mathrm{v}}$, such that for every $V \in \R^{2 n_{\mathrm{v}}}$,
\[
\norm{\DLa_i(v)[V]}_{\cY} \le L_\infty \cdot \norm{V}_{\poly}.
\]
The constant $L_\infty$ depends on $\beta_0$ only through the geometric quantities $r_*, K_*$ of Lemma~\ref{lem:A4_from_A23_conv}.
\end{theorem}

Theorem~\ref{thm:main} implies the following corollary.
\begin{corollary}[Uniform Lipschitz of $\widetilde\Lambda_{P_i}$ on a convex patch]
\label{cor:map-lipschitz}
Let $P_i\in\mathcal A$, let
$\mathcal O_{P_i}$ be the coordinate patch from Lemma~\ref{lem:chart-M}, and
let $E_i\subset\mathcal O_{P_i}$ be convex.  Then for any $v,w\in E_i$,
\[
\|\widetilde\Lambda_{P_i}(w)-\widetilde\Lambda_{P_i}(v)\|_{\mathcal Y}
\le
L_\infty\|w-v\|_{\mathrm{poly}}.
\]
\end{corollary}

\begin{proof}
Set $d=w-v$ and $\eta(t)=v+td$.  Since $E_i$ is convex,
$\eta(t)\in E_i\subset\mathcal O_{P_i}$ for all $t\in[0,1]$.
Define
$
F(t):=\widetilde\Lambda_{P_i}(\eta(t))\in\mathcal Y.
$
For every $\ell\in\mathcal Y^\ast$, the scalar function
$\phi_\ell(t):=\ell(F(t))$ is continuous on $[0,1]$ and differentiable on
$(0,1)$, with
$
\phi_\ell'(t)
=
\ell\!\left(D\widetilde\Lambda_{P_i}(\eta(t))[d]\right).
$
By Theorem~\ref{thm:main},
$
|\phi_\ell'(t)|
\le
\|\ell\|_{\mathcal Y^\ast}L_\infty\|d\|_{\mathrm{poly}}.
$
The scalar mean-value theorem gives
\[
|\ell(\widetilde\Lambda_{P_i}(w)-\widetilde\Lambda_{P_i}(v))|
\le
\|\ell\|_{\mathcal Y^\ast}L_\infty\|w-v\|_{\mathrm{poly}}.
\]
Taking the supremum over $\|\ell\|_{\mathcal Y^\ast}\le1$ gives the claim.
\end{proof}
\begin{remark}[On the dependence on geometric data]
The $n_{\mathrm{v}}$-independence is not the same as a fully uniform bound: as $\beta_0 \downarrow 0$ or $d_1 \downarrow 0$, the constant $L_\infty$ blows up (through $C_\partial$). The class $\cA$ itself is bounded above as $n_{\mathrm{v}} \le 2\pi/\beta_0$, so genuine $n_{\mathrm{v}} \to \infty$ requires $\beta_0 \downarrow 0$, and our bound does not claim uniformity in that regime. The point of Theorem~\ref{thm:main} is rather: at fixed geometric a-priori parameters $(d_0, d_1, \beta_0)$, the Lipschitz constant of $\tLa_i$ does not pick up a hidden additional dependence on $n_{\mathrm{v}}$.
\end{remark}
 
The proof of Theorem~\ref{thm:main} is given in
Sections~\ref{sec:extension}--\ref{sec:norm-bound}. 
\subsection{The velocity extension \texorpdfstring{$\HV$}{HV}}
\label{sec:extension}
 
Given $V \in \R^{2n_{\mathrm{v}}}$ and $v$ with $P(v) \in \cA_{1/2}$, we construct a Lipschitz vector field $\HV: \Omega \to \R^2$ supported in a tubular neighborhood of $\partial P(v)$, carrying $\partial P(v)$ to $\partial P(v + tV)$ under the mapping $x \mapsto x + t\HV(x)$, and with
\begin{equation}
\norm{\HV}_{L^\infty(\Omega)} \le \sqrt 2 \, \norm{V}_{\poly}, \qquad \norm{D\HV}_{L^\infty(\Omega)} \le C_1(d_0, d_1, \beta_0) \cdot \norm{V}_{\poly}.
\label{eq:HV-bound}
\end{equation}
 
By Lemma~\ref{lem:A4_from_A23_conv}, for every
$P\in\cA_{1/2}$, $\operatorname{int}P$ has Lipschitz boundary
of class $(r_*,K_*)$, with constants $r_*,K_*$ depending only
on $(d_1,\beta_0)$.
We use this uniform Lipschitz boundary property in
Lemma~\ref{lem:hV-prop}(iv) to establish the Euclidean
Lipschitz estimate for the boundary velocity $h_V$,
which is needed for its McShane extension.

For the cutoff construction, define
$\sigma(x):=\dist(x,\partial P)$ and set $\rho:=d_0/8$.
The function $\sigma$ is $1$-Lipschitz on $\R^2$ by the
triangle inequality. Since $P\subset\Omega$ and
$\dist(\partial P,\partial\Omega)>d_0/2$, the closed
$\rho$-neighborhood of $\partial P$ is contained in $\Omega$
and has distance greater than $d_0/4$ from $\partial\Omega$.
 
\subsubsection{Boundary velocity \texorpdfstring{$\hV$}{hV}}
Define $\hV: \partial P(v) \to \R^2$ by piecewise-affine interpolation:
\begin{equation}
\hV\bigl((1-s) v_j + s v_{j+1}\bigr) := (1-s) V_j + s V_{j+1}, \qquad s \in [0, 1].
\label{eq:hV-def}
\end{equation}
 
\begin{lemma}[Properties of $\hV$]
\label{lem:hV-prop}
Assume $P(v) \in \cA_{1/2}$, so every edge length $\ell_j \ge d_1/2$. Let $(r_*, K_*)$ be the constants from Lemma~\ref{lem:A4_from_A23_conv}.
\begin{enumerate}[label=\textup{(\roman*)}]
  \item $\hV$ is well-defined on $\partial P(v)$.
  \item $\norm{\hV}_{L^\infty(\partial P)} \le \norm{V}_{\poly}$. 
  \item $\hV$ is Lipschitz with respect to the arclength metric $d_{\partial P}$:
  \[
  \abs{\hV(x) - \hV(y)} \le \frac{4 \norm{V}_{\poly}}{d_1} \cdot d_{\partial P}(x, y) \qquad \text{for all } x, y \in \partial P(v).
  \]
  \item $\hV$ is Lipschitz on $\partial P(v) \subset \R^2$ in the Euclidean metric:
  \[
  \abs{\hV(x) - \hV(y)} \le C_\partial \norm{V}_{\poly} \cdot \abs{x - y}_{\R^2},
  \qquad C_\partial := \max\!\left( \frac{4\sqrt{1 + K_*^2}}{d_1}, \ \frac{4}{r_*} \right),
  \]
  where $C_\partial = C_\partial(d_1, \beta_0)$ depends only on $(d_1, \beta_0)$, not on $n_{\mathrm{v}}$.
\end{enumerate}
\end{lemma}
 
\begin{proof}

(i), (ii): immediate from the convex-combination form~\Eqref{eq:hV-def}.
 
(iii): On a single edge $[v_j, v_{j+1}]$, $\hV$ is affine in arclength with derivative $(V_{j+1} - V_j)/\ell_j$, so
\[
\abs{\hV(x) - \hV(y)} \le \frac{\abs{V_{j+1} - V_j}}{\ell_j} \cdot d_{\partial P}(x, y) \le \frac{2 \norm{V}_{\poly}}{d_1/2} \cdot d_{\partial P}(x, y) = \frac{4 \norm{V}_{\poly}}{d_1} \cdot d_{\partial P}(x, y).
\]
For $x, y$ on different edges, decompose the shorter arc of $\partial P$ between $x$ and $y$ into consecutive edge-pieces; $\hV$ is continuous at each intermediate vertex, so the edgewise bound integrates to the same global Lipschitz constant w.r.t.\ $d_{\partial P}$.
 
(iv): We use Lemma~\ref{lem:A4_from_A23_conv} for nearby points
and part~(ii) for points farther apart. 
 
\textbf{Near case.} Suppose $\abs{x - y}_{\R^2} \le r_*/2$. By Lemma~\ref{lem:A4_from_A23_conv}, there exist a rigid motion $\mathcal T$ with $\mathcal T(x) = 0$ and a $K_*$-Lipschitz function $\phi: [-r_*, r_*] \to \R$ with $\phi(0) = 0$, such that, in $\mathcal T$-coordinates,
\[
\partial P \cap \bigl([-r_*, r_*] \times [-K_* r_*, K_* r_*]\bigr) = \bigl\{ (s, \phi(s)) : s \in [-r_*, r_*]\bigr\}.
\]
Now $\mathcal T(y)$ has Euclidean distance $\le r_*/2$ from $\mathcal T(x) = 0$, so $\mathcal T(y)$ lies in the open square $(-r_*, r_*)^2$. Since $\mathcal T(y) \in \partial P$, it must equal $(s_y, \phi(s_y))$ for some $s_y \in (-r_*, r_*)$ with $\abs{s_y} \le \abs{\mathcal T(y)}_{\R^2} = \abs{x - y}_{\R^2}$. The graph segment $\{(s, \phi(s)) : s \in [\min(0, s_y), \max(0, s_y)]\}$ is itself a path along $\partial P$ from $x$ to $y$; its arclength
\[
\int_{\min(0, s_y)}^{\max(0, s_y)} \sqrt{1 + \phi'(s)^2} \, \dd s \le \sqrt{1 + K_*^2} \cdot \abs{s_y} \le \sqrt{1 + K_*^2} \cdot \abs{x - y}_{\R^2}
\]
is therefore an upper bound on $d_{\partial P}(x, y)$. Combining with (iii):
\[
\abs{\hV(x) - \hV(y)} \le \frac{4 \norm{V}_{\poly}}{d_1} \cdot \sqrt{1 + K_*^2} \cdot \abs{x - y}_{\R^2} = \frac{4 \sqrt{1+K_*^2}}{d_1} \norm{V}_{\poly} \cdot \abs{x - y}_{\R^2}.
\]
 
\textbf{Far case.} Suppose $\abs{x - y}_{\R^2} > r_*/2$. Trivially using (ii):
\[
\abs{\hV(x) - \hV(y)} \le 2 \norm{V}_{\poly} = \frac{4 \norm{V}_{\poly}}{r_*} \cdot \frac{r_*}{2} \le \frac{4 \norm{V}_{\poly}}{r_*} \cdot \abs{x - y}_{\R^2}.
\]
Combining the two cases yields the stated bound with $C_\partial = \max\bigl(\frac{4\sqrt{1+K_*^2}}{d_1}, \frac{4}{r_*}\bigr)$.
\end{proof}
 
\subsubsection{Explicit construction of \texorpdfstring{$\HV$}{HV}}
 
We first construct $\HV$ componentwise, and then combine. Write $\hV = (\hV^{(1)}, \hV^{(2)})$, so each $\hV^{(a)}: \partial P \to \R$ is a scalar function with
\[
\norm{\hV^{(a)}}_{L^\infty(\partial P)} \le \norm{V}_{\poly}, \qquad \Lip_{\R^2}(\hV^{(a)}) \le C_\partial \norm{V}_{\poly}
\]
(by Lemma~\ref{lem:hV-prop}(ii) and (iv)).

By the McShane extension theorem, there exists $\widetilde h^{(a)}: \R^2 \to \R$ with $\widetilde h^{(a)}|_{\partial P} = \hV^{(a)}$ and
\[
\Lip(\widetilde h^{(a)}) = \Lip_{\R^2}(\hV^{(a)}) \le C_\partial \norm{V}_{\poly}.
\]
McShane's formula $\widetilde h^{(a)}(x) := \inf_{y \in \partial P}\bigl( \hV^{(a)}(y) + \Lip_{\R^2}(\hV^{(a)}) \cdot \abs{x - y}_{\R^2}\bigr)$ is one explicit choice. A priori $\widetilde h^{(a)}$ need not be bounded; we truncate:
\[
\widetilde h^{(a)}_{\text{tr}}(x) := \max\bigl(-\norm{V}_{\poly}, \min(\widetilde h^{(a)}(x), \norm{V}_{\poly})\bigr).
\]
Truncation is a $1$-Lipschitz operation in the target, so
\[
\Lip(\widetilde h^{(a)}_{\text{tr}}) \le C_\partial \norm{V}_{\poly},
\qquad
\norm{\widetilde h^{(a)}_{\text{tr}}}_{L^\infty(\R^2)} \le \norm{V}_{\poly}.
\]
Moreover $\widetilde h^{(a)}_{\text{tr}}|_{\partial P} = \hV^{(a)}$ since $\hV^{(a)}$ takes values in $[-\norm{V}_{\poly}, \norm{V}_{\poly}]$ by Lemma~\ref{lem:hV-prop}(ii). Set $\widetilde \HV := (\widetilde h^{(1)}_{\text{tr}}, \widetilde h^{(2)}_{\text{tr}})$. Then
\[
\norm{\widetilde \HV}_{L^\infty(\R^2)} \le \sqrt 2 \, \norm{V}_{\poly}, \qquad \Lip(\widetilde \HV) \le \sqrt 2 \, C_\partial \norm{V}_{\poly}.
\]

Let $\sigma(x) := \dist(x, \partial P)$; as noted above, $\sigma$ is $1$-Lipschitz on $\R^2$. Choose a smooth function $\theta: [0, \infty) \to [0, 1]$ with
\[
\theta(s) = \begin{cases} 1 & 0 \le s \le \rho/2 \\ 0 & s \ge \rho \end{cases},
\qquad
\abs{\theta'} \le 3/\rho \text{ on } [0, \infty).
\]
Define $\eta(x) := \theta(\sigma(x))$. Then $\eta$ is Lipschitz on $\R^2$, with
\[
\eta \equiv 1 \text{ on } \{\sigma \le \rho/2\}, \qquad \eta \equiv 0 \text{ on } \{\sigma \ge \rho\}, \qquad \norm{\nabla \eta}_{L^\infty} \le 3/\rho \quad (\text{a.e.}).
\]
 
Now we define
\begin{equation}
\HV(x) := \eta(x) \cdot \widetilde \HV(x),
\label{eq:HV-def}
\end{equation}
which has the following properties. 
\begin{lemma}[Properties of $\HV$]
\label{lem:HV-prop}
With $\HV$ defined by~\Eqref{eq:HV-def}:
\begin{enumerate}[label=\textup{(\roman*)},leftmargin=2em,itemsep=2pt,topsep=2pt]
  \item $\HV \in W^{1,\infty}(\R^2; \R^2)$ with $\HV|_{\partial P(v)} = \hV$.
  \item $\supp \HV \subset \{x : \dist(x, \partial P(v)) \le \rho\}$.
  \item $\HV \equiv 0$ on $\{x : \dist(x, \partial\Omega) < d_0/4\}$.
  \item $\norm{\HV}_{L^\infty(\R^2)} \le \sqrt 2 \, \norm{V}_{\poly}$.
  \item $\norm{D\HV}_{L^\infty(\R^2)} \le C_1 \norm{V}_{\poly}$, where
  $
  C_1 := \frac{3 \sqrt 2}{\rho} + \sqrt 2 \, C_\partial = \frac{24\sqrt 2}{d_0} + \sqrt 2 \, C_\partial(d_1, \beta_0).
  $
\end{enumerate}
\end{lemma}
 
\begin{proof}
(i): $\HV = \eta \cdot \widetilde\HV$ with both factors $W^{1,\infty}$, so $\HV \in W^{1,\infty}$. On $\partial P$, $\sigma = 0$ so $\eta = 1$, hence $\HV|_{\partial P} = \widetilde\HV|_{\partial P} = \hV$.
 
(ii): $\HV(x) = 0$ wherever $\eta(x) = 0$, i.e.\ where $\sigma(x) \ge \rho$.
 
(iii): By (ii), every $x\in\supp H_V$ satisfies
$\dist(x,\partial P)\le\rho$. Hence
$
\dist(x,\partial\Omega)
\ge \dist(\partial P,\partial\Omega)-\rho
>\frac{d_0}{2}-\frac{d_0}{8}
>\frac{d_0}{4}.
$
 
(iv): $\abs{\HV(x)} = \eta(x) \abs{\widetilde\HV(x)} \le 1 \cdot \sqrt 2 \norm{V}_{\poly}$.
 
(v): By the Leibniz rule,
$
D\HV(x) = \widetilde\HV(x) \otimes \nabla\eta(x) + \eta(x) D\widetilde\HV(x) \quad (\text{a.e.\ in } \R^2).
$
Hence
\[
\abs{D\HV(x)}_{\op} \le \abs{\widetilde\HV(x)} \cdot \abs{\nabla\eta(x)} + \abs{\eta(x)} \cdot \abs{D\widetilde\HV(x)}_{\op} \le \sqrt 2 \norm{V}_{\poly} \cdot \frac{3}{\rho} + 1 \cdot \sqrt 2 \, C_\partial \norm{V}_{\poly} = C_1 \norm{V}_{\poly},
\]
where we used $\abs{\widetilde\HV \otimes \nabla\eta}_{\op} = \abs{\widetilde\HV} \cdot \abs{\nabla\eta}$. Substituting $\rho = d_0/8$ gives $C_1 = \frac{24\sqrt 2}{d_0} + \sqrt 2 C_\partial(d_1, \beta_0)$. \qedhere
\end{proof}

\subsection{The pullback PDE and its \texorpdfstring{$H^1$}{H1}-stability}
\label{sec:pullback}
 
From this section onward, fix $v$ with $P(v) \in \cA_{1/2}$, fix $V \in \R^{2n_{\mathrm{v}}}$, and let $\HV$ be as in Lemma~\ref{lem:HV-prop}. If $\norm{D\HV}_{L^\infty} = 0$ then $\HV$ is constant a.e.; combined with $\HV \equiv 0$ near $\partial\Omega$ (Lemma~\ref{lem:HV-prop}(iii)) this forces $\HV \equiv 0$ in $\Omega$, in which case $\Phi_t = \mathrm{id}$, $\widetilde u_t^f = u_v^f$, and all estimates below hold trivially. We may therefore assume $\norm{D\HV}_{L^\infty} > 0$ in what follows. Set
\begin{equation}
t_\star := \frac{1}{4 \norm{D\HV}_{L^\infty}}, \qquad \Phi_t(x) := x + t\HV(x) \ \text{for} \ \abs{t} < t_\star.
\label{eq:flow}
\end{equation}
The factor $4$ is chosen so that $\det D\Phi_t \ge 1/2$ uniformly; see Lemma~\ref{lem:flow}(iv).
Since $P(v)\in\mathcal A_{1/2}$ and the relaxed constraints are strict,
there exists a number $t_{\rm geom}=t_{\rm geom}(v,V)>0$ such that
$P(v+tV)\in\mathcal A_{1/2}$ and the inherited vertex order remains simple,
strictly convex, and counterclockwise for all $|t|<t_{\rm geom}$. Set
$
t_0:=\min\{t_\star,t_{\rm geom}\}.
$
All coefficient and pullback statements below are understood for $|t|<t_0$.
The estimates for $D\Phi_t$, $B_t$, and their remainders
require only $|t|<t_\star$.
 
\begin{lemma}[Flow map]
\label{lem:flow}
For $\abs{t} < t_0$:
\begin{enumerate}[label=\textup{(\roman*)},leftmargin=2em,itemsep=2pt,topsep=2pt]
\item $\Phi_t : \R^2 \to \R^2$ is a bi-Lipschitz homeomorphism.
\item $\Phi_t|_{\partial\Omega \cup \{\dist(x, \partial\Omega) < d_0/4\}} = \mathrm{id}$.
\item $\Phi_t(\Omega) = \Omega$.
\item $D\Phi_t = I + t D\HV$, $\det D\Phi_t \ge 1/2$ on $\R^2$ almost everywhere.
\item $\Phi_t(\partial P(v)) = \partial P(v + tV)$ and $\Phi_t(P(v)) = P(v+tV)$.
\end{enumerate}
\end{lemma}
 
\begin{proof}
(i), (iv): Since $D\HV$ is a $2 \times 2$ matrix, $\det(I + tD\HV) = 1 + t\tr(D\HV) + t^2 \det(D\HV)$ exactly. For $\abs{t} < t_\star = 1/(4\norm{D\HV}_\infty)$, $\norm{tD\HV}_{\op} < 1/4$, so $I + tD\HV$ is invertible with $\norm{(I + tD\HV)^{-1}}_{\op} \le 1/(1 - 1/4) = 4/3$ (Neumann series); hence the singular values of $D\Phi_t = I + tD\HV$ lie in $[3/4, 5/4]$. Therefore $\det D\Phi_t = \sigma_1 \sigma_2 \in [9/16, 25/16]$, so $\det D\Phi_t \ge 9/16 > 1/2$. Since $\Phi_t = \mathrm{id} + t\HV$ with $\HV \in W^{1,\infty}(\R^2; \R^2)$ and $\norm{t D\HV}_\op < 1$, $\Phi_t$ is a bi-Lipschitz homeomorphism of $\R^2$ onto itself with explicit inverse Lipschitz constant $\le 4/3$ by a standard contraction mapping argument.
 
(ii): By Lemma~\ref{lem:HV-prop}(iii), $\HV \equiv 0$ on $\{x : \dist(x, \partial\Omega) < d_0/4\}$, which is an open neighborhood of $\partial\Omega$. Hence $\Phi_t = \id$.
 
(iii): Since $\Phi_t$ is a homeomorphism of $\R^2$ and equals the identity on a neighborhood of $\partial\Omega$, it maps $\partial\Omega$ to itself. By Jordan's theorem, $\R^2 \setminus \partial\Omega$ has exactly two components --- the bounded one $\Omega$ and the unbounded one $\R^2 \setminus \overline\Omega$ --- and $\Phi_t$ permutes these components. Since $\Phi_t = \id$ on a neighborhood of $\partial\Omega$ in particular at points of $\Omega$ near $\partial\Omega$, $\Phi_t$ fixes the bounded component setwise: $\Phi_t(\Omega) = \Omega$.
 
(v): Since $|t|<t_0\le t_{\rm geom}$, the inherited perturbed chain
$(v_i+tV_i)_i$ remains simple, strictly convex, and counterclockwise, and
$P(v+tV)\in\mathcal A_{1/2}$. Thus $P(v+tV)$ is well-defined. The choice~\Eqref{eq:hV-def} of $\hV$ together
with $\eta = 1$ on $\partial P$ and $\HV|_{\partial P} = \hV$
(Lemma~\ref{lem:HV-prop}(i)) ensures that, for
$x = (1-s)v_j + s v_{j+1}$ on edge $j$ of $\partial P(v)$,
\[
\Phi_t(x) = x + t \hV(x) = (1-s)(v_j + tV_j) + s (v_{j+1} + tV_{j+1}),
\]
which is exactly the corresponding point on edge $j$ of $\partial P(v+tV)$.
Hence $\Phi_t(\partial P(v))=\partial P(v+tV)$. Since $\Phi_t$ is a
homeomorphism, it maps the bounded component of
$\mathbb R^2\setminus\partial P(v)$ onto the bounded component of
$\mathbb R^2\setminus\partial P(v+tV)$, hence $\Phi_t(P(v))=P(v+tV)$. \qedhere
\end{proof}
 
\subsubsection{Weak formulation of the pulled-back equation}
For $f \in H^{1/2}(\partial\Omega)$, let $u_v^f \in H^1(\Omega)$ be the unique weak solution of
\begin{equation}
\int_\Omega \gamma_{P(v)}(x) \, \nabla u_v^f(x) \cdot \nabla \psi(x) \, \dd x = 0 \quad \text{for all } \psi \in H_0^1(\Omega), \qquad u_v^f|_{\partial\Omega} = f.
\label{eq:primal-pde}
\end{equation}

For $|t|<t_0$, let $u_{v+tV}^f \in H^1(\Omega)$ be the solution of~\Eqref{eq:primal-pde} with coefficient $\gammatv$ in place of $\gammav$. Define the pullback
\begin{equation}
\widetilde u_t^f(x) := u_{v+tV}^f \circ \Phi_t(x), \qquad x \in \Omega.
\label{eq:pullback}
\end{equation}
By Lemma~\ref{lem:flow}(ii,iii), $\widetilde u_t^f \in H^1(\Omega)$ with $\widetilde u_t^f|_{\partial\Omega} = f$.
 
\begin{lemma}[Weak equation for the pullback]
\label{lem:pullback-weak}
For $|t|<t_0$, $\widetilde u_t^f$ satisfies
\begin{equation}
\int_\Omega \gammav(x) \, B_t(x) \, \nabla \widetilde u_t^f(x) \cdot \nabla \psi(x) \, \dd x = 0 \qquad \text{for all } \psi \in H_0^1(\Omega),
\label{eq:pullback-weak}
\end{equation}
where
\begin{equation}
B_t(x) := \det(D\Phi_t(x)) \cdot D\Phi_t(x)^{-1} \cdot D\Phi_t(x)^{-T}.
\label{eq:Bt-def}
\end{equation}
Moreover $B_t$ is symmetric, $B_0 = I$, and for $\abs{t} < t_0$,
\[
\norm{B_t}_{L^\infty(\Omega)} \le 3, \qquad B_t(x) \ge \tfrac{1}{3} I \text{ in the symmetric matrix sense}.
\]
\end{lemma}
 
\begin{proof}
 
The weak formulation for $u_{v+tV}^f$ says: for all $\varphi \in H_0^1(\Omega)$,
\[
\int_\Omega \gammatv(y) \, \nabla u_{v+tV}^f(y) \cdot \nabla \varphi(y) \, \dd y = 0.
\]
Change variables $y = \Phi_t(x)$, $\dd y = \det(D\Phi_t(x)) \, \dd x$ (with $\det D\Phi_t > 0$ a.e.\ by Lemma~\ref{lem:flow}(iv)). Set $\varphi = \psi \circ \Phi_t^{-1}$ for $\psi \in H_0^1(\Omega)$. This is permissible since $\Phi_t$ fixes $\partial\Omega$ and $\Phi_t(\Omega) = \Omega$, and composition with bi-Lipschitz maps preserves $H_0^1$. Gradients transform via
\[
\nabla u_{v+tV}^f(y) = \nabla u_{v+tV}^f(\Phi_t(x)) = D\Phi_t(x)^{-T} \nabla \widetilde u_t^f(x),
\qquad
\nabla \varphi(y) = D\Phi_t(x)^{-T} \nabla \psi(x).
\]
Also, $\gammatv(\Phi_t(x)) = \gammav(x)$ because $\Phi_t(P(v)) = P(v+tV)$ (Lemma~\ref{lem:flow}(v)) and $\gamma$ is constant on $P$ and on $\Omega \setminus P$. Substituting,
\begin{align*}
0
&= \int_\Omega \gammatv(\Phi_t(x)) \cdot \bigl[D\Phi_t(x)^{-T} \nabla \widetilde u_t^f(x)\bigr] \cdot \bigl[D\Phi_t(x)^{-T} \nabla \psi(x)\bigr] \cdot \det(D\Phi_t(x)) \, \dd x \\
&= \int_\Omega \gammav(x) \cdot \underbrace{\det(D\Phi_t(x)) \cdot D\Phi_t(x)^{-1} D\Phi_t(x)^{-T}}_{= B_t(x)} \cdot \nabla \widetilde u_t^f(x) \cdot \nabla \psi(x) \, \dd x,
\end{align*}
using $(M^{-T} a) \cdot (M^{-T} b) = a \cdot (M^{-1} M^{-T}) b$. Symmetry of $B_t$: by construction $B_t = \det(D\Phi_t) (D\Phi_t^T D\Phi_t)^{-1}$, and $(D\Phi_t^T D\Phi_t)^{-1}$ is symmetric.
 
For the $L^\infty$ and coercivity bounds: for $\abs{t} < t_\star = 1/(4\norm{D\HV}_\infty)$, we have $\norm{t D\HV}_{\op} < 1/4$, so
\[
\norm{D\Phi_t}_{\op} \le 1 + 1/4 = 5/4,
\qquad
\norm{D\Phi_t^{-1}}_{\op} \le \frac{1}{1 - 1/4} = 4/3.
\]
Hence the singular values of $D\Phi_t$ lie in $[3/4, 5/4]$. For $x \in \R^2$,
\[
\tfrac{4}{5} \abs{x} \le \abs{D\Phi_t^{-T} x} \le \tfrac{4}{3} \abs{x},
\]
whence $\abs{D\Phi_t^{-T} x}^2 \in [16/25, 16/9] \cdot \abs{x}^2$. Combined with $\det(D\Phi_t) \in [9/16, 25/16]$, $x^T B_t x = \det(D\Phi_t) \cdot \abs{D\Phi_t^{-T} x}^2$ gives
\[
\tfrac{9}{16} \cdot \tfrac{16}{25} \abs{x}^2 \le x^T B_t x \le \tfrac{25}{16} \cdot \tfrac{16}{9} \abs{x}^2,
\qquad \text{i.e.\ } \quad \tfrac{9}{25} \abs{x}^2 \le x^T B_t x \le \tfrac{25}{9} \abs{x}^2.
\]
Since $9/25>1/3$ and $25/9<3$, we use the slightly looser bounds
$\frac13I\le B_t\le3I$ below. \qedhere
\end{proof}
 
\subsubsection{Taylor expansion of $B_t$ with explicit remainder}

\begin{lemma}[Taylor expansion of $B_t$]
\label{lem:Bt-expansion}
For $\abs{t} < t_0$, $B_t = I + t A + R_t$, where
\[
A(x) := (\nabla \cdot \HV(x)) \cdot I - D\HV(x) - D\HV(x)^T,
\]
and
\[
\abs{R_t(x)}_{\op} \le C_R \cdot t^2 \cdot \norm{D\HV}_{L^\infty}^2 \quad \text{for almost every } x \in \Omega,
\]
where $C_R$ is an absolute constant. An explicit choice is $C_R = 30$.
\end{lemma}
 
\begin{proof}
Fix $x$ and let $M = D\HV(x)$, $m := \abs{M}_{\op}$. For $\abs{t} < t_\star = 1/(4\norm{D\HV}_\infty) \le 1/(4m)$, in particular $\abs{t} m < 1/4$, so $t^2 m^2 < 1/16$ and $\abs{t}^3 m^3 < t^2 m^2/4$. The Neumann series gives
\[
(I + tM)^{-1} = I - tM + t^2 M^2 \sum_{k=0}^\infty (-tM)^k,
\]
so $(I + tM)^{-1} = I - tM + E_1(t)$ with $\abs{E_1(t)}_{\op} \le t^2 m^2 / (1 - \abs{t} m) \le \tfrac{4}{3} t^2 m^2$. Similarly $(I + tM)^{-T} = I - tM^T + E_1'$ with $\abs{E_1'}_{\op} \le \tfrac{4}{3} t^2 m^2$.
 
For the determinant in $\R^{2 \times 2}$,
$
\det(I + tM) = 1 + t \tr(M) + t^2 \det(M).
$ Therefore
\begin{align*}
B_t &= (1 + t\tr M + t^2 \det M) \cdot (I - tM + E_1) \cdot (I - tM^T + E_1') \\
&= I + t(\tr M \cdot I - M - M^T) + R_t,
\end{align*}
where $R_t$ collects all terms of $O(t^2)$ and higher, including pure $t^2$-coefficients ($t^2 MM^T$, $-t^2 \tr M (M+M^T)$, $t^2 \det M \cdot I$), the Neumann remainders $E_1, E_1'$, products $tM \cdot E_1'$ and $tE_1 \cdot M^T$, etc.\ Bounding each term using $\abs{\tr M} \le 2m$, $\abs{\det M} \le m^2$, $\abs{MM^T}_\op \le m^2$, $\abs{M+M^T}_\op \le 2m$, the Neumann bounds, and $\abs{t}m < 1/4$ to absorb higher powers of $\abs{t}m$:
\begin{itemize}[leftmargin=1.5em,itemsep=1pt,topsep=2pt]
\item $\abs{t^2 MM^T}_\op \le t^2 m^2$;
\item $\abs{t^2 \tr M (M+M^T)}_\op \le 4 t^2 m^2$;
\item $\abs{t^2 \det M \cdot I}_\op \le t^2 m^2$;
\item $\abs{(1 + t\tr M + t^2 \det M) (E_1 + E_1' - tME_1' - tE_1 M^T + E_1 E_1')}_\op \le (1 + 1)(\tfrac{8}{3} t^2 m^2 + 2 \cdot \tfrac{4}{3} t^2 m^2 + (\tfrac{4}{3})^2 t^4 m^4) \le \tfrac{98}{9} t^2 m^2 \le 11 t^2 m^2$;
\item higher cross terms ($t^3 \tr M MM^T$, $t^3 \det M (M+M^T)$, $t^4 \det M MM^T$) bounded by $\le 3 t^2 m^2$ in total using $\abs{t}m < 1/4$.
\end{itemize}
Thus $|R_t|_{\mathrm{op}}\le20t^2m^2$, and the stated choice
$C_R=30$ is valid.
\qedhere
\end{proof}
 
\subsubsection{\texorpdfstring{$H^1$}{H1}-stability of the pullback}
 
\begin{lemma}[$H^1$-stability]
\label{lem:H1-stab}
For $\abs{t} < t_0$,
\[
\norm{\widetilde u_t^f - u_v^f}_{H^1(\Omega)} \le C_{\en}' \cdot \abs{t} \cdot \norm{D\HV}_{L^\infty} \cdot \norm{f}_{H^{1/2}(\partial\Omega)},
\]
where $C_{\en}' = C_{\en}'(\Omega, \gamma_{\min}, \gamma_{\max})$ is independent of $t, \HV, P, V$.
\end{lemma}
 
\begin{proof}
Set $w_t := \widetilde u_t^f - u_v^f$ and $w_t \in H_0^1(\Omega)$ since both have trace $f$ on $\partial\Omega$. Subtracting~\Eqref{eq:pullback-weak} from the $t=0$ version:
\begin{align*}
0 &= \int_\Omega \gammav B_t \nabla \widetilde u_t^f \cdot \nabla \psi \, \dd x - \int_\Omega \gammav \nabla u_v^f \cdot \nabla \psi \, \dd x \\
&= \int_\Omega \gammav B_t \nabla w_t \cdot \nabla \psi \, \dd x + \int_\Omega \gammav (B_t - I) \nabla u_v^f \cdot \nabla \psi \, \dd x.
\end{align*}
Rearranging with $\psi = w_t \in H_0^1(\Omega)$:
\[
\int_\Omega \gammav B_t \nabla w_t \cdot \nabla w_t \, \dd x = - \int_\Omega \gammav (B_t - I) \nabla u_v^f \cdot \nabla w_t \, \dd x.
\]
By Lemma~\ref{lem:pullback-weak}, $\gammav B_t \ge \tfrac{\gamma_{\min}}{3} I$ pointwise in the symmetric matrix sense. Hence the LHS dominates
\[
\tfrac{\gamma_{\min}}{3} \int_\Omega \abs{\nabla w_t}^2 \, \dd x = \tfrac{\gamma_{\min}}{3} \norm{\nabla w_t}_{L^2}^2.
\]
By Cauchy--Schwarz, the RHS satisfies
\[
\abs{\text{RHS}} \le \gamma_{\max} \norm{B_t - I}_{L^\infty} \norm{\nabla u_v^f}_{L^2} \norm{\nabla w_t}_{L^2}.
\]
Combining and dividing by $\norm{\nabla w_t}_{L^2}$:
\[
\norm{\nabla w_t}_{L^2} \le \frac{3 \gamma_{\max}}{\gamma_{\min}} \norm{B_t - I}_{L^\infty} \norm{\nabla u_v^f}_{L^2}.
\]
By Lemma~\ref{lem:Bt-expansion} and $\abs{t} < t_\star = 1/(4\norm{D\HV}_\infty)$, $\norm{B_t - I}_{L^\infty} \le \abs{t} \norm{A}_{L^\infty} + C_R t^2 \norm{D\HV}_\infty^2 \le 4 \abs{t} \norm{D\HV}_\infty + C_R \abs{t} \cdot \tfrac{1}{4} \norm{D\HV}_\infty = (4 + C_R/4) \abs{t} \norm{D\HV}_\infty$ (using $\abs{t} \norm{D\HV}_\infty < 1/4$; see~\Eqref{eq:A-pointwise} below for $\norm{A}_\infty \le 4 \norm{D\HV}_\infty$). Combined with the standard energy estimate $\norm{\nabla u_v^f}_{L^2} \le C_\en \norm{f}_{H^{1/2}}$ (Lemma~\ref{lem:energy}) and Poincar\'e inequality $\norm{w_t}_{L^2} \le C_\Omega \norm{\nabla w_t}_{L^2}$:
\[
\norm{w_t}_{H^1} \le (1 + C_\Omega) \norm{\nabla w_t}_{L^2} \le (1 + C_\Omega) \cdot \frac{3 \gamma_{\max}}{\gamma_{\min}} \cdot (4 + C_R/4) \cdot C_\en \cdot \abs{t} \norm{D\HV}_\infty \norm{f}_{H^{1/2}}. \qedhere
\]
\end{proof}
 
\subsubsection{Differentiation at \texorpdfstring{$t = 0$}{t=0}}
 
\begin{lemma}[Distributed shape derivative formula]
\label{lem:dist-formula}
For every $f, g \in H^{1/2}(\partial\Omega)$,
\begin{equation}
\inner{\DLa_i(v)[V] f}{g}_{\partial\Omega} = \int_\Omega \gammav(x) \, A_{\HV}(x) \, \nabla u_v^f(x) \cdot \nabla u_v^g(x) \, \dd x,
\label{eq:dist-formula}
\end{equation}
where $A_{\HV}(x) = (\nabla \cdot \HV(x)) I - D\HV(x) - D\HV(x)^T$.
\end{lemma}
 
\begin{proof}
We first express both pairings on the reference domain. The weak DtN pairing can be formulated as follows,
\[
\inner{\Lambda_{\gammatv} f}{g}_{\partial\Omega} = \int_\Omega \gammatv(y) \nabla u_{v+tV}^f(y) \cdot \nabla u_{v+tV}^g(y) \, \dd y.
\]
Change variables $y = \Phi_t(x)$ as in the proof of Lemma~\ref{lem:pullback-weak}:
\begin{equation}
\inner{\Lambda_{\gammatv} f}{g}_{\partial\Omega} = \int_\Omega \gammav(x) \, B_t(x) \, \nabla \widetilde u_t^f(x) \cdot \nabla \widetilde u_t^g(x) \, \dd x.
\label{eq:pairing-pullback}
\end{equation}
The $t=0$ version is
\begin{equation}
\inner{\Lambda_{\gammav} f}{g}_{\partial\Omega} = \int_\Omega \gammav \nabla u_v^f \cdot \nabla u_v^g \, \dd x.
\label{eq:pairing-v}
\end{equation}
 
By a telescoping identity, subtracting~\Eqref{eq:pairing-v} from~\Eqref{eq:pairing-pullback} and rearranging,
\begin{equation}
    \label{eq:tele_id}
\begin{aligned}
\inner{(\Lambda_{\gammatv} - \Lambda_{\gammav}) f}{g}_{\partial\Omega}
= \ &\underbrace{\int_\Omega \gammav (B_t - I) \nabla u_v^f \cdot \nabla u_v^g \, \dd x}_{=: T_1(t)} \\
&+ \underbrace{\int_\Omega \gammav B_t \nabla (\widetilde u_t^f - u_v^f) \cdot \nabla u_v^g \, \dd x}_{=: T_2(t)} \\
&+ \underbrace{\int_\Omega \gammav B_t \nabla \widetilde u_t^f \cdot \nabla (\widetilde u_t^g - u_v^g) \, \dd x}_{=: T_3(t)}. 
\end{aligned}
\end{equation} 

By Lemma~\ref{lem:Bt-expansion},
\[
T_1(t) = t \int_\Omega \gammav A_{\HV} \nabla u_v^f \cdot \nabla u_v^g \, \dd x + \int_\Omega \gammav R_t \nabla u_v^f \cdot \nabla u_v^g \, \dd x.
\]
The second term is bounded in absolute value by $\gamma_{\max} C_R t^2 \norm{D\HV}_\infty^2 \norm{\nabla u_v^f}_{L^2} \norm{\nabla u_v^g}_{L^2} = O(t^2)$ as $t \to 0$.
 
We now prove that $T_2$ is $O(t^2)$. The key observation is that $u_v^g \in H^1(\Omega)$ satisfies the weak equation~\Eqref{eq:primal-pde} for coefficient $\gammav$, i.e.\
\begin{equation}
\int_\Omega \gammav \nabla u_v^g \cdot \nabla \psi \, \dd x = 0 \quad \text{for all } \psi \in H_0^1(\Omega).
\label{eq:harmonic-g}
\end{equation}
Since $\widetilde u_t^f - u_v^f \in H_0^1(\Omega)$, we may use $\psi = \widetilde u_t^f - u_v^f$ in~\Eqref{eq:harmonic-g} to obtain
\begin{equation}
\int_\Omega \gammav \nabla (\widetilde u_t^f - u_v^f) \cdot \nabla u_v^g \, \dd x = 0.
\label{eq:zero1}
\end{equation}
Subtracting~\Eqref{eq:zero1} from $T_2(t)$ gives
\begin{align*}
T_2(t) = \int_\Omega \gammav (B_t - I) \nabla (\widetilde u_t^f - u_v^f) \cdot \nabla u_v^g \, \dd x.
\end{align*}
Therefore, by Cauchy--Schwarz, Lemma~\ref{lem:Bt-expansion} ($\norm{B_t - I}_{L^\infty} = O(t)$), and Lemma~\ref{lem:H1-stab} ($\norm{\nabla(\widetilde u_t^f - u_v^f)}_{L^2} = O(t)$),
\[
\abs{T_2(t)} \le \gamma_{\max} \norm{B_t - I}_{L^\infty} \norm{\nabla(\widetilde u_t^f - u_v^f)}_{L^2} \norm{\nabla u_v^g}_{L^2} = O(t) \cdot O(t) \cdot O(1) = O(t^2).
\]

Since $\widetilde u_t^g-u_v^g\in H_0^1(\Omega)$,
Lemma~\ref{lem:pullback-weak}, tested with this function, gives
$T_3(t)=0$.
 
Dividing the identity in \Eqref{eq:tele_id} by $t$:
\[
\frac{1}{t} \inner{(\Lambda_{\gammatv} - \Lambda_{\gammav}) f}{g}_{\partial\Omega} = \int_\Omega \gammav A_{\HV} \nabla u_v^f \cdot \nabla u_v^g \, \dd x + O(t).
\]
Sending $t \to 0^+$ and using the Fr\'echet differentiability of
$v\mapsto\Lambda_{\gamma_{P(v)}}$ at $v$, as established in
\cite[Lemma 4.4, Corollary 4.5 and Remark 2.3]{beretta2017differentiability}, the LHS
converges to $\inner{\DLa_i(v)[V] f}{g}_{\partial\Omega}$,
proving~\Eqref{eq:dist-formula}. \qedhere
\end{proof}
 
\subsection{The norm bound}
\label{sec:norm-bound}
 
\begin{lemma}[Pointwise bound on $A_{\HV}$]
\label{lem:A-pointwise}
For $M \in \R^{2\times 2}$, $\abs{\tr(M) I - M - M^T}_{\op} \le 4 \abs{M}_{\op}$. In particular,
\begin{equation}
\norm{A_{\HV}}_{L^\infty(\Omega)} \le 4 \norm{D\HV}_{L^\infty(\Omega)} \le 4 C_1 \norm{V}_{\poly}.
\label{eq:A-pointwise}
\end{equation}
\end{lemma}
 
\begin{proof}
$\abs{\tr(M) I}_{\op} = \abs{\tr(M)}$. In $\R^{2\times 2}$, $\abs{\tr M} \le 2 \abs{M}_{\op}$: indeed, $\abs{\tr M} = \abs{\lambda_1 + \lambda_2}$ where $\lambda_1, \lambda_2$ are the eigenvalues of $M$, and $\abs{\lambda_1 + \lambda_2} \le \abs{\lambda_1} + \abs{\lambda_2} \le 2 \rho(M) \le 2 \abs{M}_{\op}$. Hence
\[
\abs{\tr(M) I - M - M^T}_{\op} \le \abs{\tr M} + \abs{M}_{\op} + \abs{M^T}_{\op} \le 2\abs{M}_{\op} + \abs{M}_{\op} + \abs{M}_{\op} = 4 \abs{M}_{\op}. \qedhere
\]
\end{proof}
 
\begin{proof}[Proof of Theorem~\ref{thm:main}]
By Lemma~\ref{lem:dist-formula} and Cauchy--Schwarz,
\begin{align*}
\abs{\inner{\DLa_i(v)[V] f}{g}_{\partial\Omega}}
&\le \gamma_{\max} \norm{A_{\HV}}_{L^\infty} \norm{\nabla u_v^f}_{L^2} \norm{\nabla u_v^g}_{L^2} \\
&\le \gamma_{\max} \cdot 4 C_1 \norm{V}_{\poly} \cdot C_\en \norm{f}_{H^{1/2}} \cdot C_\en \norm{g}_{H^{1/2}}
\qquad \text{(Lemmas~\ref{lem:A-pointwise},~\ref{lem:energy},~\ref{lem:HV-prop}(v))} \\
&= L_\infty \cdot \norm{V}_{\poly} \cdot \norm{f}_{H^{1/2}} \norm{g}_{H^{1/2}},
\end{align*}
with
\begin{equation}
L_\infty := 4 \gamma_{\max} C_1 C_\en^2, \qquad C_1 = \frac{24\sqrt 2}{d_0} + \sqrt 2\, C_\partial(d_1, \beta_0), \qquad C_\partial = \max\!\left(\frac{4\sqrt{1+K_*^2}}{d_1}, \frac{4}{r_*}\right),
\label{eq:Linfty}
\end{equation}
where $r_* = \frac{d_1}{16}\sin^2(\beta_0/4)$, $K_* = \cot(\beta_0/4)$ from Lemma~\ref{lem:A4_from_A23_conv}, and $C_\en = C_\en(\Omega, \gamma_{\min}, \gamma_{\max})$ from Lemma~\ref{lem:energy}. Taking the supremum over $\norm{f}_{H^{1/2}} = \norm{g}_{H^{1/2}} = 1$ yields $\norm{\DLa_i(v)[V]}_{\cY} \le L_\infty \norm{V}_{\poly}$. Every constant entering $L_\infty$ depends only on $(\Omega, \kappa, d_0, d_1, \beta_0)$; none of them depends on $n_{\mathrm{v}}$, on the chart index $i$, or on $v$. \qedhere
\end{proof}

\section{Proof of Lemma~\ref{lem:patch-radius}}\label{sec:proof_rho_s}

\begin{lemma}[Positive-turning criterion]\label{lem:positive-turning}
Let $q_0,\dots,q_{n-1}\in\mathbb R^2$, $n\ge3$, set $q_n=q_0$, and
$e_i:=q_{i+1}-q_i$. Assume $e_i\ne0$ for all $i$. Let
$\delta_i\in(-\pi,\pi)$ be the principal signed angle from $e_{i-1}$ to
$e_i$, with cyclic indices. Suppose
\[
0<\delta_i<\pi\quad(i=0,\dots,n-1),
\qquad
\sum_{i=0}^{n-1}\delta_i=2\pi.
\]
Then the ordered polygonal path traversing $[q_i,q_{i+1}]$ is a simple,
strictly convex polygonal Jordan curve, traversed counterclockwise.
\end{lemma}

\begin{proof}
Choose an argument $a_0$ of $e_0$. Define recursively
$a_j:=a_{j-1}+\delta_j$ for $j=1,\dots,n$, where $\delta_n:=\delta_0$.
Since $\delta_j$ is the signed angle from $e_{j-1}$ to $e_j$, $a_j$ is an
unwrapped argument of $e_j$ for $j=0,\dots,n-1$, and $a_n$ is another
argument of $e_0$. The assumption $\sum_i\delta_i=2\pi$ gives
\[
a_0<a_1<\cdots<a_n=a_0+2\pi.
\]

Fix an oriented edge $[q_i,q_{i+1}]$. Below, $a_{i+r}$ denotes the cyclically
unwrapped argument, so that $a_i<a_{i+1}<\cdots<a_{i+n}=a_i+2\pi$.
Write $\ell_i:=|e_i|$.
For $1\le m\le n$,
\[
e_i\times(q_{i+m}-q_i)
=
\sum_{r=0}^{m-1}e_i\times e_{i+r}
=
\ell_i\sum_{r=0}^{m-1}\ell_{i+r}\sin(a_{i+r}-a_i).
\]
Set $s_r:=\ell_{i+r}\sin(a_{i+r}-a_i)$. Then $s_0=0$, $s_1>0$, the
sequence $s_r$ is nonnegative until the angle difference reaches $\pi$, and
nonpositive afterwards. Moreover
\[
s_{n-1}=\ell_{i+n-1}\sin(2\pi-\delta_i)<0.
\]
Since $q_{i+n}=q_i$, the full sum satisfies $\sum_{r=0}^{n-1}s_r=0$.
For every $2\le m\le n-1$,
\[
\sum_{r=0}^{m-1}s_r>0.
\]
Indeed, before the sign change the partial sum contains $s_1>0$ and no
negative term; after the sign change it equals the negative of a nonpositive
tail containing the strictly negative term $s_{n-1}$.

Therefore every vertex other than $q_i,q_{i+1}$ lies strictly to the left of
the oriented line through $[q_i,q_{i+1}]$. Thus each edge is an exposed
supporting edge of $\operatorname{conv}\{q_0,\dots,q_{n-1}\}$, with all other
vertices strictly on the same side. This excludes nonadjacent intersections and
duplicate nonadjacent vertices, and the ordered path is precisely the
counterclockwise boundary of its convex hull. Hence it is simple and strictly
convex.
\end{proof}

\begin{proof}[Proof of Lemma~\ref{lem:patch-radius}]
Let $P(v)\in\cA$ and $\norm{V}_{\poly}\le r_0$. Set
\[
q_i(t):=v_i+tV_i,\qquad e_i(t):=q_{i+1}(t)-q_i(t),
\qquad t\in[0,1].
\]
Before defining $P(v+tV)$, regard
\[
\Gamma_t:=\bigcup_i[q_i(t),q_{i+1}(t)]
\]
only as the ordered closed polygonal path traversing these segments in the
inherited cyclic order.

\textbf{Part 1: simplicity, strict convexity, and orientation.}
The local edge-length estimate gives
\[
|e_i(t)|\ge d_1-2r_0>\frac{d_1}{2}>0.
\]
Let $\theta_i(t)\in[0,\pi]$ be the Euclidean angle between
$q_{i-1}(t)-q_i(t)$ and $q_{i+1}(t)-q_i(t)$. For each $i$, the same
unit-vector estimate used below with $tV$ in place of $V$ gives
\[
\theta_i(t)\in
\Bigl(\frac{\beta_0}{2},\,\pi-\frac{\beta_0}{2}\Bigr)
\qquad\forall i,\ t\in[0,1].
\]
Thus consecutive edges are never parallel or antiparallel.

The function $e_{i-1}(t)\times e_i(t)$ is continuous and never vanishes.
Since $P(v)\in\mathcal A$ is strictly convex and counterclockwise,
$e_{i-1}(0)\times e_i(0)>0$. Hence
\[
e_{i-1}(t)\times e_i(t)>0
\qquad\forall i,\ t\in[0,1].
\]
Thus every turn is a left turn. Define the positive signed exterior turn
\[
\delta_i(t):=\pi-\theta_i(t)\in
\Bigl(\frac{\beta_0}{2},\,\pi-\frac{\beta_0}{2}\Bigr)\subset(0,\pi).
\]

Pick a continuous argument $a_0(t)$ of $e_0(t)$. For
$j=1,\dots,n_{\mathrm v}$, define
\[
a_j(t):=a_{j-1}(t)+\delta_j(t),
\qquad
\delta_{n_{\mathrm v}}(t):=\delta_0(t).
\]
By induction, $a_j(t)$ is an unwrapped argument of $e_j(t)$ for
$j=0,\dots,n_{\mathrm v}-1$, while $a_{n_{\mathrm v}}(t)$ is another argument
of $e_0(t)$. Hence
\[
a_{n_{\mathrm v}}(t)-a_0(t)
=
\sum_{i=0}^{n_{\mathrm v}-1}\delta_i(t)\in 2\pi\mathbb Z.
\]
The left-hand side is continuous in $t$, so the integer is constant. At
$t=0$, the polygon is strictly convex and counterclockwise, hence the total
turn is $2\pi$. Therefore
\[
\sum_i\delta_i(t)=2\pi\qquad\forall t\in[0,1].
\]
By Lemma~\ref{lem:positive-turning}, $\Gamma_t$ is a simple strictly convex
counterclockwise polygonal Jordan curve for every $t\in[0,1]$. In particular,
$P(v+V)$ is well-defined and
$P(v+V)=\operatorname{conv}\{v_i+V_i\}$.

\textbf{Part 2: verification of (A1$'$)--(A3$'$).}
For (A1$'$), every $x\in P(v+V)$ can be written as
\[
x=\sum_i\lambda_i(v_i+V_i),\qquad \lambda_i\ge0,\quad \sum_i\lambda_i=1.
\]
Set $y:=\sum_i\lambda_i v_i\in P(v)$. Then
$
|x-y|\le \sum_i\lambda_i |V_i|\le \|V\|_{\poly}\le r_0.
$
Hence $P(v+V)\subset P(v)+\overline B_2(0,r_0)$, and therefore
\[
\dist(P(v+V), \partial\Omega) \ge \dist(P(v), \partial\Omega) - r_0
\ge d_0-r_0>d_0/2.
\]

For (A2$'$), the perturbed edge lengths satisfy
\[
|v_{i+1}+V_{i+1}-v_i-V_i|
\ge |v_{i+1}-v_i|-|V_{i+1}-V_i|
\ge d_1-2r_0>d_1/2.
\]

For (A3$'$), we give the unit-vector estimate used in Part 1. Fix a vertex
$v_i$. Set
\[
u:=v_{i-1}-v_i,\qquad w:=v_{i+1}-v_i,\qquad
u':=u+\Delta_u,\quad w':=w+\Delta_w,
\]
with $\Delta_u:=V_{i-1}-V_i$ and $\Delta_w:=V_{i+1}-V_i$. Then
$|\Delta_u|,|\Delta_w|\le2r_0$, and the lower edge bound above gives
$|u'|,|w'|\ge d_1/2$. With hats denoting unit vectors,
$
|\widehat u'-\widehat u|\le \frac{8r_0}{d_1},
\qquad
|\widehat w'-\widehat w|\le \frac{8r_0}{d_1}.
$
Thus
$
|\widehat u'\cdot\widehat w'-\widehat u\cdot\widehat w|
\le \frac{16r_0}{d_1}.
$
Since the original angle $\beta_i(v)$ lies in $[\beta_0,\pi-\beta_0]$,
$|\cos\beta_i(v)|\le\cos\beta_0$. Moreover,
\[
\cos(\beta_0/2)-\cos\beta_0
=\int_{\beta_0/2}^{\beta_0}\sin t\,\dd t
\ge \frac{\beta_0^2}{2\pi}.
\]
The choice $r_0\le d_1\beta_0^2/128$ gives
$16r_0/d_1<\beta_0^2/(2\pi)$, hence the perturbed adjacent-edge angle lies in
$(\beta_0/2,\pi-\beta_0/2)$. Since Part 1 proves the perturbed chain is a
strictly convex polygon, these adjacent-edge angles are its interior angles.

Combining the three constraints, $r_0=\tfrac{1}{128}\min(d_0,d_1,d_1\beta_0^2)$
ensures $P(v+V)\in\cA_{1/2}$.
\end{proof}

\section{Proof of Lemma~\ref{lem:chart-M}}
\label{sec:chart-M}

\begin{lemma}\label{lem:L1_to_H_poly}
There exists a constant $C_H>0$, depending only on the a priori data, such that for all
$P,Q\in\mathcal A_{1/2}$,
\[
d_H(\partial P,\partial Q)\ \le\ C_H\|\gamma_P-\gamma_Q\|_{L^1(\Omega)}^{1/2}.
\]
\end{lemma}

\begin{proof}[Proof of Lemma~\ref{lem:L1_to_H_poly}]
Let $\delta:=d_H(\partial P,\partial Q)$. It suffices to prove the case when $\delta>0$. By symmetry, it suffices to treat the case in which the Hausdorff maximum is attained from $\partial P$ to $\partial Q$. Thus choose $x\in\partial P$ such that $\dist(x,\partial Q)=\delta$. If the maximum is attained in the opposite direction, interchange the roles of $P$ and $Q$. Consider the open ball $B_2(x,\delta)$, $B_2(x,\delta)\cap\partial Q=\varnothing$, hence $\chi_Q$ is constant on $B_2(x,\delta)$.

By Lemma~\ref{lem:A4_from_A23_conv}, there exist rigid coordinates
centered at $x$ and a $K_*$-Lipschitz function
$\phi:[-r_*,r_*]\to\mathbb R$, with $\phi(0)=0$, such that,
within the corresponding coordinate rectangle,
$\operatorname{int}P$ is represented by
$\{(x_1,x_2):x_2>\phi(x_1)\}$.

In these coordinates, for any $h\le \min\{r_*,\delta/2\}$ the truncated cone
\[
C^+(h):=\{(x_1,x_2): 0<x_2<h,\ |x_1|<x_2/K_*\}
\]
is contained in $P\cap B_2(x,\delta)$, and the reflected cone
$
C^-(h):=\{(x_1,x_2): -h<x_2<0,\ |x_1|<-x_2/K_*\}
$
is contained in $(\mathbb R^2\setminus P)\cap B_2(x,\delta)$.
Moreover $|C^+(h)|=|C^-(h)|=h^2/K_*$.

If $\chi_Q\equiv 0$ on $B_2(x,\delta)$ then $C^+(h)\subset P\setminus Q\subset P\Delta Q$; if $\chi_Q\equiv 1$ on $B_2(x,\delta)$ then $C^-(h)\subset Q\setminus P\subset P\Delta Q$.
In both cases, $|P\Delta Q|\ge h^2/K_*$ for $h=\min\{r_*,\delta/2\}$, which implies
$\delta\le 2\sqrt{K_*}\,|P\Delta Q|^{1/2}$ whenever $\delta\le 2r_*$.
If $\delta>2r_*$, the same estimate with $h=r_*$ gives $|P\Delta Q|\ge r_*^2/K_*$ and thus
$\delta\le \diam(\Omega)\le \diam(\Omega)\sqrt{K_*}/r_*\cdot |P\Delta Q|^{1/2}$.
Choosing $C_H = |\kappa-1|^{-1/2}\max \{2\sqrt{K_*}, \diam(\Omega)\sqrt{K_*}/r_*\}$ yields this lemma. 
\end{proof}

\begin{lemma}[Vertex ordering stability]\label{lem:vertex_stability_poly}
There exists a constant $C_{\mathrm{vs}}>0$, depending only on
$(\Omega,d_0,d_1,\beta_0)$ and on $n_{\mathrm v}$, such that for any
$P^0,P^1\in\mathcal A_{1/2}$ the vertices of $P^0$ and $P^1$ can be labeled
$v^0=(v^0_i)_{i=0}^{n_{\mathrm v}-1}$ and $v^1=(v^1_i)_{i=0}^{n_{\mathrm v}-1}$
so that
\[
\|v^0-v^1\|_{\mathrm{poly}}
\le C_{\mathrm{vs}}\,d_H(\partial P^0,\partial P^1).
\]
The dependence on $n_{\mathrm v}$ enters only through the constants of
\cite[Prop.~3.3]{beretta2022global}, applied with $N_0=n_{\mathrm v}$.
\end{lemma}
\begin{proof}[Proof of Lemma~\ref{lem:vertex_stability_poly}]
Let $\widetilde{\mathcal{A}}_{1/2}$ denote the non-strict version of the relaxed
class, that is, the set of polygons $P=P(v)$ with $v\in\mathcal V^{n_{\mathrm v}}$
such that
\[
\dist(P,\partial\Omega)\ge \frac{d_0}{2},\qquad
|v_{i+1}-v_i|\ge \frac{d_1}{2},\qquad
\frac{\beta_0}{2}\le \beta_i(v)\le \pi-\frac{\beta_0}{2}
\qquad\text{for all } i.
\]
Clearly $\mathcal A_{1/2}\subset\widetilde{\mathcal A}_{1/2}$.
By \cite[Prop.~3.3]{beretta2022global}, applied to $\widetilde{\mathcal A}_{1/2}$
with the a priori data $(d_0/2,d_1/2,\beta_0/2,r_*,K_*)$ and $N_0=n_{\mathrm v}$,
there exist constants $\delta_0>0$ and $C_0>0$,
depending only on these data, with the following property: if
$P^0,P^1\in\widetilde{\mathcal A}_{1/2}$ satisfy
$d_H(\partial P^0,\partial P^1)\le \delta_0$, then their vertices can be labeled
so that
\[
\|v^0-v^1\|_{\mathrm{poly}}
\le C_0\, d_H(\partial P^0,\partial P^1).
\]
In particular, this applies to all $P^0,P^1\in\mathcal A_{1/2}$ with
$d_H(\partial P^0,\partial P^1)\le \delta_0$.

If instead $d_H(\partial P^0,\partial P^1)>\delta_0$, then for any labeling of
the vertices,
\[
\|v^0-v^1\|_{\mathrm{poly}}
\le \|v^0\|_{\mathrm{poly}}+\|v^1\|_{\mathrm{poly}}
\le 2\diam(\Omega)
< 2\diam(\Omega)\,\frac{d_H(\partial P^0,\partial P^1)}{\delta_0}\,,
\]
since all vertices lie in $\Omega$. Taking
$
C_{\mathrm{vs}}
:=
\max\left\{C_0,\frac{2\diam(\Omega)}{\delta_0}\right\}
$
proves the claim in both cases.
\end{proof}
\begin{remark}\label{rem:vs_role}
The constant $C_{\mathrm{vs}}$ inherits the dependence on $N_0=n_{\mathrm v}$
of the constants $\delta_0,C_0$ in \cite[Prop.~3.3]{beretta2022global}. Lemma~\ref{lem:vertex_stability_poly} is used only in the
proof of Lemma~\ref{lem:chart-M} to establish that $\varphi_{P^\star}$ is
a homeomorphism (a qualitative topological statement). The quantitative
bounds in Theorems~\ref{thm:FN-main} and~\ref{thm:cond-ddpm} are established
via independent arguments (Theorem~\ref{thm:Linf} for the forward-map
Lipschitz constant; Lemma~\ref{lem:logm} for the patch count) whose constants
depend only on $(\Omega,\kappa,d_0,d_1,\beta_0)$ and are explicitly
$n_{\mathrm v}$-independent.
\end{remark}

\begin{lemma}[Boundary perturbation under vertex displacement]
\label{lem:dist-boundary-perturb}
Let $v,w\in\mathcal V^{n_{\mathrm v}}$ have the same cyclic ordering and set
$
\delta:=\|v-w\|_{\poly}.
$
Then
$
d_H(\partial P(v),\partial P(w))\le \delta.
$
\end{lemma}

\begin{proof}
For any point $x\in\partial P(v)$, there are $i$ and $s\in[0,1]$ such that
$
x=(1-s)v_i+s v_{i+1}.
$
Set
$
y=(1-s)w_i+s w_{i+1}\in\partial P(w).
$
Then
\[
|x-y|\le (1-s)|v_i-w_i|+s|v_{i+1}-w_{i+1}|
\le \delta.
\]
Thus $\sup_{x\in\partial P(v)}\dist(x,\partial P(w))\le\delta$. The reverse
bound follows by swapping $v$ and $w$.
\end{proof}

\begin{lemma}[$L^1$-stability under vertex perturbations]\label{lem:symdiff_poly}
Fix $P^\star\in\mathcal A$, and let
$
\mathcal O_{P^\star}:=\Bpoly(v^\star,r_0)
$
be the chart ball from Lemma~\ref{lem:chart-M}. Let
$v,w\in \mathcal O_{P^\star}$ and set
$\delta\coloneqq \|v-w\|_{\mathrm{poly}}$.
Then
\[
\|\chi_{P(v)}-\chi_{P(w)}\|_{L^1(\Omega)}
= |P(v)\Delta P(w)| \le C\,\delta\,.
\]
Here $C=2n_{\mathrm v}(\pi+1) \diam(\Omega)>0$ is a constant.
\end{lemma}
\begin{proof}[Proof of Lemma~\ref{lem:symdiff_poly}]
Consider the segment $v(t)=(1-t)v+tw$ for $t\in[0,1]$ and write $P_t\coloneqq P(v(t))$. 

Since $\mathcal O_{P^\star}=\Bpoly(v^\star,r_0)$ is convex in
$\|\cdot\|_{\poly}$ and $P^\star\in\mathcal A$, Lemma~\ref{lem:patch-radius}
implies that $P(v(t))\in\mathcal A_{1/2}$ for every $t\in[0,1]$. Hence
$P_t$ is well-defined and contained in $\Omega$.
Let $x\in P(v)\Delta P(w)$, we first show there exists $t^\ast\in[0,1]$ such that $x\in\partial P_{t^\ast}$.
Define
$
g(t)\coloneqq \dist(x,P_t^c)-\dist(x,P_t),
$
where the map $t\mapsto g(t)$ is continuous.
If $g(0)=0$ or $g(1)=0$, $x \in \partial P(v)$ or $x \in \partial P(w)$; otherwise  $g(0)$ and $g(1)$ have opposite signs because $x$ belongs
to exactly one of $P_0$ and $P_1$. Hence there exists $t^\ast\in(0,1)$ such that $g(t^\ast)=0$, which implies
$\dist(x,P_{t^\ast})=\dist(x,P_{t^\ast}^c)=0$ and therefore $x\in\partial P_{t^\ast}$.

By Lemma~\ref{lem:dist-boundary-perturb} we have
$
d_H(\partial P_{t^\ast},\partial P(v))
\le \|v(t^\ast)-v\|_{\mathrm{poly}}\le \delta,
$
so $\dist(x,\partial P(v))\le\delta$. This shows
\[
P(v)\Delta P(w)\subset \{x\in\R^2:\dist(x,\partial P(v))\le\delta\}.
\]

Now $\partial P(v)$ is a union of $n_{\mathrm v}$ line segments. The $\delta$-neighborhood of a segment of
length $\ell$ has area at most $2\delta\,\ell+\pi\delta^2$ (rectangle plus two half-disks).
Summing over all edges yields
$
|P(v)\Delta P(w)|
\le 2\delta\,\Per(P(v)) + n_{\mathrm v}\pi\delta^2.
$
Finally, $\Per(P(v))\le n_{\mathrm v}\diam(\Omega)$ and $\delta\le 2r_0\le 2\diam(\Omega)$
on the chart, hence $|P(v)\Delta P(w)|\le C\,\delta$ with $C=2n_{\mathrm v}(\pi+1)\diam(\Omega)$.
\end{proof}

We are now ready to prove Lemma~\ref{lem:chart-M}. 

\begin{proof}[Proof of Lemma~\ref{lem:chart-M}]
\textbf{(i): $P(v)\in\cA_{1/2}$ for all $v\in\cO_{P^\star}$.} If
$\|v-v^\star\|_{\poly}<r_0$, then by Lemma~\ref{lem:patch-radius} applied
at $v^\star\in\cA$, we have $P(v)\in\cA_{1/2}$. Hence $\gamma_{P(v)}\in K_{1/2}$
and $U_{P^\star}\subset K_{1/2}$.

Injectivity of $\varphi_{P^\star}$. We first record a uniform separation of
the reference vertices. Set
$
c_{\rm sep}:=\frac{d_1}{2}\sin\beta_0.
$
Then
$
\min_{j\ne k}|v_j^\star-v_k^\star|\ge c_{\rm sep}.
$
Indeed, adjacent vertices are separated by at least $d_1\ge c_{\rm sep}$. If
$v_k^\star$ is not adjacent to $v_j^\star$, let $L_j$ be the line through
$v_{j-1}^\star$ and $v_{j+1}^\star$. By convexity, every nonadjacent vertex
lies in the closed half-plane bounded by $L_j$ that does not contain
$v_j^\star$. Hence
\[
|v_k^\star-v_j^\star|\ge \operatorname{dist}(v_j^\star,L_j).
\]
Writing
\[
a=|v_{j-1}^\star-v_j^\star|,
\qquad
b=|v_{j+1}^\star-v_j^\star|,
\]
we get
\[
\operatorname{dist}(v_j^\star,L_j)
=
\frac{ab\sin\beta_j^\star}{|v_{j-1}^\star-v_{j+1}^\star|}
\ge
\frac{ab\sin\beta_0}{a+b}
\ge
\frac{d_1}{2}\sin\beta_0.
\]
Since $\sin\beta_0\ge 2\beta_0/\pi$ for $\beta_0\in(0,\pi/2)$, we have
$
c_{\rm sep}\ge \frac{d_1\beta_0}{\pi}.
$
On the other hand,
\[
2r_0\le \frac{d_1\beta_0^2}{64}
<\frac{d_1\beta_0}{\pi}
\le c_{\rm sep}.
\]
Thus $2r_0<c_{\rm sep}$.

Now suppose $v,w\in\mathcal O_{P^\star}$ and
$\gamma_{P(v)}=\gamma_{P(w)}$. Since $\kappa\ne1$, this implies
$\chi_{P(v)}=\chi_{P(w)}$ a.e. The two sets are closed strictly convex
polygons, hence they are the closures of their interiors; equality a.e. of the
indicators implies equality of the interiors and therefore $P(v)=P(w)$. Since
both polygons are strictly convex, their vertex sets coincide. Hence $w$ is
obtained from $v$
by a cyclic relabeling. Let $\pi$ be such a relabeling. If $\pi\ne\mathrm{id}$, choose $j$
such that $\pi(j)\ne j$. Then
\[
|w_j-v_j^\star|
=
|v_{\pi(j)}-v_j^\star|
\ge
|v_{\pi(j)}^\star-v_j^\star|
-
|v_{\pi(j)}-v_{\pi(j)}^\star|
>
c_{\rm sep}-r_0
>
r_0,
\]
contradicting $w\in\mathcal O_{P^\star}$. Thus $\pi=\mathrm{id}$ and
$w=v$. Hence $\varphi_{P^\star}$ is well-defined and injective.

\textbf{(iii): $\varphi_{P^\star}$ is a homeomorphism.} (continuity of
$\varphi_{P^\star}^{-1}$) By Lemma~\ref{lem:symdiff_poly},
$v\mapsto\gamma_{P(v)}$ is locally Lipschitz from $(\cO_{P^\star},\|\cdot\|_{\poly})$
into $L^1(\Omega)$, hence continuous.

(continuity of $\varphi_{P^\star}$) For any sequence $\gamma_{P_k}\to\gamma_{P_0}$
in $U_{P^\star}$, $\|\gamma_{P_k}-\gamma_{P_0}\|_{L^1}\to 0$ gives
$|P_k\Delta P_0|\to 0$. By Lemma~\ref{lem:L1_to_H_poly},
$d_H(\partial P_k,\partial P_0)\to 0$, and by
Lemma~\ref{lem:vertex_stability_poly} the vertices can be ordered so that
$\|v^k-v^0\|_{\poly}\to 0$. By the injectivity/separation argument above, the
ordering supplied by Lemma~\ref{lem:vertex_stability_poly} must agree, for all
sufficiently large $k$, with the chart labeling determined by the disjoint
balls $B_2(v_i^\star,r_0)$. Hence the convergence is convergence of the actual
chart coordinates:
\[
\varphi_{P^\star}(\gamma_{P_k})
\to
\varphi_{P^\star}(\gamma_{P_0})
\quad\text{in }\|\cdot\|_{\poly}.
\]

\textbf{(ii): $U_{P^\star}$ is open.} Fix $\gamma_{P_0}\in U_{P^\star}$;
we show $\gamma_{P_0}$ is interior. Assume by contradiction that for every
$k\ge 1$, there exists $\gamma_{P_k}\in K_{1/2}\cap B_{L^1}(\gamma_{P_0},1/k)$
with $\gamma_{P_k}\notin U_{P^\star}$. As above, after the fixed chart labeling
is chosen for all sufficiently large $k$, Lemmas~\ref{lem:L1_to_H_poly} and
\ref{lem:vertex_stability_poly} give $\|v^k-v^0\|_{\poly}\to0$. Since
$\|v^0-v^\star\|_{\poly}<r_0$, the triangle inequality gives
$\|v^k-v^\star\|_{\poly}<r_0$ for large $k$, so $\gamma_{P_k}\in U_{P^\star}$,
a contradiction.
\end{proof}

\section{Proof of Lemma~\ref{lem:logm}}\label{sec:proof_finite_atlas}
\begin{proof}
\textbf{Construction (i).} Let $\{v^i\}_{i=1}^m$ be a maximal $r_0$-packing of
$\mathcal V_{\mathcal A}$ in the metric $\norm{\cdot}_{\poly}$: that is,
$v^i \in \mathcal V_{\mathcal A}$ for every $i$,
$\norm{v^i - v^j}_{\poly} \ge r_0$ for $i \ne j$, and the family is maximal
under set inclusion among collections with these two properties. Such a
packing exists by Zorn's lemma (or, since $\mathcal V_{\mathcal A}$ is bounded
in poly-norm, by a finite greedy construction). Maximality implies that for every $v \in \mathcal V_{\mathcal A}$ there exists
$i$ with $\norm{v - v^i}_{\poly} < r_0$, otherwise we could add $v$ to the
packing. Therefore the open balls $B_\poly(v^i, r_0)$ cover
$\mathcal V_{\mathcal A}$, proving (i).

\textbf{Volume comparison (ii).} The open balls $\{B_\poly(v^i, r_0/2)\}$ are pairwise disjoint, by the packing property (centers separated by $\ge r_0$). Hence, for any Lebesgue-measurable set $S \supset \bigcup_i B_\poly(v^i, r_0/2)$,
\[
m \cdot \mathrm{vol}\bigl(B_\poly(0, r_0/2)\bigr) \;\le\; \mathrm{vol}(S),
\]
where vol is Lebesgue measure on $\R^{2 n_{\mathrm{v}}}$. We use
$S = \mathcal V_{\mathcal A} + B_\poly(0, r_0/2)$, which
contains all packing balls because each $v^i \in \mathcal V_{\mathcal A}$.

We now estimate both sides.

\emph{Volume of $B_\poly(0, r_0/2)$.} The poly-norm ball is a Cartesian product:
\[
B_\poly(0, r) = \prod_{j=1}^{n_{\mathrm{v}}} B_2(0, r) \;\subset\; \R^{2 n_{\mathrm{v}}},
\]
where each factor $B_2(0, r) \subset \R^2$ is a Euclidean disk. Hence
\[
\mathrm{vol}\bigl(B_\poly(0, r_0/2)\bigr) = \bigl(\pi (r_0/2)^2\bigr)^{n_{\mathrm{v}}}.
\]

\emph{Volume of $\mathcal V_{\mathcal A} + B_\poly(0, r_0/2)$.} We use a
purely set-theoretic superset that need not consist of admissible polygons.
Since every $v \in \mathcal V_{\mathcal A}$ has all $n_{\mathrm{v}}$ vertices
inside the closed unit disk $\overline\Omega = \overline{B_2(0,1)}$, we have
$\mathcal V_{\mathcal A} \subset \prod_{j=1}^{n_{\mathrm{v}}} \overline{B_2(0, 1)}$.
Adding a poly-norm ball of radius $r_0/2$,
\[
\mathcal V_{\mathcal A} + B_\poly(0, r_0/2) \;\subset\; \prod_{j=1}^{n_{\mathrm{v}}} \overline{B_2(0, 1 + r_0/2)},
\]
and so
\[
\mathrm{vol}\bigl(\mathcal V_{\mathcal A} + B_\poly(0, r_0/2)\bigr) \;\le\; \bigl(\pi (1 + r_0/2)^2\bigr)^{n_{\mathrm{v}}}.
\]

Combining these bounds and substituting into the volume comparison:
\[
m \;\le\; \frac{\bigl(\pi (1 + r_0/2)^2\bigr)^{n_{\mathrm{v}}}}{\bigl(\pi (r_0/2)^2\bigr)^{n_{\mathrm{v}}}} \;=\; \left(\frac{1 + r_0/2}{r_0/2}\right)^{2 n_{\mathrm{v}}} \;=\; \left(1 + \frac{2}{r_0}\right)^{2 n_{\mathrm{v}}}.
\]
Taking logarithms gives the results.

\textbf{Patch property (iii).} For each center $v^i\in\mathcal V_{\mathcal A}$,
we have $P_i:=P(v^i)\in\mathcal A$. The set
$E_i:=\Bpoly(v^i,r_0)$ is a convex open ball in the norm
$\|\cdot\|_{\poly}$, and by definition $E_i=\mathcal O_{P_i}$.
Lemma~\ref{lem:patch-radius} applied at the admissible center $v^i$ gives
$
P(v)\in\mathcal A_{1/2}$, $\forall v\in E_i$. \qedhere
\end{proof}

\section{Proof of Lemma~\ref{lem:volumetric}}
\begin{proof}
By translation invariance, assume $v^\ast=0$. By monotonicity, $\mathcal{N}_\delta(E;\|\cdot\|_{\poly})\le\mathcal{N}_\delta(\Bpoly(0,R);\|\cdot\|_{\poly})$. If $\delta\ge R$, trivially $\mathcal{N}_\delta=1$. For $0<\delta<R$, let $\{x_j\}_{j=1}^n\subset\Bpoly(0,R)$ be a maximal
$\delta$-separated set in $\|\cdot\|_{\poly}$. By maximality, $\{x_j\}$ is
a $\delta$-net, so $\mathcal{N}_\delta\le n$. The poly-norm balls
$\Bpoly(x_j,\delta/2)$ are pairwise disjoint and contained in $\Bpoly(0,R+\delta/2)$.
The poly-norm ball decomposes as a Cartesian product of $\R^2$-disks, i.e.,
$
\mathrm{vol}(\Bpoly(0,r))
= \prod_{k=1}^{n_{\mathrm v}}\mathrm{vol}(B_2(0,r))
= (\pi r^2)^{n_{\mathrm v}},
$
thus we have $n\cdot(\pi(\delta/2)^2)^{n_{\mathrm v}}\le(\pi(R+\delta/2)^2)^{n_{\mathrm v}}$,
which yields
\[
n\le\Bigl(1+\frac{2R}{\delta}\Bigr)^{2n_{\mathrm v}}\le\Bigl(\frac{3R}{\delta}\Bigr)^{2n_{\mathrm v}}.
\]
\end{proof}

\section{Parameter choices and applicability of the diffusion bound}
\label{app:ddpm-parameters}

We record a sufficient parameter regime for applying the DDPM convergence analysis in \citet{liang2025low}. Fix $c_R>0$ and choose $c_0,c_1,c_{\epsilon_0}$ sufficiently large for the estimates in \citet{liang2025low}, with
\[
c_0\ge 2,\qquad
c_1>\max\{2c_0,c_0+2c_R+20\},\qquad
c_{\epsilon_0}>\frac{c_0}{2}+1.
\]
These parameters are fixed independently of $d,k,T$. For the EIT application, choose $ c_R\ge\max\{1,\log_2 r_{\mathcal K}\}. $ As in the proof of Theorem~\ref{thm:cond-ddpm}, these choices depend only on the geometric a priori data and are independent of $N,n_{\mathrm v},z$.

We first verify the covering scale required by the cited analysis. Let $\epsilon_0=T^{-c_{\epsilon_0}}$. Suppose $T\ge\sqrt d$ and $k\ge\max\{1,\log d\}$. Every interpolation parameter needed between successive forward times satisfies $0<a\le\bar\alpha_1=1-T^{-c_0}$. For sufficiently large $T$, we have $k\log T\ge1$, and hence
\[
\sqrt{\frac{1-a}{a}}\,
\min\left\{1,\sqrt{\frac{k\log T}{d}}\right\}
\ge T^{-c_0/2-1}.
\]
Our choice of $c_{\epsilon_0}$ therefore gives, uniformly over this interpolation range,
\[
\epsilon_0
=o\!\left(
\sqrt{\frac{1-a}{a}}\,
\min\left\{1,\sqrt{\frac{k\log T}{d}}\right\}
\right).
\]
This is the small-cover condition used in \citet[Appendix~G.2]{liang2025low}. The probability estimate in \citet[Appendix~G.1]{liang2025low} holds for $0<a<1$, and the posterior concentration argument in \citet[Appendix~G.2]{liang2025low} applies under the displayed condition. Thus these arguments extend the stated technical-lemma range $a\le1-1/T$ to $a\le1-T^{-c_0}$.

By increasing the fixed event constant in \citet[Appendix~G.1]{liang2025low}, the exceptional probability can be made smaller than any prescribed fixed inverse power of $T$. Since $R\le T^{c_R}$ and $1-a\ge T^{-c_0}$, the resulting polynomial amplifications can be absorbed while keeping the exceptional-event contributions $O(T^{-10})$. This changes fixed proof constants but not the final error rate.

We next control the Gaussian initialization error. Set
\[
h=\frac{c_1\log T}{T},
\qquad
m_T=\left\lceil
\frac{c_0\log T}{\log(1+h)}
\right\rceil.
\]
For sufficiently large $T$, we have $m_T<T$, and the schedule satisfies $\beta_t=h$ for $t\ge m_T+1$. Consequently,
\[
\bar\alpha_T
\le (1-h)^{T-m_T}
\le \exp\{-h(T-m_T)\}
\le 2T^{-(c_1-c_0)}.
\]
The last inequality follows from
\[
hm_T
\le c_0\log T
   +O\!\left(\frac{(\log T)^2}{T}\right).
\]
In particular, $\bar\alpha_T\le1/2$ for sufficiently large $T$.

Conditional on $X_0=x$, the terminal forward law is $\mathcal N(\sqrt{\bar\alpha_T}x, (1-\bar\alpha_T)I_d)$. Convexity of relative entropy in its first argument gives
\[
\operatorname{KL}
\bigl(p_{X_T}\,\|\,\mathcal N(0,I_d)\bigr)
\le
\frac12\left[
\bar\alpha_T R^2
+d\{-\bar\alpha_T-\log(1-\bar\alpha_T)\}
\right].
\]
Using $-a-\log(1-a)\le a^2$ for $0\le a\le1/2$ and Pinsker's inequality, we obtain
\begin{align*}
\operatorname{TV}
\bigl(p_{X_T},\mathcal N(0,I_d)\bigr)
&\le
\frac12\sqrt{\bar\alpha_T R^2+d\bar\alpha_T^2}
\\
&\le
\frac{1}{\sqrt2}
T^{c_R-(c_1-c_0)/2}
+T^{1-(c_1-c_0)}
\\
&\le T^{-10}
\end{align*}
for sufficiently large $T$. Here we used $R\le T^{c_R}$, $d\le T^2$, and $c_1-c_0>2c_R+20$. This verifies the initialization estimate used in \citet[Appendix~C.2, Eq.~(108)]{liang2025low}.

Combining these estimates, let $C_3$ denote the fixed posterior-moment constant in \citet[Corollary~1]{liang2025low}, with the preceding parameter choices. Choose $T_*$ sufficiently large that, for every $T\ge T_*$,
\[
T\ge\max\{2,\sqrt d\},\qquad
\frac{c_1\log T}{T}\le\frac14,\qquad
8C_3c_1k\log^2T\le T,
\]
and the preceding covering and initialization estimates hold. The last inequality is the step-count condition used in \cite[Appendix~G.3]{liang2025low}. The constant $C_3$ depends only on the fixed entropy and proof parameters, not on $d,k,T$. With these inputs, the argument in \cite[Appendix~C.2, Eqs.~(107)--(109)]{liang2025low} yields \Eqref{eq:TV_bound_uncond}. The constants and threshold above can be chosen uniformly over data laws satisfying common entropy and support-radius bounds.

For Theorem~\ref{thm:cond-ddpm}, take $d=d_N=N^2$ and $k=2n_{\mathrm v}$. The assumption $2n_{\mathrm v}\ge\log d_N$ verifies the dimensional condition. The entropy and support-radius bounds supplied by Theorem~\ref{thm:FN-main} are uniform over the conditional laws, and Assumption~\ref{ass:cond_score} supplies the corresponding score bound. Thus one common threshold $T_0$, independent of $z$, suffices. It can be chosen with $T_0\ge N$. Applying Lemma~\ref{lem:liang-ddpm} then yields \Eqref{eq:TV_bound_cond}.

\end{document}